\documentclass[11pt]{amsart}

\usepackage{enumerate,graphics}
\usepackage[latin1]{inputenc}

\newcommand{\spr}[1]{\langle #1\rangle}
\newcommand{\thespr}{\langle\cdot,\cdot\rangle}
\DeclareMathOperator{\ev}{ev}

\DeclareMathOperator{\End}{End}
\DeclareMathOperator{\Hom}{Hom}
\DeclareMathOperator{\GL}{GL}

\DeclareMathOperator{\SL}{SL}

\DeclareMathOperator{\tr}{tr}
\DeclareMathOperator{\Span}{Span}
\DeclareMathOperator{\Id}{Id}

\DeclareMathOperator{\image}{im}

\DeclareMathOperator{\Ad}{Ad} 
\usepackage{stmaryrd}
\newlength{\argheight}
\newcommand{\fwd}[1]{\settoheight{\argheight}{#1}%
	\overset{\shortrightarrow}{\smash{#1}\rule{0pt}{.75\argheight}}}
\newcommand{\bwd}[1]{\settoheight{\argheight}{#1}%
	\overset{\shortleftarrow}{\smash{#1}\rule{0pt}{.75\argheight}}}

\renewcommand{\Im}{\operatorname{Im}}
\renewcommand{\Re}{\operatorname{Re}}

\renewcommand{\ker}{\operatorname{ker}}
\newcommand{\im}{\operatorname{im}}
\DeclareMathOperator{\ord}{ord}

\newcommand{\R}{\mathbb{R}}
\newcommand{\C}{\mathbb{C}}
\newcommand{\N}{\mathbb{N}}

\renewcommand{\H}{\mathbb{H}}  

\newcommand{\CP}{\mathbb{CP}}
\newcommand{\HP}{\mathbb{HP}}

\newcommand{\ii}{\mathbf{i}}        
\newcommand{\jj}{\mathbf{j}}
\newcommand{\kk}{\mathbf{k}}
\newcommand{\theset}[2]{\{\,#1\mid#2\,\}}
\newcommand{\dvector}[1]{{\left(\begin{matrix}#1\end{matrix}\right)}}
\newcommand{\tvector}[1]{{\left(\begin{smallmatrix}#1\end{smallmatrix}\right)}}
\newcommand{\dvectork}[1]{{\left[\begin{matrix}#1\end{matrix}\right]}}
\newcommand{\tvectork}[1]{{\left[\begin{smallmatrix}#1\end{smallmatrix}\right]}}

\numberwithin{equation}{subsection}

\theoremstyle{plain}
\newtheorem{The}[subsection]{Theorem}
\newtheorem{Pro}[subsection]{Proposition}
\newtheorem{Lem}[subsection]{Lemma}
\newtheorem{Cor}[subsection]{Corollary}

\newtheorem{TheA}[subsection]{Theorem}
\newtheorem{ProA}[subsection]{Proposition}
\newtheorem{LemA}[subsection]{Lemma}
\newtheorem{CorA}[subsection]{Corollary}

\theoremstyle{definition}
\newtheorem*{Def}{Definition}

\theoremstyle{remark}

\newtheorem{Rem}[subsection]{Remark}
\newtheorem{Exa}[subsection]{Example}
\newtheorem{RemA}[subsection]{Remark}

\usepackage[
  breaklinks, 
  colorlinks=true, 
  linkcolor=black, 
  anchorcolor=black,
  citecolor=black,
  filecolor=black,
  menucolor=black,
  urlcolor=blue,
  bookmarks,
  bookmarksnumbered,
  hyperfootnotes=false,
  pdfpagelabels,
  pdfstartview=FitH, 
  pdfpagemode=UseOutlines, 
  pdfnewwindow=true,
  pagebackref=false, 
  bookmarksopen=false,
  bookmarks=true
]{hyperref}
\usepackage[figure]{hypcap}

\begin{document}

\title{Soliton Spheres}

\author{Christoph Bohle}
\author{G. Paul Peters}

\address{Institut f\"ur Mathematik\\ 
Technische Universität Berlin\\
Stra{\ss}e des 17.\ Juni 136\\
10623 Berlin\\
Germany}

\email{bohle@math.tu-berlin.de\\ peters@math.tu-berlin.de}


\date{May 14, 2009}

\begin{abstract}
  Soliton spheres are immersed 2--spheres in the conformal 4--sphere
  $S^4=\HP^1$ that allow rational, conformal parametrizations $f\colon
  \CP^1\rightarrow \HP^1$ obtained via twistor projection and dualization from
  rational curves in $\CP^{2n+1}$.  Soliton spheres can be characterized as
  the case of equality in the quaternionic Plücker estimate.  A special class
  of soliton spheres introduced by Taimanov are immersions into~$\R^3$ with
  rotationally symmetric Weierstrass potentials that are related to solitons
  of the mKdV--equation via the ZS--AKNS linear problem.  We show that
  Willmore spheres and Bryant spheres with smooth ends are further examples of
  soliton spheres. The possible values of the Willmore energy for soliton
  spheres in the 3--sphere are proven to be $W=4\pi d$ with
  $d\in\N\backslash\{0,2,3,5,7\}$. The same quantization was previously known
  individually for each of the three special classes of soliton spheres
  mentioned above.
\end{abstract}

\thanks{Both authors supported by DFG SPP 1154 ''Global Differential
  Geometry''.}

\maketitle


\section{Introduction}

The study of explicitly parametrized surfaces is one of the oldest subjects in
differential geometry. For more than two centuries the focus was essentially
on local parametrizations.  Towards the end of the 20th century, with the rise
of the global theory of minimal surfaces~\cite{Os86,Co84,HoMe85,cime} and the
developments~\cite{Ab87,PiSt89,Hi90,Bo91} initiated by Wente's
solution~\cite{We86} to the Hopf problem about constant mean curvature
surfaces in $\R^3$, the focus shifted to rather global considerations.
Nevertheless, our knowledge about explicit parametrizations of compact
surfaces is still surprisingly rudimentary compared to, for example, the
highly developed theory of complex algebraic curves.

An important source of explicitly parametrized spheres and tori is integrable
systems theory in combination with complex algebraic geometry.  The use of
integrable systems methods in the study of general conformal immersions
without special curvature properties has been pioneered by
Konopelchenko~\cite{Ko96} and Taimanov~\cite{Ta97,Ta98,Ta99}. Their approach
is based on the Weierstrass representation for conformal immersions into
$\R^3$ which provides an intimate connection between conformal immersions and
Dirac operators, see the survey~\cite{Ta06}.  These ideas were of great
influence in the development of quaternionic holomorphic geometry
\cite{PP98,FLPP01} and the starting point of our investigations.

The present paper is devoted to the study of \emph{soliton spheres}, a class
of immersed 2--spheres in the conformal 4--sphere $S^4$ that admit explicit
rational, conformal parametrizations $f\colon \CP^1\rightarrow\HP^1$ obtained
by some twistorial construction from rational curves in $\CP^{2n+1}$ and are
characterized by the fact that equality holds in the quaternionic Plücker
estimate \cite{FLPP01}
\begin{equation*}
  W \geq 4\pi\big[ (n+1)(n(1-g)-d) + |\ord H| \big],
\end{equation*}
a fundamental estimate for the Willmore energy $W$ of quaternionic holomorphic
line bundles.

The term soliton spheres was introduced by Iskander Taimanov~\cite{Ta99} for
conformal immersions $f\colon \CP^1 \rightarrow \R^3=\Im\H$ related to
multi--solitons of the modified Korteweg--de Vries (mKdV) equation.  With the
more general notion of soliton spheres discussed in the present paper, all
Willmore spheres and Bryant spheres with smooth ends are examples of soliton
spheres.

Section~\ref{sec:twistor_holomorphic_equality} gives an overview about the
basic concepts of quaternionic holomorphic geometry. In particular, we explain
the Möbius invariant representation of conformally immersed surfaces $f\colon
M\rightarrow S^4=\HP^1$ as quotients of quaternionic holomorphic sections.  In
Section~\ref{sec:equal-pluck-estim} we discuss the relation between equality
in the Plücker estimate and twistor holomorphic curves in $\HP^n$.  In
Section~\ref{sec:Dirac_Taimanov} we show how Taimanov's soliton
spheres~\cite{Ta99} can be treated using the quaternionic language. For this
we review the quaternionic version \cite{PP98} of the Weierstrass
representation for conformal immersions into Euclidean 4--space $\R^4=\H$.

In Section~\ref{sec:soliton_spheres} we define soliton spheres using the
Möbius invariant representation of conformal immersions.  We derive an
alternative characterization of soliton spheres in terms of the Weierstrass
representation which shows that Taimanov's soliton spheres are soliton spheres
in our sense.

In Sections~\ref{sec:examples-ii-bryant_spheres} and~\ref{sec:Willmore} we use
Darboux and B\"acklund transformations in order to show that Bryant spheres
with smooth ends and Willmore spheres are examples of soliton spheres. In
Section~\ref{sec:sos_s3} we prove that all soliton spheres in $\R^3$ with
Willmore $W\leq 32\pi$ are Willmore spheres or Bryant spheres with 
smooth ends and show that the possible Willmore energies of immersed soliton
spheres in $\R^3$ are $W=4\pi d$ for $d\in(\N\setminus\{0,2,3,5,7\})$. This 
generalizes previously known results about the quantization of the Willmore
energy for Willmore spheres in $\R^3$~\cite{Br88}, Taimanov's soliton
spheres~\cite{Ta99}, and Bryant spheres with smooth ends~\cite{SmoothEnds}.

\subsection*{Acknowledgment}

We would like to thank Ulrich Pinkall for directing us to the subject and for
supporting our project. We thank Iskander Taimanov, Franz Pedit and Ekkehard 
Tjaden for helpful discussions.

\section{Quaternionic holomorphic geometry}%
\label{sec:twistor_holomorphic_equality}
The principal idea of quaternionic holomorphic geometry
\cite{PP98,FLPP01,BFLPP02} is to approach conformal surface theory using the
concept of quaternionic holomorphic line bundles. From this perspective the
theory of conformal immersions appears as a ``deformation'' of the theory of
holomorphic curves in $\CP^n$.  This section gives a quick overview about the
basic notions of quaternionic holomorphic geometry.  Some special topics of
surface theory in the conformal 4--sphere $S^4=\HP^1$ are discussed in the
appendices.

\subsection{Quaternionic vector spaces}\label{sec:quat-vect-spac}
The quaternions are the 4--dimensional real associative algebra
$\H=\R\oplus \R\ii\oplus\R\jj\oplus\R\kk$ with multiplication rules
$\ii^2=\jj^2=\kk^2=\ii\jj\kk=-1$. A quaternion $\lambda = a+b\ii+c\jj+d\kk$ is
the sum of its real part $\Re(\lambda)=a$ and its imaginary part
$\Im(\lambda)=b\ii+c\jj+d\kk$. We write $\bar \lambda=a-b\ii-c\jj-d\kk$ for
the quaternionic conjugate. We identify $\R^4$ with $\H$ and $\R^3$ with the
space of imaginary quaternions $\Im\H=\R\ii\oplus\R\jj\oplus\R\kk$.

All quaternionic vector spaces in this paper are \emph{right vector
spaces}. The \emph{dual} $V^*$ of a quaternionic right vector space
$V$, naturally a left vector space, is made into a right vector space by
defining $\alpha \lambda :=(x\mapsto \bar\lambda \alpha(x))$ for
$\lambda \in\H$ and $\alpha\in V^*$.

\subsection{Quaternionic holomorphic line bundles}\label{sec:quat-holom-line}
Let $L$ be a quaternionic line bundle over a Riemann surface $M$, i.e., a
vector bundle whose fibers are modeled on $\H$. A complex structure on $L$,
that is, a field $J\in \Gamma(\End L)$ of quaternionic linear endomorphisms
with $J^2=-\Id$, makes $L$ into a so called \emph{complex quaternionic line
  bundle}. Denote by $*$ the action on 1--forms of the complex structure of
$TM$ and by $KL$ and $\bar K L$ the complex quaternionic line bundles whose
sections are $L$--valued 1--forms $\omega$ that satisfy $*\omega=J\omega$ or
$*\omega=-J\omega$, respectively.

A \emph{quaternionic holomorphic line bundle} over a Riemann surface $M$ is a
complex quaternionic line bundle $L$ together with a \emph{quaternionic
  holomorphic structure} $D$, a differential operator $D\colon \Gamma(L) \to
\Gamma(\bar KL)$ satisfying the Leibniz rule
\begin{equation*}
  D(\psi \lambda) = (D \psi)\lambda + (\psi d\lambda)''
\end{equation*}
for all $\psi \in \Gamma(L)$ and $\lambda\colon M \to \H$, where
$\omega'':=\frac12(\omega + J{*}\omega)$.  For the kernel of $D$ one writes
$H^0(L)$ and its elements are called \emph{holomorphic sections} of $L$. The
Leibniz rule implies that a nowhere vanishing holomorphic section uniquely
determines a quaternionic holomorphic structure on a complex quaternionic line
bundle.

The complex structure of a complex quaternionic line bundle $L$ induces
a decomposition
\begin{equation*}
  L=\hat L\oplus \hat L\jj, \qquad \hat L=\theset{\psi\in L}{J\psi=\psi\ii}
\end{equation*}
of $L$ into isomorphic complex line bundles $\hat L$ and $\hat L\jj$ .
The quaternionic holomorphic structure 
\begin{equation*}
  D=\bar\partial +Q
\end{equation*}
decomposes into a $J$ commuting part $\bar\partial$, which induces isomorphic
complex holomorphic structures on $\hat L$ and $\hat L\jj$, and a $J$
anti--commuting part $Q$ called the \emph{Hopf field} of $D$ which is a
1--form with values in $\End(L)$ that satisfies $*Q=-JQ=QJ$. In other words, a
quaternionic holomorphic line bundle is the direct sum of a complex
holomorphic line bundle $\hat L$ with itself plus a Hopf field.

In case $M$ is compact one defines the degree of a complex quaternionic line
bundle $L$ over $M$ as the degree of the underlying complex line bundle~$\hat
L$:
\begin{equation*}
  \deg(L):=\deg(\hat L).
\end{equation*}

\subsection{Holomorphic curves in $\HP^n$}\label{sec:holom-curv-hpn}
The n--dimensional quaternionic projective space $\HP^n$ is the space
of 1--dimensional subspaces $[x]=x\H$ of $\H^{n+1}$. Its tangent space
at $[x]\in \HP^n$ is $T_{[x]}\HP^1=\Hom([x],\H^{n+1}/[x])$.

A smooth map $M\to\HP^n$ is the same as a smooth line subbundle $L$ of the
trivial bundle $\H^{n+1}$ over $M$. In the following we will neither
distinguish between $L$ and the corresponding map $M\to\HP^n$ nor between the
trivial vector bundle and the corresponding vector space. For example, a
smooth map into $\HP^n$ is usually denoted by $L\subset \H^{n+1}$. The
\emph{derivative} of the map $L$ is then given by
\begin{equation*}
  \delta = \pi \nabla_{|L},
\end{equation*} where $\pi\colon \H^{n+1}\to
\H^{n+1}/L$ is the canonical projection and $\nabla$ denotes the trivial
connection on $\H^{n+1}$. A line subbundle $L\subset \H^{n+1}$ of the trivial
$\H^{n+1}$--bundle over a Riemann surface $M$ is called a
\emph{holomorphic curve} if $L$ admits a complex structure $J\in
\Gamma(\End(L))$, $J^2=-\Id$ such that
\begin{equation*}
    {*}\delta = \delta J.
\end{equation*}
A holomorphic curve $\HP^n$ is \emph{immersed} if $\delta$ is nowhere
vanishing. Otherwise, a non--constant holomorphic curve is called
\emph{branched}, cf.\ Appendix~\ref{subsec:branch_points_holomorphic_curves}.

\subsection{The quaternionic projective line $\HP^1$ as a model of the
  conformal 4--sphere $S^4$}\label{sec:ident-hp1-with}
The \emph{affine chart} 
\begin{equation*}
  \sigma\colon \HP^1\setminus\{\tvectork{1\\0}\}\to\H,
  \qquad
  \tvectork{\lambda\\1}\mapsto \lambda
\end{equation*} 
identifies the quaternionic projective
line $\HP^1$ with the conformal 4--sphere, i.e., the conformal
compactification $S^4=\H\cup \{\infty\}$ of $\R^4=\H$.  Under this
identification, the group of projective transformations of $\HP^1$ corresponds
to the orientation preserving Möbius transformations of $S^4$.

An immersed curve $L\subset \H^2$ in $\HP^1$ is holomorphic if and only if the
corresponding immersion $M\to S^4$ is conformal: let $L\subset \H^2$ be given
by
\begin{equation*}
  L=\psi\H\subset \H^2 \qquad\textrm{ with } \qquad
  \psi=\dvector{f\\1}\in\Gamma(L) 
\end{equation*}
where $f:=\sigma\circ L\colon M\to\H$ is $L$ seen in the affine
chart~$\sigma$.  By definition $L$ is a holomorphic curve if and only if
$*\delta =\delta J$ for some $J\in \Gamma(L)$, $J^2=-\Id$. But this is
equivalent to
\begin{equation*}
  *df=-dfR
\end{equation*}
with smooth $R\colon M\to S^2\subset \Im\H$ which is then related to $J$ via
$J\psi =-\psi R$.  In case $L$ is immersed, its holomorphicity is thus
equivalent to the conformality of the corresponding immersion $f\colon
M\rightarrow \H$, because an immersion $f$ is conformal if and only if
there is $R\colon M\to\H$ such that $*df=-dfR$, see \cite[Section
2.2]{BFLPP02}.

\subsection{Holomorphic curves and linear systems}%
\label{sec:holom-curv-and-linear-systems}
The so called \emph{Kodaira correspondence} is a fundamental relation between
holomorphic line bundles and holomorphic curves: let $L\subset \H^{n+1}$ be a
holomorphic curve in $\HP^n$.  To avoid technicalities suppose that $L$ is
full, i.e., not contained in a linear subspace. The complex structure of $L$
induces a complex structure on the dual bundle
\[L^{-1}\cong(\H^{n+1})^*/L^\perp\] 
which we again denote by $J$. Let $\pi\colon
(\H^{n+1})^*\to (\H^{n+1})^*/L^\perp$ be the canonical projection. Then
$L^{-1}$ carries a unique holomorphic structure $D$ such that
\begin{equation*}
  D(\pi\psi)  =\frac12(\pi\nabla\psi +J{*}\pi\nabla\psi)
\end{equation*}
for all $\psi \in \Gamma((\H^{n+1})^*)$, cf.~\cite[Theorem~ 2.3]{FLPP01}.  The
isomorphism type of the holomorphic line bundle $L^{-1}$ is a projective
invariant of the holomorphic curve $L\subset \H^{n+1}$ in $\HP^n$, which we
call the \emph{canonical holomorphic line bundle} of the curve $L$. For an
immersed holomorphic curve $L\subset \H^2$ in $\HP^1$ we therefore refer to
$L^{-1}$ as one of the \emph{M\"obius invariant holomorphic line bundles} of
$L$, see also Section~\ref{sec:sos}.

Like in the complex case, the \emph{degree} of a holomorphic curve $L\subset
\H^{n+1}$ in $\HP^{n}$ is defined as the degree of the corresponding
quaternionic holomorphic line bundle $L^{-1}$. In other words, the degree of
$L \subset \H^{n+1}$ seen as a holomorphic curve is \emph{minus} the degree of
$L$ seen as a complex quaternionic line bundle.

The holomorphic structure on $L^{-1}$ is the unique holomorphic structure with
the property that all projections of constant sections of $(\H^{n+1})^*$ are
holomorphic.  The linear system of holomorphic sections of $L^{-1}$ obtained
by projection from constant sections of $(\H^{n+1})^*$ is called the
\emph{canonical linear system} of the curve $L$ and denote it by
$(\H^{n+1})^*\subset H^0(L^{-1})$. The canonical linear system of a
holomorphic curve is always \emph{base point free}, i.e., there are no
simultaneous zeros of all holomorphic sections in $(\H^{n+1})^*$.

Let conversely $H\subset H^0(\tilde L)$ be a base point free linear system of
a quaternionic holomorphic line bundle $\tilde L$.  For $p\in M$ denote
$L_p\subset H^*$ the 1--dimensional subspace perpendicular to the hyperplane
in $H$ of sections vanishing at $p$. Then $L$ is a holomorphic curve, the
quaternionic holomorphic line bundle $\tilde L$ is isomorphic to $L^{-1}$, and
$H$ corresponds to the canonical linear system of $L$,
cf.~\cite[Section~2.6]{FLPP01}.

In case of holomorphic curves in $\HP^1$, Kodaira correspondence can be seen
as a representation of conformal immersions as quotients of holomorphic
sections: let $L\subset \H^2$ be the holomorphic curve given by $L =
\tvector{f\\1}\H$, see Section~\ref{sec:ident-hp1-with}, then the
holomorphic sections $e_1^*$ and $e_2^*$ of $L^{-1}$ obtained by projecting 
the standard basis of $(\H^2)^*$ are related by
\[ e_1^* = e_2^* \bar f. \] 
Changing the basis of the canonical linear system
amounts to a quaternionic fractional linear transformation of $f$, i.e., an 
orientation preserving M\"obius transformation.

\subsection{Weierstrass gaps and flag, dual
  curve} \label{sec:weierstr-gaps-flag} Let $H\subset H^0(L)$ be an
$(n+1)$--dimensional linear system of holomorphic sections of a holomorphic
line bundle $L$. The \emph{Weierstrass gap sequence} of $H$ at a point $p\in
M$ is the sequence $0\leq n_0(p) < n_1(p) < ... < n_n(p)$ of possible
vanishing orders $\ord_p(\psi)$ at $p$ of sections $\psi \in H$.  The
\emph{Weierstrass points}, i.e., the points at which the Weierstrass sequence
differs from $0,1$, \ldots, $n$, are isolated, see
\cite[Section~4.1]{FLPP01}. The integer
\begin{equation*}
  \ord_p(H)=\sum_{k=0}^n (n_k(p) - k)
\end{equation*}
is the \emph{Weierstrass order} of $H$ at $p$ and $\ord(H)$ is the
\emph{Weierstrass divisor} of $H$.  If $M$ is compact the \emph{Weierstrass
  degree} $|\ord H|=\sum_{p\in M} \ord_p (H)$ is finite.

The members ${H_k}_{|p}=\theset{\psi\in H}{\ord_p(\psi)\geq n_{n-k}(p)}$ of
the \emph{Weierstrass flag} of $H$ form continuous subbundles of $H$ of rank
$k$ which are smooth away from the Weierstrass points \cite[Lemma
4.10]{FLPP01}. The line subbundle $H_0$ is by definition the \emph{dual curve}
$L^d$ of $H$. One can show that, away from the Weierstrass points of $H$, the
dual curve $L^d$ is a holomorphic curve \cite[Theorem 4.2]{FLPP01}.

\subsection{Willmore energy}\label{sec:willmore-energy}
For compact surfaces, in addition to the invariants like degree and branching
order of the osculating curves known from complex curve theory, quaternionic
holomorphic curves have a further invariant: the Willmore energy. The
\emph{Willmore energy} of the holomorphic line bundle $L$ is
\begin{equation*}
  W(L) = \frac12 \int_M \tr_{\R}(Q\land *Q),
\end{equation*}
where $\tr_{\R}$ denotes real trace.  The 2--form $\tr_{\R}(Q\land *Q)$ is
positive and $W(L)=0$ is equivalent to $Q\equiv0$. In other words, the
Willmore energy $W(L)$ of a quaternionic holomorphic line bundle $L$ measures
the deviation from the (double of the) underlying complex holomorphic line
bundle $\hat L$ (Section~\ref{sec:quat-holom-line}).

The Willmore energy of the Möbius invariant holomorphic line bundle
$L^{-1}$ of a holomorphic curve $L\subset \H^2$ satisfies \cite[Section
6.2]{BFLPP02}
\begin{align*}
  W(L^{-1}) &=\int_M(|\mathcal H|^2-G+K^\perp) d\sigma , 
\end{align*}
where $\mathcal H$ is the mean curvature, $G$ the Gaussian curvature, and
$K^\perp$ the curvature of the normal bundle of $L$ with respect to any
compatible metric on the conformal 4--sphere $S^4\cong\HP^1$.

\section{Equality in the Plücker estimate}\label{sec:equal-pluck-estim}

In this section we describe the relation between equality in the
Plücker estimate and twistor holomorphic curves.

\subsection{Plücker estimate}\label{sec:plucker-estimate}
The \emph{quaternionic Plücker estimate} \cite[Cor.~4.8]{FLPP01}
\begin{equation*}
  W(L) \geq 4\pi\big[ (n+1)(n(1-g)-d) + |\ord H| \big]
\end{equation*}
gives a lower bound for the Willmore energy $W(L)$ of a quaternionic
holomorphic line bundle $L$ of degree $d$ over a compact Riemann surface of
genus~$g$ with $(n+1)$--dimensional linear system $H\subset H^0(L)$.

\subsection{Twistor holomorphic curves}\label{sec:twist-holom-curv-1}
The \emph{twistor projection} is the map
\begin{equation*}
  \CP^{2n+1}\rightarrow \HP^n \qquad  v\C  \mapsto v\H,
\end{equation*}
where $v\C$ denotes the complex line spanned by $v\in \H^{n+1}\backslash
\{0\}$ in the complex vector space $(\H^{n+1},\ii)\cong \C^{2n+2}$ obtained by
restricting the scalar field of the quaternionic right vector space
$\H^{n+1}$ to $\C=\R\oplus\R\ii$.

\begin{Def}
  The \emph{twistor lift} of a holomorphic curve $L\subset \H^{n+1}$ in
  $\HP^n$ is the complex line subbundle $\hat L = \theset{\psi \in L}{J\psi =
    \psi {\ii}} \subset(\H^{n+1},\ii)$.  A holomorphic curve in $\HP^n$ is
  called \emph{twistor holomorphic} if it is the twistor projection of a
  complex holomorphic curve in $\CP^{2n+1}$.
\end{Def}

\begin{Lem}[\cite{Fr84}]
  A holomorphic curve $L\subset \H^{n+1}$ in $\HP^n$ is twistor holomorphic if
  and only if its twistor lift $\hat L$ is complex holomorphic.
\end{Lem}
\begin{proof}
  Let $L\subset \H^{n+1}$ be a holomorphic curve and $E\subset(\H^{n+1},\ii)$
  a complex holomorphic line subbundle that twistor projects
  to~$L$. Holomorphicity of $E$ means that every smooth section
  $\psi\in\Gamma(E)$ satisfies $*\nabla\psi\equiv\nabla\psi\ii \mod E$. But
  then $*\delta=\delta J$, cf.~Section~\ref{sec:holom-curv-hpn}, yields
  $\delta\psi\ii=*\delta\psi=\delta J\psi$ such that $J\psi=\psi\ii$ for all
  $\psi\in\Gamma(E)$. This implies that $E$ is the twistor lift $\hat L$ of
  $L$.
\end{proof}

\subsection{Equality in the Plücker estimate}\label{sec:plucker-formula}
The link between equality in the Plücker estimate,
Section~\ref{sec:plucker-estimate}, and twistor holomorphic curves is
established by the quaternionic Plücker formula \cite[Theorem 4.7]{FLPP01}:
let $L$ be a quaternionic holomorphic line bundle of degree $d$ over a compact
Riemann surface of genus $g$ and $H\subset H^0(L)$ an $(n+1)$--dimensional
linear system, then
\begin{equation*}
  W(L) - W((L^d)^{-1}) =  4\pi\big[ (n+1)(n(1-g)-d) + |\ord H| \big]
\end{equation*}
where $L^d\subset H$ denotes the dual curve of $H$, see
Section~\ref{sec:weierstr-gaps-flag}.  In particular, although the dual curve
$L^d$ is only defined away form the Weierstrass points of $H$, the Willmore
energy $W((L^d)^{-1})$ is finite.

Equality in the Pl\"ucker estimate is thus equivalent to $W((L^d)^{-1})=0$
which, by Lemma~\ref{lem:twistor_hol_A=0} below, is equivalent to
holomorphicity of the twistor lift of $L^d$. As we will see in
Lemma~\ref{cont_twist_lift_dual}, the twistor lift of $L^d$ then extends
continuously through the Weierstrass points. This yields the following theorem
(cf.~\cite[Section 4.4]{FLPP01}).

\begin{The}\label{T:equality_is_holomorphic_twistorlift_of_Ld}
  A linear system of a quaternionic holomorphic line bundle over a compact
  Riemann surface has equality in the Plücker estimate if and only if its dual
  curve is twistor holomorphic. The twistor lift then extends holomorphically
  through the Weierstrass points of the linear system.
\end{The}

\subsection{Example: degree formula}\label{sec:degree_formula}
Applying the quaternionic Plücker formula
to a 1--dimensional linear system yields the so called \emph{degree formula}.
The degree of a complex holomorphic line bundle over a compact Riemann surface
equals the degree of the vanishing divisor of an arbitrary holomorphic
section.  In the quaternionic case the degree formula additionally involves
the Willmore energy: let $L$ be a quaternionic holomorphic line bundle over a
compact Riemann surface and $\varphi$ a holomorphic section of $L$.  Denote
by $H$ the 1--dimensional linear system spanned by $\varphi$. The
quaternionic holomorphic line bundle $(L^d)^{-1}$ of the dual curve $L^d$ of $H$
(Section~\ref{sec:weierstr-gaps-flag}) is isomorphic to $L^{-1}$ restricted to
$M\setminus\{\text{zeros of $\varphi$}\}$ equipped with the holomorphic
structure $\nabla''=\frac12(\nabla+J{*}\nabla)$, where $\nabla$ denotes the
connection on $L$ defined by $\nabla \varphi=0$. The Plücker formula applied
to $H$ thus becomes
\begin{equation*}
  W(L)+4\pi\deg(L)=W(L^{-1},\nabla'')+4\pi|\ord \varphi |,
\end{equation*}
where $|\ord \varphi|$ is the zero divisor of $\varphi$.

\subsection{The canonical complex structure}\label{sec:canonical_complex}
The \emph{canonical complex structure} \cite{FLPP01} of an
$(n+1)$--dimensional linear system $H$ of holomorphic sections of a
quaternionic holomorphic line bundle $L$ without Weierstrass points is the
unique complex structure $S\in \Gamma(\End(H))$, $S^2=-1$ that respects the
Weierstrass flag (Section~\ref{sec:weierstr-gaps-flag}), induces the given
complex structure on $L\cong H/H_{n-1}$ and satisfies $H_{n-1}\subset
\ker(Q)$ and $\im(A)\subset H_0=L^d$, where $Q=\frac14(S\nabla S-*\nabla S)$ and
$A=\frac14(S\nabla S+*\nabla S)$ are the so called \emph{Hopf fields}
of $S$. It follows that $S$ restricted to $L^d$ is the complex
structure of the dual curve. Moreover, the restriction of $Q$ to
$L\cong H/H_{n-1}$ coincides with the Hopf field of the
holomorphic structure of $L$,
cf.~Section~\ref{sec:holom-curv-and-linear-systems}.

The \emph{canonical complex structure} $S_L$ of a full holomorphic curve
$L\subset \H^{n+1}$ in $\HP^n$ is defined away from Weierstrass points as the
adjoint $S_L:=S^*$ of the canonical complex structure $S$ of the canonical
linear system $(\H^{n+1})^* \subset H^0(L^{-1})$, see
Section~\ref{sec:holom-curv-and-linear-systems}. Its Hopf fields satisfy
$Q_L=-A^*$ and $A_L=-Q^*$. In case of a holomorphic curve in $\HP^1$, the
canonical complex structure is also called \emph{mean curvature sphere
  congruence}, see Section~\ref{sec:mean-curv-sphere}.

\begin{Lem}\label{lem:twistor_hol_A=0}
  Let $L\subset \H^{n+1}$ be a holomorphic curve in $\HP^n$ and denote by $A$
  the Hopf field of the canonical complex structure of $L$. The following
  properties are equivalent:
  \begin{itemize}
  \item[i)] $L$ is twistor holomorphic, 
  \item[ii)] $A$ restricted to $L$ vanishes identically,
  \item[iii)] $A$ vanishes identically,
  \item[iv)] the Willmore energy $W(L^{-1})$ of the holomorphic line bundle
    $L^{-1}$ vanishes.
  \end{itemize}
\end{Lem}

\begin{proof}
  Let $\psi\in \Gamma(\hat L)$ be a section of the twistor lift of $L$.
  Holomorphicity of $L$ implies that $*\nabla\psi=\nabla\psi \ii + \psi
  \alpha$ with $\alpha$ a quaternion valued 1--form.  Using $S\psi =\psi \ii$,
  $\nabla S=2{*}Q-2{*}A$, and $Q\psi=0$ we obtain
  \begin{align*} 
    2{*}A \psi&= *A\psi+SA\psi
    =\tfrac12(-(\nabla S)\psi+S({*}\nabla S)\psi)\\
    &=\tfrac12(-\nabla\psi\ii +S\nabla\psi +S{*}\nabla\psi\ii+{*}\nabla\psi)
    =\tfrac12(\psi\alpha+S\psi\alpha\ii).
  \end{align*} 
  Hence $L$ is twistor holomorphic if and only if $A_{|L}\equiv0$, because
  holomorphicity of the twistor lift $\hat L$ of $L$ is equivalent to $\alpha$
  being complex valued.

  In order to check that $A_{|L}\equiv0$ if and only if $A\equiv0$, we prove
  that $A_{|L_k}\equiv0$ implies $A_{|L_{k+1}}\equiv0$, where $L_k$ denotes
  the $k^{th}$--osculating curve, i.e., the rank $k+1$ subbundles of
  $\H^{n+1}$ dual to $H_{n-k-1}$ in the Weierstrass flag of the canonical
  linear system $(\H^{n+1})^* \subset H^0(L^{-1})$, see
  Section~\ref{sec:weierstr-gaps-flag}. Let $\psi\in \Gamma(L_k)$ and $A\psi
  =0$. Then $*A\wedge \nabla\psi = d{*}A \psi=d{*}Q\psi=0$, because
  $d{*}A=d{*}Q$ vanishes on $L_{n-1}$. Since $*A=-AS$ and $*\nabla\psi \equiv
  S\nabla\psi \mod L_{k}$, this implies $A_{|L_{k+1}}=0$.

  The last equivalence holds because $-A^*$ is the Hopf field of the canonical
  complex structure of the linear system $(\H^{n+1})^* \subset H^0(L^{-1})$
  and hence induces the Hopf field of the quaternionic holomorphic line bundle
  $L^{-1}$.
\end{proof}

\subsection{Twistor lifts extend continuously through Weierstrass
  points}\label{sec:tl_cont}
In order to complete the proof of
Theorem~\ref{T:equality_is_holomorphic_twistorlift_of_Ld} it remains to check
that the twistor lift of the dual curve of a linear system extends
continuously through the Weierstrass points.  The proof of this fact given in
\cite{FLPP01} rests on the false claim that the canonical complex structure
(Section~\ref{sec:canonical_complex}) of a holomorphic curve in $\HP^n$
extends continuously through the Weierstrass points. A counterexample is the
holomorphic curve $L=\psi\H$ in $\HP^1$ defined by $\psi =\tvector{1+\jj
  z\\z^2}$ whose canonical complex structure does not extend continuously into
$z=0$, because $\psi(0)$ and $\psi'(0)$ are linearly dependent over $\H$,
cf.~Lemma~\ref{L:smooth_S_equiv_E1_not_quaternionic}.

We show now how to modify the arguments given in \cite{FLPP01} in order to
prove that the twistor lift of the dual curve extends continuously into the
Weierstrass points.

\begin{Lem}\label{cont_twist_lift_dual}
  Let $H\subset H^0(L)$ be a linear system of a quaternionic holomorphic line
  bundle $L$. The twistor lift of the dual curve $L^d \subset H$ of $H$ then
  extends continuously through the Weierstrass points of $H$.
\end{Lem}

\begin{proof}
  Let $S$ be the canonical complex structure of $H$. 
  We need to show that the twistor lift $\widehat{L^d}=\theset{\psi\in
    L^d}{S\psi=\psi\ii}$ of $L^d$ extends continuously through the
    Weierstrass points of $H$.  

  Let $p$ be a Weierstrass point of $H\subset H^0(L)$. By Lemma~4.9
  in~\cite{FLPP01} there exists a basis $\psi_k$, $k=0,\ldots,n$
  of $H$ that realizes the Weierstrass gap sequence $n_k(p)$ of $H$ at
  $p$ and has the following properties:
  \begin{enumerate}[(i)]
  \item There exists an open neighborhood $V$ of $p$ and a smooth map $B\colon
    V\to M(n+1, \H)$ that is $\GL(n+1,\H)$ valued on $V_0:=V\setminus\{p\}$
    such that $\underline{\psi}B^{-1}$ is an adapted frame of the Weierstrass
    flag ${H_k}_{|V_0} \subset H$.
  \item Let $z\colon V\to\C$, $z(p)=0$ be a coordinate,
    $Z=\operatorname{diag}(1,z^{-1},\ldots,z^{-n})$, and
    $W=\operatorname{diag}(z^{n_0(p)},\ldots,z^{n_n(p)})$. Then there exists
    $B_0\in \GL(n+1,\R)$ such that
    \[ B=Z(B_0+ O(1))W,\]
    where $O(1)$ stands for a continuous map on $V$ that vanishes to
    first order at $p$.
  \item\label{i:fact_S} The canonical complex structure $S$ satisfies
    $S\underline{\psi}B^{-1}= \underline{\psi}B^{-1}{\ii}$ on $V_0$.
  \end{enumerate}

  Let $\tilde L,\tilde U\colon V\to\GL(n+1,\H)$ be the LU--decomposition of
  $Z^{-1}BW^{-1}=B_0+O(1)$ with diagonal entries of $\tilde U$ equal to $1$.
  The upper triangular matrix $U:=W^{-1}\tilde U W$ then converges to the
  identity matrix when $z\rightarrow 0$. With the lower triangular matrix $L:=
  Z\tilde L W$ one obtains the LU decomposition $B=LU$ of $B$ restricted to
  $V_0$. On $V_0$ the frame
  \[ \underline{\psi} B^{-1} L = \underline{\psi}U^{-1} \] is adapted to the
  Weierstrass flag and converges to $\underline\psi$ when $z\rightarrow
  0$. The section
  \[
    \varphi:=\underline{\psi} B^{-1} Le_{n+1}= \underline{\psi}U^{-1} e_{n+1}
  \] 
  thus defines a continuous section of the trivial bundle $H$ over $V$. It
  spans the dual curve $L^d$ on $V_0$ and $\varphi\H$ extends $L^d$
  continuously into $p$. Because the restriction of $S$ to $L^d$ is the
  complex structure of $L^d$ on $V_0$ one concludes from (\ref{i:fact_S})
  above:
  \[ 
  S\varphi=S\underline{\psi} B^{-1} Le_{n+1}=\underline{\psi} B^{-1} {\ii}
  Le_{n+1} =\varphi\mu
  \] 
  on $V_0$, where $\mu=(L^{-1}{\ii} L)_{(n+1,n+1)}\colon V_0\to\H$ denotes the
  lower right entry of the matrix $L^{-1}{\ii} L$. Let $\lambda:=(\tilde
  L)_{(n+1,n+1)}$. Then $\lambda(0)\in\R\setminus\{0\}$, because $B_0$ is an
  invertible real matrix, and
  \[
    \mu=(W^{-1}\tilde L^{-1}{\ii} \tilde LW)_{(n+1,n+1)}
    =z^{-n_n(p)}\lambda^{-1}{\ii} \lambda z^{n_n(p)}.
  \]
  Hence $S\varphi =\varphi ({\ii} + O(1))$ 
  and $\varphi(p)\C\subset(\H,{\ii})$ extends the twistor lift
  $\widehat{L^d}$ of the dual curve continuously into the Weierstrass
  point $p$.
\end{proof}

\subsection{Two--dimensional linear systems with equality in the Plücker
  estimate}\label{sec:two_D_equality}
We have seen that the
canonical complex structure of a linear system does not in general extend 
continuously through the Weierstrass
points. However, in the
special case of a base point free linear system with equality in the Plücker
estimate it does extend smoothly into the Weierstrass points.  We prove this
only for 2--dimensional linear systems (which is what we need in
Section~\ref{sec:relat-superm-curv}). The extension to higher
dimensional systems is straightforward, although slightly more involved.

\begin{Pro}\label{P:2-d_equality_bundle_smooth_S}
  Let $L$ be a quaternionic holomorphic line bundle over a compact Riemann
  surface $M$ with base point free, $2$--dimensional linear system $H\subset
  H^0(L)$ for which equality holds in the Plücker estimate.  Then the dual
  curve $L^d\subset H$ of $H$ has a globally defined holomorphic twistor lift
  and the mean curvature sphere congruence of $L^d$ extends smoothly through
  the branch points of $L^d$.

  Conversely, let $L^d\subset\H^2$ be a twistor holomorphic curve whose mean
  curvature sphere congruence extends smoothly into the branch points.  Then
  $\H^2$ induces a base point free, 2--dimensional linear system of
  holomorphic sections of the quaternionic holomorphic line bundle
  $L=\H^2/L^d$ with equality in the quaternionic Plücker estimate.
\end{Pro}

\begin{proof}
  By Theorem~\ref{T:equality_is_holomorphic_twistorlift_of_Ld}, the curve
  $L^d\subset H$ is twistor holomorphic and its twistor lift $\hat L^d$
  extends holomorphically through the Weierstrass points of~$H$. The tangent
  line congruence $\widehat{L_1^d}$ of $\widehat{L^d}$ is then also globally
  defined.  By Lemma~\ref{L:smooth_S_equiv_E1_not_quaternionic}, 
  the canonical complex structure $S$ of
  $H$ (which, for a 2--dimensional linear system, is the mean curvature sphere
  congruence $S$ of the dual curve) extends smoothly into a Weierstrass point
  $p\in M$ of $H$ if and only if
  $(\widehat{L^d_1}\oplus\widehat{L^d_1}\jj)_{|p}=H$.

  Because $H$ is base point free, the evaluation map $\ev\colon H\to L$,
  $\ev_p(\psi)=\psi(p)$ induces a quaternionic bundle isomorphism between
  $H/L^d$ and $L$. The subbundle $\hat L=\theset{\varphi\in
    L}{J\varphi=\varphi\ii}$ has a section $\varphi\in\Gamma(\hat L)$ that
  does not vanish at $p$. Let $\psi\in\Gamma(H)$ such that
  $\ev(\psi)=\varphi$. On the open dense set on which $S$ is defined we have
  $S\psi \equiv \psi \ii \mod L^d$, because $\ev\circ S=J\circ\ev$ by
  definition of $S$. This implies $\psi\in \Gamma(\widehat{L^d_1}+L^d)$.
  Because $\psi(p)\not\in L^d_{|p}$ we obtain that
  $\widehat{L^d_1}_{|p}\not\subset L^d_{|p}$, hence 
  $\widehat{L^d_1}_{|p}$ is not $\jj$--invariant and
  $(\widehat{L^d_1}\oplus\widehat{L^d_1}\jj)_{|p}=H$.
  
  Conversely, the extended mean curvature sphere congruence induces a
  complex structure on $(L^d)^\perp\subset (\H^2)^*$ which guarantees
  that $(L^d)^\perp$ is a holomorphic curve. Hence, $L=\H^2/L^d$ is,
  as explained in Section~\ref{sec:holom-curv-and-linear-systems}, a
  quaternionic holomorphic line bundle and the canonical linear 
  system $\H^2$ has equality in the Plücker estimate, by
  Theorem~\ref{T:equality_is_holomorphic_twistorlift_of_Ld}, because
  its dual curve $L^d$ is twistor holomorphic.
\end{proof}

\section{Example: Taimanov soliton spheres}\label{sec:Dirac_Taimanov}

The term soliton sphere was introduced by Iskander Taimanov~\cite{Ta99} for
immersions into $\R^3$ with rotationally symmetric Weierstrass potential
corresponding to mKdV--solitons.  In this section we treat Taimanov's soliton
spheres in the framework of quaternionic holomorphic geometry using Pedit and
Pinkall's Weierstrass representation~\cite{PP98} for conformal immersions into
$\R^4=\H$.

\subsection{Weierstrass representation for conformal immersions into
  $\R^4$}\label{sec:gener-weierstr-repr-R4}
Taimanov's approach~\cite{Ta99} to soliton spheres is based on the
generalization of the Weierstrass representation for minimal surfaces to
arbitrary conformal immersion into $\R^3$. This representation describes the
differential of the immersion as the ``square'' of a Dirac spinor. In contrast
to the Kodaira correspondence
(Section~\ref{sec:holom-curv-and-linear-systems}), the Weierstrass
representation of conformal immersions is not M\"obius invariant, but only
invariant under similarity transformations.

The quaternionic version of the Weierstrass representation for conformal
immersions into $\R^4=\H$ is based on the following observation: given a
quaternionic holomorphic line bundle $L$, the complex quaternionic line bundle
$KL^{-1}$ admits a unique \emph{paired holomorphic structure}
\cite[Theorem~4.2]{PP98} such that
\begin{equation*}
  d(\alpha ,\psi)=0
\end{equation*}
for all local holomorphic sections $\alpha \in H^0(KL^{-1}_{|U})$, $\psi\in
H^0(L_{|U})$, where $(\alpha,\psi)$ denotes the evaluation paring
$\alpha(\psi)$. This pairing in particular satisfies
\[ *(\alpha,\psi)=(J\alpha,\psi)=(\alpha,J\psi).\]

The \emph{Weierstrass representation} of a conformal immersion $f\colon
M\to\R^4=\H$ into Euclidean 4--space then reads as follows:
\begin{The}[{\cite[Theorem~4.3]{PP98}}]\label{th:weierstr4D}
  Let $f\colon M\to\R^4=\H$ be a conformal immersion. Then there exist,
  uniquely up to isomorphism, paired quaternionic holomorphic line bundles $L$
  and $KL^{-1}$ and nowhere vanishing holomorphic sections $\psi\in H^0(L)$,
  $\alpha\in H^0(KL^{-1})$ such that
  \begin{equation*}
    df=(\alpha,\psi).
  \end{equation*}  
\end{The}

We call $L$ and $KL^{-1}$ the \emph{Euclidean holomorphic line bundles} of the
conformal immersion $f\colon M\to\R^4=\H$. By \cite[Proposition~8]{BFLPP02} we
have
\begin{equation*}
  Q\psi=\psi \frac12\bar{\mathcal H} df
\end{equation*}
and the Willmore energies (Section~\ref{sec:willmore-energy}) of $L$ and
$KL^{-1}$ satisfy
\begin{equation*}
  W(L)=W(KL^{-1})=\int_M|\mathcal H|^2d\sigma,
\end{equation*}
where $\mathcal H$ denotes the mean curvature vector and $d\sigma$ the
area element of $f$.

\subsection{Weierstrass representation for conformal immersions into
  $\R^3$}\label{sec:gener-weierstr-repr-R3} 
The image of the differential $df$ of a conformal immersion lies in
$\Im(\H)=\R^3$ if and only if
\[ *df=Ndf=-dfN,\] where $*$ denotes the complex structure on $T^*M$ and
$N\colon M\to S^2\subset \Im\H$ is the Gauss map of $f$. This implies that
$df=(\alpha ,\psi)$ is $\Im(\H)$--valued if and only if $\alpha \mapsto \psi$
induces a holomorphic bundle isomorphism $KL^{-1}\cong L$. A quaternionic
holomorphic line bundle $L$ that is isomorphic to $KL^{-1}$ is called a
\emph{quaternionic spin bundle} \cite{PP98}.

\begin{The}[{\cite[Theorem~4.4]{PP98}}]\label{the:weisterass_3d}
  For a conformal immersion $f\colon M\to\Im\H$ there exists, uniquely up to
  isomorphism, a quaternionic spin bundle with a holomorphic section $\psi$
  satisfying
  \begin{equation*}
    df=(\psi,\psi).
  \end{equation*}
\end{The}

\begin{Exa}[Minimal Surfaces]
  Let $f\colon M \rightarrow \R^3=\Im\H$ be a conformal immersion.  Denote by
  $L$ the corresponding spin bundle. Then $f$ is minimal if and only if $L$
  has vanishing Willmore energy $W(L)=0$. The holomorphic section $\psi \in
  H^0(L)$ in its Weierstrass representation $df=(\psi,\psi)$ is then the sum
  $\psi=\psi_1+ \psi_2\jj$ of two complex holomorphic spinors $\psi_1$,
  $\psi_2\in\Gamma(\hat L)$, where $\hat L = \{ \psi\in L \mid
  J\psi=\psi\ii\}$.  The differential of the minimal immersion $f$ is given by
  \begin{equation*}
    df=(\psi,\psi)=2\Im(\psi_1\otimes
    \psi_2)+\jj(\psi_1\otimes\psi_1+\overline{\psi_2\otimes\psi_2}), 
  \end{equation*}
  where $\psi_i\otimes \psi_j$ stands for the complex valued (1,0)--form
  $(\psi_i\jj,\psi_j)$.  This is the well known spinor Weierstrass
  representation of minimal surfaces.
\end{Exa}

\subsection{Weierstrass representation in coordinates}
The underlying complex holomorphic line bundle $\hat L=\theset{\psi\in
  L}{J\psi=\psi\ii}$ is a complex holomorphic spin bundle over $M$, because
$\hat L\otimes \hat L\to K$, $\varphi \otimes \psi \mapsto (\varphi \jj,\psi)$
is a bundle isomorphism.  In particular, for $z\colon M\supset U\to\C$ a local
chart with $U$ simply connected, there exists a complex holomorphic section
$\varphi$ of $\hat L_{|U}$ such that $dz={\jj}(\varphi,\varphi)$.  The formula
for $Q\psi$ given in Section~\ref{sec:gener-weierstr-repr-R4} implies
\begin{equation*}
  Q\varphi=\varphi\kk q dz, \qquad q=\frac{H|df|}{2|dz|}.
\end{equation*}
The Willmore energy then takes the form $W(L)=2\ii\int q^2(z)dz\land d\bar z$.
Writing a quaternionic holomorphic section $\psi\in\H^0(L)$ with respect to
the complex frame $\varphi,\varphi\kk$ as $\psi=\varphi\mu_1+\varphi\kk\mu_2$
with $\mu_1$, $\mu_2\colon U\to\C^2$, the holomorphicity $D\psi=0$ of $\psi$
becomes equivalent to $\mu=\tvector{\mu_1 \\ \mu_2}$ solving the Dirac equation
\begin{equation*}
  \label{eq:dirac_equation}
  \mathcal D\mu =0, \qquad \mathcal D= \dvector{q&\partial \\ -\bar\partial &
    q}. 
\end{equation*}

\subsection{Taimanov soliton spheres}\label{sec:taim-solit-spher}
Let $L$ be a quaternionic spin bundle over $M=\CP^1$ and $z\colon
\CP^1\setminus \{\infty\}\to\C$ an affine chart.  Then $dz$ is a meromorphic
section of the canonical bundle $K$ with a second order pole at $\infty$ and
no zeroes. The section $\varphi$ above is thus a meromorphic section of $\hat
L$ without zeroes and with a first order pole at $\infty$. A function $q\colon
\C\rightarrow \R$ is then a coordinate representation of a smooth Hopf field
$Q$ of $L$ if and only if $|w|^{-2}q(w^{-1})$ is smooth at $w=0$.  Similarly,
$\mu_1$, $\mu_2\colon \C\rightarrow \C$ are the coordinate representation of a
globally smooth section of $L$ if and only if $ |\mu_1|^2+|\mu_2|^2=O(
|z|^{-2} )$ for $|z|\to\infty$.

Suppose that $q$ is rotationally symmetric, i.e., $q(z)=q( |z|)$.
Taimanov then proves \cite[Lemma~4]{Ta99} that there are $N+1$
integers $0\leq n_0<\ldots <n_N$ and a basis 
$\{\psi_j\}_{0\leq j\leq N}$ of $H^0(L)$ which is
``rotationally symmetric'' in the sense that
\begin{align*}
  \psi_j&=\varphi e^{\frac{\ii}2 y}(\nu_{j1}(x)+\kk
  \nu_{j2}(x))e^{\frac{\ii}2(2n_j +1)y}, \qquad
  z=e^{x+\ii y},\\
  \intertext{where $\nu_j\colon \R\to\C^2$ is a rapidly decaying solution of
    the ZS--AKNS linear problem} \mathcal L\nu_j&=-\frac12(2n_j+1)
  \nu_j,\qquad \mathcal L =\dvector{ -\partial_x&2U\\ 2U & \partial_x},\qquad
  U(x)=q(e^x)e^x.
\end{align*}
The eigenvalues of $\mathcal L$ are preserved under the mKdV hierarchy and the
trace formula for $\mathcal L$ implies
\begin{equation*}
  \frac1{4\pi}W(L)=2 \int_{-\infty}^\infty U^2(x)dx\geq \sum_{j=0}^{N}(2n_j+1).
\end{equation*}

Equality in this estimate holds if and only if $U(x)=q(e^x)e^x$ is a
reflectionless potential of $\mathcal L$, i.e., a multi--soliton of
the mKdV equation, and $\frac\ii2(2n_j+1)$ are all points in the spectrum of
$L$ with positive imaginary part, cf.~\cite{Ta99}. On the other hand, equality
in this estimate, is equivalent to equality in the Plücker estimate
(Section~\ref{sec:plucker-estimate}) for the full linear system $H^0(L)$ of
holomorphic sections of the spin bundle $L$, because
\begin{multline*}
  (N+1)(N(1-g)-\deg(L))+\ord (H^0(L)) \\
  = (N+1)^2+ \sum_{j=0}^{N}(\ord_0\psi_j-j)+\sum_{j=0}^{N}(\ord_\infty
  \psi_j-j)= \sum_{j=0}^{N}(2n_j+1), 
\end{multline*}
since $g=0$, $\deg(L)=g-1=-1$, $\ord_0(\psi_j)=\ord_\infty(\psi)=n_j$, and $0$
and $\infty$ are the only possible Weierstrass points.

\begin{The}[Taimanov~\cite{Ta99}]\label{T:Taimanov}
  Let $L$ be a holomorphic spin bundle $L$ over $\CP^1$ whose potential $q$ is
  rotationally symmetric with respect to some affine coordinate $z\colon
  \CP^1\setminus \{\infty\}\to\C$. Then equality in the Plücker estimate holds
  for $H^0(L)$ if and only if the Willmore energy satisfies
  \begin{align*}
    W(L)=4\pi \sum_{j=0}^N (2n_j+1),
  \end{align*}
  where $0\leq n_0<\ldots <n_N$ are the integers such that there exists a
  basis of $H^0(L)$ consisting of rotationally symmetric holomorphic sections
  $\{\psi_j\}_{0\leq j\leq N}$ with $\ord_0(\psi_j)=\ord_\infty(\psi_j)=n_j$.
\end{The}

We call a conformal immersion $f\colon \CP^1\to\R^3$ a \emph{Taimanov soliton
  sphere} if its Euclidean holomorphic line bundle has a rotationally
symmetric potential with equality in the Plücker estimate for the full linear
system.  A special example of Taimanov soliton spheres are \emph{Dirac
  spheres} \cite{Ri} which have the most symmetric reflectionless potentials
$q(z)=\frac{N+1}{1+|z|^2}$. They are soliton spheres with $n_j=j$,
$j=0$,...,$N$ such that $\dim H^0(L)=N+1$ and $W(L)=4\pi(N+1)^2$.

Taimanov~\cite{Ta99} gives explicit rational formulae for all $q$ and $\psi_j$
that may arise in Theorem~\ref{T:Taimanov}. For every $(N+1)$--tuple $0\leq
n_0<\ldots <n_N$ of integers there is an $\R^{N+1}$--parameter family of $q$'s
and corresponding $\psi_j$'s. Since we are only interested in immersed
spheres, we need to start with a base point free linear system such that we
have to assume that $n_0=0$. Because the integers that may be written as the
sum of 1 with other pairwise distinct odd integers are
$\N\setminus\{0,2,3,5,7\}$ we obtain the following corollary.

\begin{Cor} \label{cor:taimanov_quant}
  The possible Willmore energies of Taimanov soliton spheres are $W\in 4\pi
  (\N\setminus\{0,2,3,5,7\})$. 
\end{Cor}

\begin{Exa}[Catenoid Cousins]\label{exa:cc-taimanov}
  Taimanov soliton spheres for $N=1$, $n_0=0$, $n_1=\mu$,
  $\lambda_0=\frac{\mu+1}\mu$, and $\lambda_1=\frac{(\mu+1)(2\mu+1)}\mu$,
  $\mu\in\N\setminus\{0\}$ (cf.~\cite{Ta99} for the meaning of $\lambda_j$)
  are the catenoid cousins that have smooth ends, see Example~\ref{exa:cc} for
  images.
\end{Exa}

\begin{Exa}[Rotationally symmetric soliton spheres]
  Figures~\ref{fig:tsos1}--\ref{fig:tsos4} show rotationally symmetric,
  branched Taimanov soliton spheres, i.e., primitives of $(\psi_j,\psi_j)$,
  $j=0,\ldots,N$.  The first example in each figure is immersed, the other
  examples have branch points of order $2n_i$ on the axis of revolution.

  \begin{figure}\label{fig:tsos1}
  	\centering
    \includegraphics{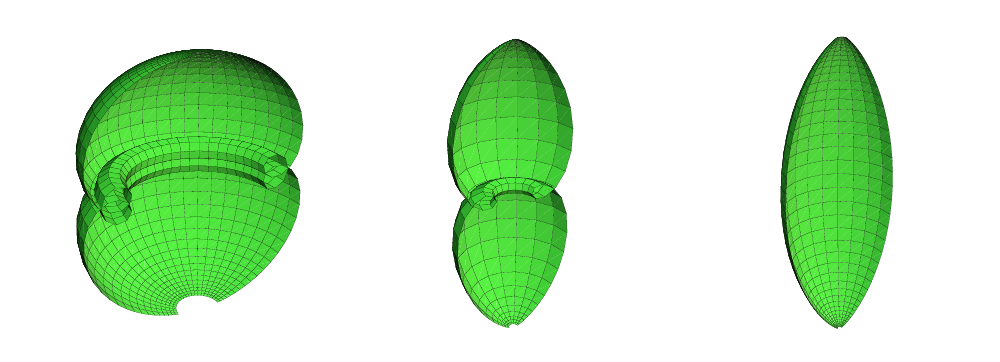}
    
    \caption{$N=2$, $n_0=0$, $n_1=1$, $n_2=2$, $\lambda_0=2$, $\lambda_1=6$, $\lambda_2=3$.}
  \end{figure}

  \begin{figure}\label{fig:tsos2}
  	\centering
    \includegraphics{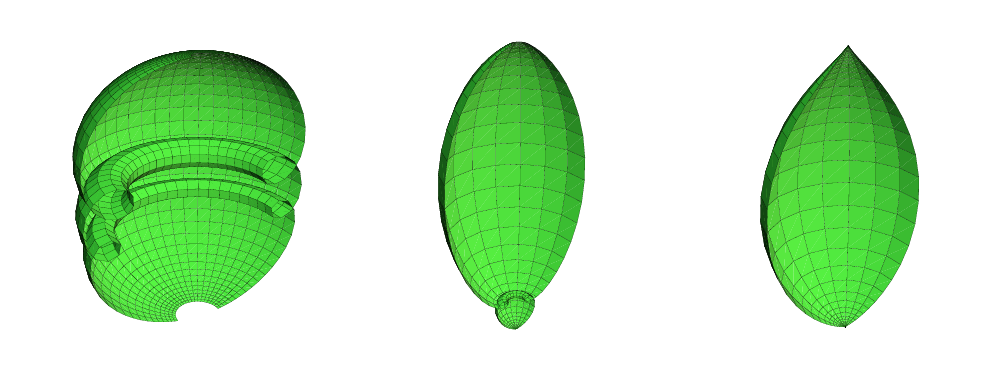}
	
	\caption{$N=2$, $n_0=0$, $n_1=1$, $n_2=2$, $\lambda_0=2$, $\lambda_1=6$, $\lambda_2=120$.}
  \end{figure}

  \begin{figure}\label{fig:tsos3}
  	\centering
    \includegraphics{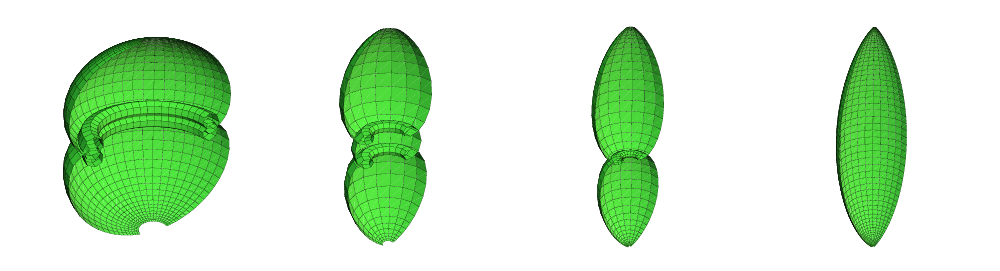}
	
	\caption{$N=3$, $n_0=0$, $n_1=1$, $n_2=2$, $n_3=3$,  
	$\lambda_0=4$, $\lambda_1=48$, $\lambda_2=120$, $\lambda_3=120$.}
  \end{figure}

  \begin{figure}\label{fig:tsos4}
  \centering
    \includegraphics{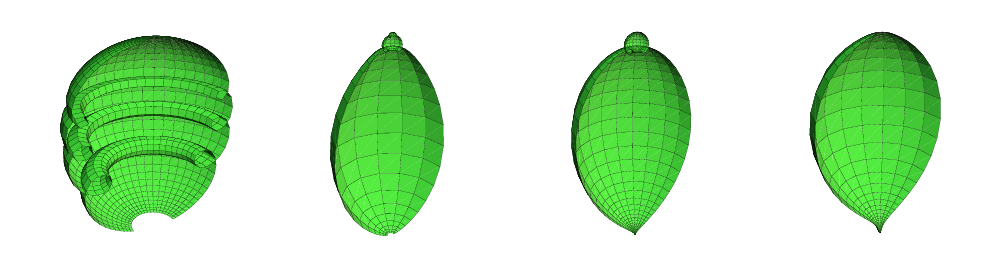}

	\caption{$N=3$, $n_0=0$, $n_1=1$, $n_2=2$, $n_3=3$,  
		$\lambda_0=6$, $\lambda_1=720$, $\lambda_2=120$, $\lambda_3=1$.}
  \end{figure}
\end{Exa}

\begin{Exa}[More Taimanov soliton spheres]\label{exa:dirac}
  Linear combinations of the $\psi_j$ lead to non rotationally symmetric
  Taimanov soliton spheres. Figure~\ref{fig:dirac} shows a deformation of the
  first into the third surface in Figure~\ref{fig:tsos1} through a family of
  surfaces $\int(\psi,\psi)$ where $\psi$ is a linear combination of $\psi_0$
  and $\psi_2$. The last surface in the figure is branched, all others are
  immersed.

	\begin{figure}\label{fig:dirac}
  		\centering
    	\includegraphics{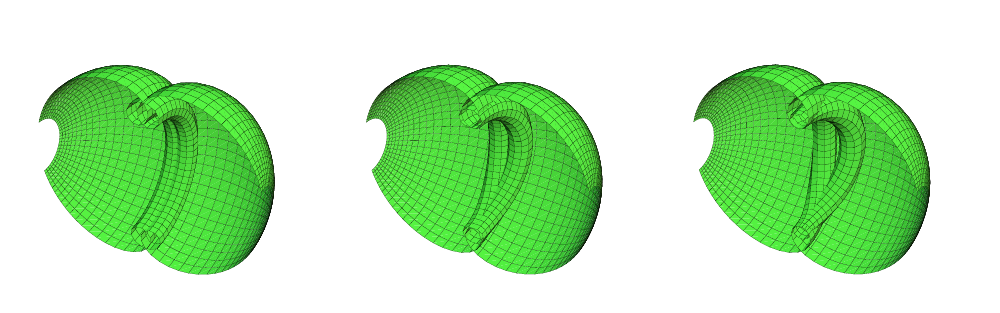}
	
 	  	\includegraphics{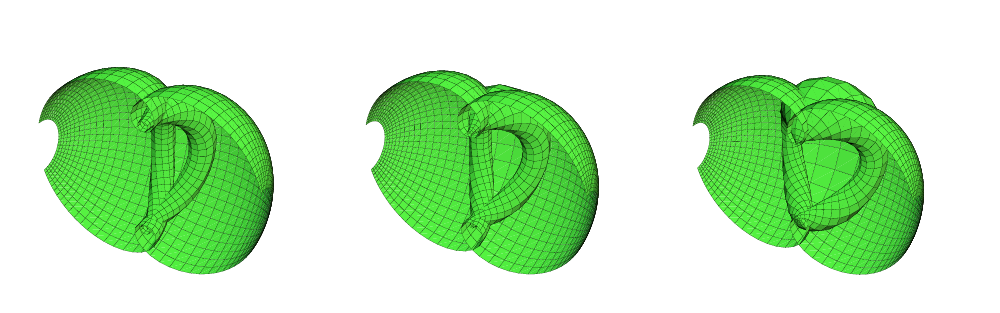}

	   	\includegraphics{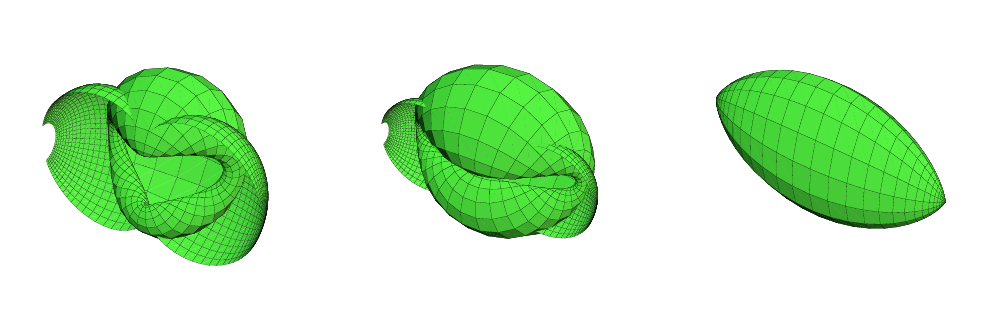}
		\caption{Deformation between the first and third surface in
                  Figure~\ref{fig:tsos1} through a family of Taimanov soliton
                  spheres with isomorphic Euclidean holomorphic line bundles.}
  	\end{figure}
	
\end{Exa}

\section{Soliton spheres}\label{sec:soliton_spheres}
We define soliton spheres in terms of the Möbius invariant holomorphic line
bundles of a conformal immersion $f\colon \CP^1\rightarrow \HP^1$.  We also
give a Euclidean characterization of soliton spheres based on the Weierstrass
representation. This in particular implies that Taimanov's soliton spheres are
examples of soliton spheres in our sense.

\subsection{Soliton spheres}\label{sec:sos}
Let $f\colon \CP^1\to \HP^1$ be a conformal immersion of a Riemann surface
into the conformal 4--sphere. Denote by $L\subset \H^2$ the corresponding
holomorphic curve and $L^\perp\subset (\H^2)^*$ its dual holomorphic curve.
In the following we mean by \emph{Möbius invariant holomorphic line bundles}
of $L$ the quaternionic line bundles $L^{-1}=(\H^2)^*/L^\perp$ and
$(L^\perp)^{-1}=\H^2/L$ with holomorphic structures as defined in
Section~\ref{sec:holom-curv-and-linear-systems}. The spaces of constant
sections of $(\H^2)^*$ and $\H^2$ project to the canonical linear systems
$(\H^2)^* \subset H^0(L^{-1})$ and $\H^2\subset H^0(\H^2/L)$. Soliton spheres
are immersions for which one of the canonical linear systems is contained in a
linear system with equality in the Plücker estimate
(Section~\ref{sec:plucker-estimate}):

\begin{Def}\label{def:moebius}
  A conformal immersion $f\colon \CP^1\to \HP^1$ into the conformal 4--sphere
  is called a \emph{soliton sphere} if one of the two Möbius invariant
  holomorphic line bundles admits a linear system with equality in the Plücker
  estimate that contains the canonical linear system.
\end{Def}

Two fundamental properties of soliton spheres are immediate con\-se\-quen\-ces of
the definition: firstly, \emph{the notion of soliton spheres is Möbius
  invariant}. Secondly, given a soliton sphere $L\subset \HP^1$, 
Theorem~\ref{T:equality_is_holomorphic_twistorlift_of_Ld} implies that either
$L$ or $L^\perp$ is the projection to $\HP^1$ of a holomorphic curve in
$\HP^n$ that is the dual curve (Section~\ref{sec:weierstr-gaps-flag}) of a
twistor holomorphic curve in $\HP^n$
(Section~\ref{sec:twist-holom-curv-1}). In other words, every soliton sphere
is obtained from a rational curve in $\CP^{2n+1}$ via twistor projection
$\CP^{2n+1}\rightarrow \HP^n$, dualization and projection to $\HP^1$
(possibly followed by a dualization in $\HP^1$), i.e., by
\begin{equation*}
  \CP^1\overset{\text{rational}}{-\!\!\!-\!\!\!-\!\!\!\longrightarrow}\CP^{2n+1}
  \overset{\text{twistor}}{-\!\!\!-\!\!\!\longrightarrow}  \HP^n
  \overset{\text{dualization}}{\dashleftarrow-\dashrightarrow}  \HP^n
  \overset{\text{projection}}{-\!\!\!-\!\!\!-\!\!\!-\!\!\!-\!\!\!\longrightarrow}\HP^1  
\end{equation*}
or the same sequence followed by a dualization in $\HP^1$, depending on
whether $H^0(L^{-1})$ or $H^0(\H^2/L)$ has a linear system with equality. The
dual curve of a holomorphic curve $L$ in $\HP^n$ is the solution of a system
of quaternionic linear equations whose coefficients are the $(n-1)^{th}$
derivatives of a generic section of $L$. Hence, \emph{every soliton sphere
  admits a rational, conformal parametrization}.

\begin{Def}
  The \emph{soliton number} of a soliton sphere that is not the round sphere
  is defined as the minimal number $n$ for which one of the canonical linear
  systems is contained in an $(n+1)$--dimensional linear system with equality
  in the Pl\"ucker estimate. The soliton number of the round sphere is defined
  to be $0$.
\end{Def}
The soliton number $n$ is the smallest number for which a soliton sphere can
be obtained via the above construction from a rational curve in
$\CP^{2n+1}$. The only $0$--soliton sphere is the round sphere. A $1$--soliton
sphere is either twistor holomorphic or the dual of a twistor holomorphic
curve, that is, all 1--soliton spheres are superconformal Willmore spheres
(cf.\ Section~8.2 of \cite{BFLPP02}) and vice versa. Examples of $2$--soliton
spheres are Bryant spheres with smooth ends
(Section~\ref{sec:examples-ii-bryant_spheres}) and non--superconformal Willmore
spheres (Section~\ref{sec:Willmore}).

\subsection{Characterization in terms of the Weierstrass
  representation}\label{sec:char-terms-weierst}
We give now a Euclidean characterization of soliton spheres in terms of the
Weierstrass representation.  As an application we show that Taimanov's soliton
spheres are also soliton spheres as in Section~\ref{sec:sos}.

Let $L\subset \H^2$ be an immersed holomorphic curve and fix a point
$\infty\in\HP^1$ that does not lie in the image of $L$. Without loss of
generality we may assume that $\infty =[e_1]$, where $e_1,e_2\in\H^2$ denotes
the standard basis. This basis of $\H^2$ defines the affine chart
(Section~\ref{sec:ident-hp1-with}) 
\begin{equation*}
  \sigma\colon \HP^1\setminus\{\infty\} \to \H, \qquad \dvectork{\lambda\\
    1}\mapsto \lambda. 
\end{equation*}
The affine chart is in fact defined uniquely up to similarity transformation
by the choice of $\infty=[e_1]\in \HP^1$.  The immersed holomorphic curve $L$
can then be written as
\begin{equation*}
  L=\psi\H \qquad \textrm{ with }  \qquad \psi=\dvector{f\\1},
\end{equation*}
where $f=\sigma\circ L\colon M\rightarrow \H$ is the affine representation of
$L$, i.e., the conformal immersion into $\H$ corresponding to $L$
via~$\sigma$.

Let $e_1^*,e_2^*\in\Gamma(L^{-1})$ be the projections to the first and second
coordinate of $\H^2$ seen as sections of $L^{-1}=(\H^2)^*/L^\perp$. The
canonical linear system (Section~\ref{sec:holom-curv-and-linear-systems}) of
$L$ is then spanned by the holomorphic sections $e_1^*$ and $e_2^*$ whose
quotient is $\bar f$, i.e.,
\begin{equation*}
  e_1^* = e_2^* \bar f.
\end{equation*}
Hence \emph{$f\colon \CP^1\to\H$ is a soliton sphere if and only if $f$ or
  $\bar f$ is the quotient of two quaternionic holomorphic sections that are
  contained in a linear system with equality in the Plücker estimate}, because
replacing $f$ by $\bar f$ is equivalent to replacing $L$ by $L^\perp$ and
interchanging the Möbius invariant and the Euclidean holomorphic line bundles.

The choice of $\infty\in \HP^1$ defines a unique flat connection
$\nabla$ on $L^{-1}$ which satisfies $\nabla e^*_2=0$ (where $e^*_2$ is
perpendicular to $\infty$).  The induced connection $\nabla$ on $L$
then satisfies $\nabla\psi =0$. Moreover $d^{\nabla}$ induces a
quaternionic holomorphic structure on $KL^{-1}$ (in the same way as $d$
defines the complex holomorphic structure on complex valued
$(1,0)$--forms) and $\nabla''=\frac12(\nabla +J{*}\nabla)$ induces a
quaternionic holomorphic structures on $L$ with respect to which
$\alpha=\nabla e_1^*=e_2^*d\bar f$ and $\psi$ are holomorphic sections
of $KL^{-1}$ and $L$, respectively.  The holomorphic structures thus
defined make $KL^{-1}$ and $L$ into paired holomorphic lines bundles
and
\begin{equation*}
  df=(\alpha,\psi)
\end{equation*}
is the Weierstrass representation of $f$. Theorem~\ref{th:weierstr4D} implies
that the line bundles $L$ and $KL^{-1}$ with holomorphic structures $d^\nabla$
and $\nabla''$ are the Euclidean holomorphic line bundles of~$f$.

\begin{The}\label{T:Euclidean_Definition}
  A conformal immersion $f\colon \CP^1\to\H=\HP^1\setminus\{\infty\}$ is a
  soliton sphere if and only if one of the two holomorphic sections of the 
  Euclidean holomorphic line bundles arising in
  the Weierstrass representation of $f$ is contained in a linear system with
  equality in the Plücker estimate.
\end{The}

\begin{Cor}
  Immersed Taimanov soliton spheres are soliton spheres.
\end{Cor}

\begin{proof}[Proof of Theorem~~\ref{T:Euclidean_Definition}]
  The holomorphic structure of the Möbius invariant holomorphic line bundle
  $L^{-1}=(\H^2)^*/L^\perp$ is given by $D=\nabla''$, where as above $\nabla$
  is the connection defined by $\nabla e_2^*=0$.  It is sufficient to show
  that $L^{-1}$ admits a linear system that contains $e_1^*$ and $e_2^*$ and
  has equality in the Pl\"ucker estimate if and only if $KL^{-1}$ admits a
  linear system with equality that contains $\alpha$. This follows from Lemma
  \ref{L:Equality_on_the_ladder} below, because $\nabla e_2^*=0$, $\nabla
  e_1^*=\alpha$, and $\CP^1$ is simply connected.
\end{proof}

\begin{Lem}\label{L:Equality_on_the_ladder}
  Let $L$ be a quaternionic holomorphic line bundle over a compact Riemann
  surface with a nowhere vanishing holomorphic section $\varphi_0$.  Let
  $\nabla$ be the flat connection on $L$ defined by $\nabla\varphi_0=0$ and
  $d^\nabla$ the induced quaternionic holomorphic structure on $KL$.

  Then $\nabla$ induces a linear map from $H^0(L)$ to $H^0(KL)$ which
  maps every $(n+1)$--dimensional linear system $H\subset H^0(L)$
  containing $\varphi_0$ to an $n$--dimensional linear system $\nabla
  H\subset H^0(KL)$. The linear system $H$ has equality in the Plücker
  estimate if and only if $\nabla H$ has equality.
\end{Lem}
  
\begin{proof}
  The Leibniz rule (Section~\ref{sec:quat-holom-line}) and $D=\nabla''$ imply
  that a section $\varphi\in\Gamma(L)$ is holomorphic if and only if
  $\nabla\varphi\in \Gamma(KL)$.  Because $\nabla$ is flat, it thus induces a
  quaternionic linear map $\varphi \in H^0(L) \mapsto \nabla \varphi \in
  H^0(KL)$ with kernel $\varphi_0\H$. The Weierstrass gap sequences $(n_k)$ of
  $H$ and $(\tilde n_k)$ of $\nabla H$ are related by $\tilde n_{k}= n_{k+1}-
  1$, hence $|\ord H| = |\ord \nabla H|$,
  cf.~Section~\ref{sec:weierstr-gaps-flag}.  Furthermore,
  $\deg(KL)=\deg(L)+2(g-1)$ and $W(KL)=W(L) + 4\pi \deg(L)$. The latter
  follows from the degree formula (Section~\ref{sec:degree_formula}) applied
  to $\varphi=\varphi_0$ using the fact that $W(KL)=W(L^{-1},\nabla'')$ (where
  $\nabla$ denotes the dual of $\nabla$ above), because $(KL,d^\nabla)$ and
  $(L^{-1},\nabla'')$ are paired bundles
  (Section~\ref{sec:gener-weierstr-repr-R4}). Plugging all these identities
  into the Plücker estimate (Section~\ref{sec:plucker-estimate}) shows that
  equality for $H$ is equivalent to equality for $\nabla H$.
\end{proof}

\section{Example: Bryant spheres with smooth
  ends}\label{sec:examples-ii-bryant_spheres}

We characterize 2--soliton spheres using the Darboux
transformation~\cite{BLPP} for conformal immersions into $S^4$.  Because the
hyperbolic Gauss map of a Bryant surface is a totally umbilic Darboux
transform \cite{JMN01} we conclude that Bryant spheres with smooth
ends~\cite{SmoothEnds} are soliton spheres.

\subsection{2--Soliton spheres and Darboux  transformations}\label{sec:2soli-dt}
Let $L\subset \H^2$ be an immersed holomorphic curve. The canonical projection
$\pi\colon \H^2\to\H^2/L$ onto the Möbius invariant holomorphic line bundle
$\H^2/L$, cf.~Section~\ref{sec:sos}, induces a 1--1--correspondence
\begin{equation*}
  \label{eq:prolongation}
  \theset{\psi^\sharp\in \Gamma(\H^2)}{\nabla\psi^\sharp \in
    \Omega^1(L)}\longrightarrow H^0(\H^2/L),\qquad \psi^\sharp\longmapsto
  \pi\psi^\sharp.  
\end{equation*}
The formula in Section~\ref{sec:holom-curv-and-linear-systems} for the
holomorphic structure of $\H^2/L$ shows that $\pi \psi^\sharp\in \Gamma(\H^2)$
is indeed holomorphic if $\nabla\psi^\sharp \in \Omega^1(L)$.  The
correspondence is bijective because $\delta=\pi\nabla_{|L}\colon L\to
K(\H^2/L)$ is a bundle isomorphism since $L$ is an immersion. The section
$\psi^\sharp\in\Gamma(\H^2)$ is called \emph{prolongation} of the holomorphic
section $\varphi = \pi\psi^\sharp$.

A \emph{Darboux transform} \cite{BLPP} of a conformal immersion $L\subset \H^2$
is a map $L^\sharp\subset \H^2$ defined away from the zeros of the prolongation
$\psi^\sharp$ of a holomorphic section $\varphi = \pi \psi^\sharp \in
H^0(\H^2/L)$ by the lines $L^\sharp=\psi^\sharp\H\subset \H^2$ spanned
by~$\psi^\sharp$ (in case of non simply connected surfaces one has to
allow for holomorphic sections with monodromy).  Darboux transforms of $L$
thus correspond to 1--dimensional linear systems of holomorphic sections of
$\H^2/L$. A Darboux transform is constant if and only if the 1--dimensional
linear system is contained in the 2--dimensional canonical linear system;
otherwise it is a branched conformal immersion.

\begin{The}\label{the:2--solitons_DT}
  A conformal immersion of $\CP^1$ into $\HP^1$ is a 2--soliton
  sphere if and only if it has a non--constant twistor holomorphic Darboux 
  transform. 
\end{The}

\begin{proof}
  The theorem is a direct consequence of the following proposition.
\end{proof}

\begin{Pro} \label{prop:2--solitons_DT} Let $L^\sharp\subset \H^2$ a
  non--constant Darboux transform of a conformally immersed sphere $L\subset
  \H^2$ in $\HP^1$. Then $L^\sharp$ is twistor holomorphic if and only if the
  corresponding 1--dimensional linear system of $\H^2/L$ together with its
  canonical linear system spans a 3--dimensional linear system with equality
  in the Pl\"ucker estimate.
\end{Pro}

\begin{proof}
  Let $\infty=e_1\H$ be a point that does not lie on the image of the immersed
  sphere $L\subset \H^2$ in $\HP^1$.  As in Section~\ref{sec:ident-hp1-with},
  denote by $f\colon\CP^1\to\H$ the representation of $L$ in the affine chart
  corresponding to a basis $e_1,e_2\in\H^2$, i.e., $L=\psi\H$ for
  $\psi=\tvector{f\\1}$.  The canonical linear system of $\H^2/L$ is then
  spanned by the sections $\varphi:=\pi e_1$ and $\pi e_2=-\varphi f$.  By the
  Leibniz rule in Section~\ref{sec:quat-holom-line}, a section $\varphi h$ of
  $\H^2/L$ is holomorphic if and only if $*dh=Ndh$ with $N\colon\CP^1\to\H$,
  $N^2=-1$ defined by $J\varphi=\varphi N$.  Hence $\varphi h$ is holomorphic
  if and only if there exists $g\colon\CP^1\to\H$ such that
  \[
    df g + dh =0.
  \]
  The prolongation $\psi^\sharp$ of $\varphi h$ is then given by
  \[
    \psi^\sharp = \dvector{f\\1} g + \dvector{1\\0}h.
  \]
  The Darboux transform $L^\sharp\subset \H^2$ corresponding to $\varphi h$ is
  defined away from the zeros of $\psi^\sharp$ as the line subbundle spanned by
  $\psi^\sharp$. Its affine part $f^\sharp$ is defined away from the zeros of $g$
  and satisfies
  \[
  f^\sharp = f + h g^{-1} \qquad \text{and} \qquad df^\sharp = h\, d(g^{-1}).
  \]
  This proves the following lemma.
  \begin{Lem}\label{lem:dt_affine}
    The map $f^\sharp = f + h g^{-1}$ is a Darboux transform of the conformal
    immersion $f$ if and only if $df g + dh =0$.
  \end{Lem}
  Unless $\varphi h$ is contained in the canonical linear system, away from the
  isolated zeros of $g$, $h$, and $dg$, the affine part $f^\sharp$ of $L^\sharp$
  is a conformal immersion.\renewcommand{\qedsymbol}{}
\end{proof}

  \begin{Lem}\label{lem:dt_twistor}
    Let $f^\sharp$ be a Darboux transform of $f$ given by $f^\sharp = f + h
    g^{-1}$ with $df g + dh =0$ and nowhere vanishing $g$, $h$, and $dg$. Then
    \begin{enumerate}[i)]
      \item $f^\sharp$ is twistor holomorphic if and only if $g$ is twistor
        holomorphic,
      \item $f^\sharp$ is totally umbilic if and only if $g$ is twistor
        holomorphic and $h^{-1}$ is Euclidean minimal, and
      \item $f^\sharp$ is planar if and only if both $g^{-1}$ and $h^{-1}$ are
        twistor holomorphic and Euclidean minimal.
    \end{enumerate}
  \end{Lem}

  \begin{proof}
    From $df^\sharp = h\, d(g^{-1})$ we obtain that the right normal vectors
    (cf.\ Appendix~\ref{sec:left-right-normal}) of $f^\sharp$ and $g^{-1}$
    coincide
    \begin{equation*}
       R_{f^\sharp}=R_{g^{-1}}.
    \end{equation*}
    On the other hand $d(g^{-1}) = h^{-1} df^\sharp$ implies
    $0=d(h^{-1})\wedge df^\sharp$ and thus
    \begin{equation*}
      N_{f^\sharp} = -R_{h^{-1}}.
    \end{equation*}

    The lemma now follows, because (as shown in
    Appendix~\ref{sec:mean-curv-sphere-affine}) a conformal immersion is twistor
    holomorphic if and only if its inversion is twistor holomorphic if and only
    if $dR''=0$; it is Euclidean minimal if and only if $dR'=0$; it is totally
    umbilic if and only if $dN''=dR''=0$; and it is planar if and only if its
    normal vectors $N$ and $R$ are both constant.
  \end{proof}

\begin{proof}[Proof of Proposition~\ref{prop:2--solitons_DT} continued]
  We have to show that $L^\sharp=\tvector{f^\sharp
    \\1}\H$ is twistor holomorphic if and only if the linear system $H$
  spanned by $\varphi$, $\varphi f$, and $\varphi h$ has equality in the
  Pl\"ucker estimate.  Applying Lemma~\ref{L:Equality_on_the_ladder} to the
  nowhere vanishing holomorphic section $\varphi$ shows that equality in the
  Pl\"ucker estimate for $H$ is equivalent to equality in the Pl\"ucker
  estimate for the 2--dimensional linear system $\nabla H$ of $K(\H^2/L)$
  spanned by $\varphi df$, $\varphi dh=-\varphi dfg$.
  Theorem~\ref{T:equality_is_holomorphic_twistorlift_of_Ld} now implies that
  equality for $\nabla H$ is equivalent to $g$ being twistor holomorphic,
  since $\tvector{g\\1}\H\subset\H^2$ is the dual curve of $\nabla H$. This
  proves the claim, because $g$ is twistor holomorphic if and only if
  $f^\sharp$ is twistor holomorphic (Lemma~\ref{lem:dt_twistor}).
\end{proof}

\begin{Rem}
 Proposition~\ref{prop:2--solitons_DT} holds verbatim for compact Riemann
 surfaces of higher genus if one allows for linear systems with monodromy.
\end{Rem}

\begin{Rem}
  In general, a Darboux transform of a conformal immersion $L\subset \H^2$ may
  not extend smoothly through the isolated zeros of the corresponding
  holomorphic section of $\H^2/L$. We show now that, \emph{in the situation of
    Proposition~\ref{prop:2--solitons_DT}, the Darboux transform $L^\sharp$
    extends smoothly through the zeros of the defining holomorphic section of
    $\H^2/L$ and has a globally smooth twistor lift (which is hence a rational
    curve in $\CP^3$).}

  Let $L\subset \H^2$ be a conformal immersion and $\varphi h$ a holomorphic
  section of $\H^2/L$ that, together with the canonical linear system, spans a
  3--dimensional linear system $H\subset H^0(\H^2/L)$ with equality in the
  Plücker estimate. Then $L^\sharp=\psi^\sharp\H$, $\psi^\sharp =
  \tvector{f\\1} g + \tvector{1\\0}h$ is defined and smooth away from the
  common zeros of $g$ and $h$. Moreover, $L^\sharp$ is a holomorphic curve
  with complex structure $J^\sharp \psi^\sharp =- \psi^\sharp R_g$ for $R_g$
  the right normal of $g$, because $\nabla \psi^\sharp = \psi \, dg$.
  
  By Theorem~\ref{T:equality_is_holomorphic_twistorlift_of_Ld}, the curve
  $\tvector{g\\1}\H$ has a globally defined holomorphic twistor lift that
  locally is of the form $\tvector{g_1+ \jj g_2\\ g_3 + \jj
    g_4}\C\subset(\H^2,\ii)$ with complex holomorphic functions
  $g_1,\ldots,g_4$. Let $p$ be a common zero of $g$ and $h$.  Without loss of
  generality we may assume that $g_3 +\jj g_4$ does not vanish at $p$, because
  $g$ has no ``poles''. Now $R_g = -(g_3 + \jj g_4) \ii (g_3 + \jj g_4)^{-1}$
  implies that the twistor lift of $L^\sharp$ is locally given as the complex
  line spanned by \[ \psi^\sharp (g_3 + \jj g_4) = \dvector{f\\1} (g_1 + \jj
  g_2) + \dvector{1\\0}h (g_3 + \jj g_4).
     \] If $n$ is the vanishing order of $g$ at $p$, then $h$ vanishes to
     order $n+1$ at $p$, because $dfg=-dh$. Since $g_3 + \jj g_4$ does not
     vanish, $n$ is the vanishing order of $g_1 + \jj g_2$. The twistor lift
     of $L^\sharp$ can be extended continuously through $p$, because the limit
     of $\psi^\sharp (g_3 + \jj g_4)z^{-n}$ at $p$ exists and is not zero,
     where $z$ is a local holomorphic chart centered at~$p$.  The claim now
     follows from Riemann's removable singularity theorem.
\end{Rem}

\subsection{Bryant spheres with smooth ends are 2--soliton spheres}
Bryant spheres with smooth ends \cite{SmoothEnds} are surfaces of mean
curvature one in hyperbolic space that compactify to immersed spheres by
adding points on the ideal boundary of hyperbolic space. In~\cite{JMN01} it is
shown that Bryant surfaces are characterized by the existence of a totally
umbilic Darboux transform which is then the hyperbolic Gauss map, see also
\cite[Theorem~9]{SmoothEnds}.  In order to apply
Theorem~\ref{the:2--solitons_DT} it remains to check that the holomorphic
section defining this Darboux transform extends smoothly through smooth Bryant
ends.

\begin{Rem} \label{rem:bryant_dt_in_s3} It seems worthwhile to note that the
  characterization of Bryant surfaces~\cite{JMN01,SmoothEnds} by the existence
  of a totally umbilic Darboux transform requires that both the surface and
  its Darboux transform take values in the same round 3--sphere in
  $S^4=\HP^1$. This holds automatically for the ``classical'' Darboux
  transform in the isothermic surface sense as used in
  \cite{JMN01,SmoothEnds}, but not for the Darboux transform of \cite{BLPP} as
  used in Section~\ref{sec:2soli-dt}.

\end{Rem}

An immersion $L\subset\H^2$ of $\CP^1$ into $\HP^1$ is a Bryant sphere with
smooth ends if and only if, up to Möbius transformation,
\begin{equation*}
  L=\psi\H,\qquad \psi=F\tvector{\kk\\1},
\end{equation*}
for a rational null immersion $F$ into $\SL(2,\C)$ such that all poles of $dF
F^{-1}$ have order 2, cf.~\cite{SmoothEnds}. Here null immersion means that
$\det(dF)=0$ and $dF$ has no zeros. The kernels and images of $dFF^{-1}$ then
coincide and extend holomorphically through the poles of $F$. The holomorphic
map
\begin{equation*}
  L^\sharp=\ker(dFF^{-1})=\image(dFF^{-1})
\end{equation*}
into the round 2--sphere $\theset{[x,1]}{x\in\C}\cup\{[1,0]\}\subset\HP^1$ is
then called the hyperbolic Gauss map of $L$.  The hyperbolic Gauss map
$L^\sharp$ of a Bryant sphere $L$ with smooth ends extends through the ends to
a rational map from $\CP^1$ to $\CP^1$.

  \begin{figure}[t!]
  \centering\newcommand{\scaling}{.08}
  \includegraphics{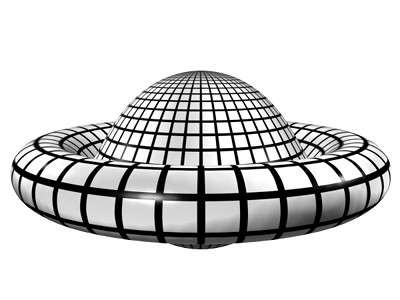}\hfill
  \includegraphics{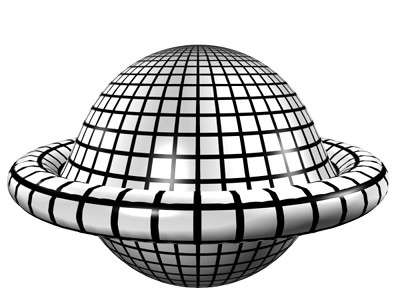}\hfill
  \includegraphics{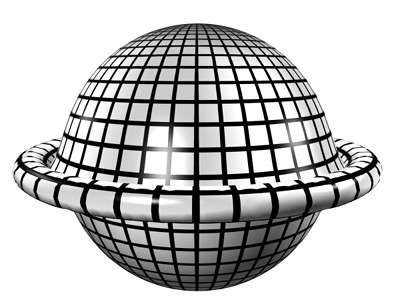}
  
  \vspace{2ex}
  \includegraphics{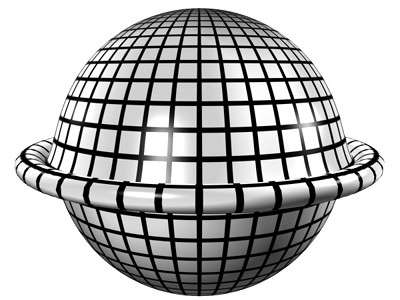}\hfill
  \includegraphics{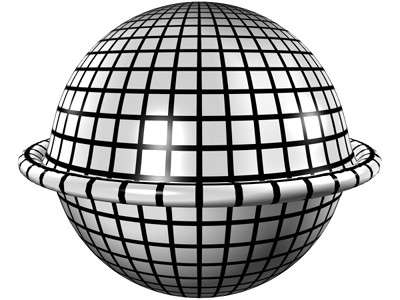}\hfill
  \includegraphics{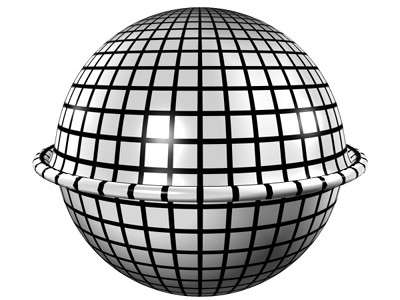}
  
  \caption{Catenoid cousins: $W=8\pi(\mu+1)$, $\mu=1,2,3,4,5,7$.}\label{ccs}
  \end{figure}

\begin{The}[\cite{SmoothEnds}]\label{the:bryant_soliton}
  Bryant spheres with smooth ends are 2--soliton spheres.
\end{The}

\begin{proof}
  Denote by $\spr{\tvector{x_1\\x_2},\tvector{y_1\\y_2}}=\bar x_2 \jj y_1 -
  \bar x_1 \jj y_2$ the indefinite Hermitian form whose null lines are the
  round 3--sphere $\{[x,1]\mid x\in \Span_\R\{1,\ii,\kk\}\} \cup \{ [1,0]\}$
  in $\HP^1$.  Then $L=\psi\H$, $\psi=F\tvector{\kk\\1}$ and
  $L^\sharp=\ker(dFF^{-1})=\image(dFF^{-1})$ as above are maps into this
  3--sphere. Denote by $\psi^\sharp$ the section of $\Gamma(L^\sharp)$ defined
  away from the poles of $F$ by
  \begin{equation*}
    \langle \psi^\sharp,\psi\rangle=1.
  \end{equation*}
  The claim follows from Theorem~\ref{the:2--solitons_DT} once we show that
  $\psi^\sharp$ is the prolongation of a holomorphic section of $\H^2/L$ that
  extends smoothly through the poles of $F$. Using that $dF$ takes values in
  the null lines $L^\sharp$, we obtain
   \begin{equation*}
   	\spr{\nabla \psi^\sharp, \psi}
   	= -\spr{\psi^\sharp, \nabla \psi}
   	= -\spr{\psi^\sharp, dF\tvector{\kk\\1}}=0.
   \end{equation*}
   Hence $\nabla\psi^\sharp$ takes values in the null lines $L$ and
   $\psi^\sharp$ is the prolongation of the holomorphic section
   $\pi\psi^\sharp$ of $\H^2/L$.  To see that $\pi\psi^\sharp$ extends
   smoothly through the poles of $F$, let $z$ be a local coordinate centered
   at a pole of~$F$. Changing coordinates in $\H^2$ one may assume that
   $F=z^{-n}\tvector{a&b\\c&d}$ for some $n\in\N$ with holomorphic functions
   $a,b,c,d$ and $d(0)\neq0$, see \cite[Lemma~3]{SmoothEnds}. Then
   $\pi\psi^\sharp=\pi e_1(-\jj\bar z^n(-c\kk+\bar d)^{-1})$ which implies
   that the holomorphic section $\pi\psi^\sharp$ extends smoothly through the
   poles of $F$.
\end{proof}

\begin{Exa}[Catenoid Cousins]\label{exa:cc}
  The holomorphic null immersion
  \begin{equation*}
  	[a,b,c,d,e]=\left[-\mu, (\mu+1) z^{2\mu+1}, 
  	-(\mu+1) z, \mu z^{2\mu+2}, \sqrt{2\mu+1} z^{\mu+1}\right]
  \end{equation*}
  into the 3--quadric $Q^3=\theset{[a,b,c,d,e]\in \CP^4}{ad-bc=e^2}$ has as
  affine part the holomorphic null immersion $F=\frac1e\tvector{a&b\\c&d}$
  into $\SL(2,\C)$. The parameters $\mu>-1$, $\mu\neq0$ yield, via
  $L=F\tvector{\kk\\1}\H$, Bryant's catenoid cousins~\cite{Br87}.  The ends of
  a catenoid cousin are smooth if and only if $\mu\in\N\setminus\{0\}$.
  Catenoid cousins with smooth ends are the simplest examples of Taimanov
  soliton spheres, cf.~\ref{exa:cc-taimanov}.  Their Willmore energy is
  $W=8\pi(\mu+1)$. Figure~\ref{ccs} shows $f\colon\CP^1\to\R^3=\Im\H$ defined
  by $\tvectork{f\\1}=\tvector{\jj&\ii\\\kk&1}F\tvector{\kk\\1}\H$ for
  different $\mu\in\N_*$.
\end{Exa}

\begin{Exa}[Bryant spheres with arbitrarily many smooth ends]
  The two ends of a catenoid cousin have order $\mu+1$.  Applying the
  transformation $(a,b,c,d,e)\mapsto (a,b,c, s^2 a+d-2se,-sa+e)$ followed by
  $(a,b,c,d,e)\mapsto (a+t^2d-2te,b,c,d,-td+e)$ allows to deform each end to
  $\mu+1$ ends of order 1, see Figures~\ref{cc2-1}--\ref{cc8-2}.
  
  \begin{figure}[t!]
  \centering\newcommand{\scaling}{.14}
  \parbox{.48\textwidth}{%
	  \includegraphics{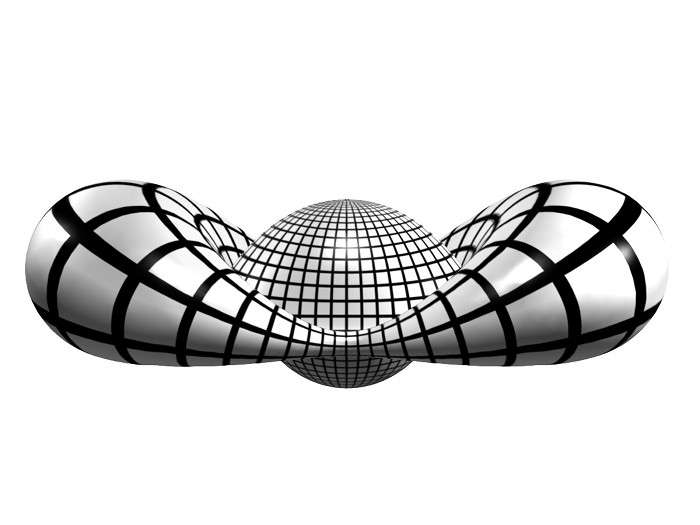}
	  \caption{$\mu=2$, $W=16\pi$, $s=0.22$, $t=0$.}
	  \label{cc2-1}}\hfill
	  \refstepcounter{figure}
  \parbox{.48\textwidth}{%
	  \includegraphics{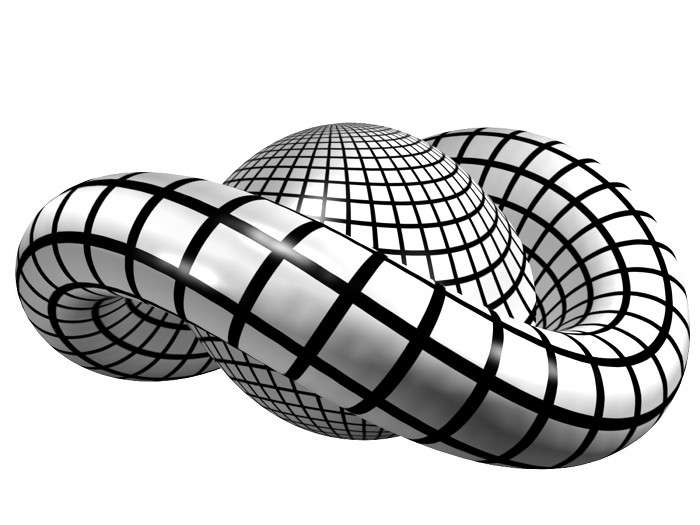}
	  \caption{$\mu=2$, $W=16\pi$, $s=0.22$, $t=0.22$.}
	  \label{cc2-2}}
  \end{figure}

  \begin{figure}[t!]
  \centering\newcommand{\scaling}{.14}
  \includegraphics{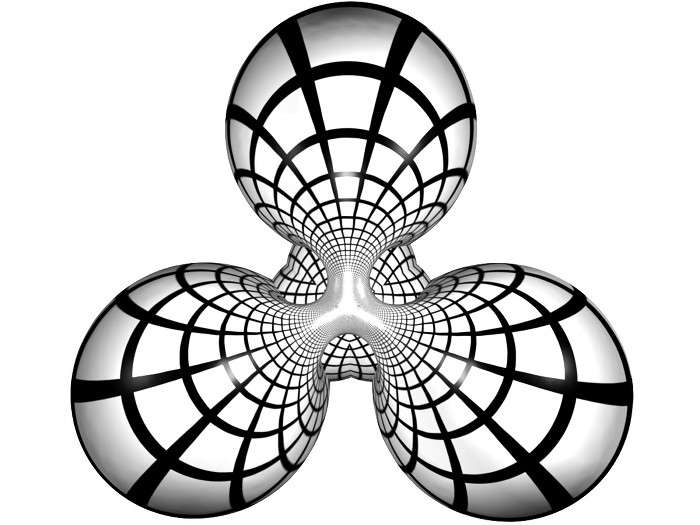}\hfill
  \includegraphics{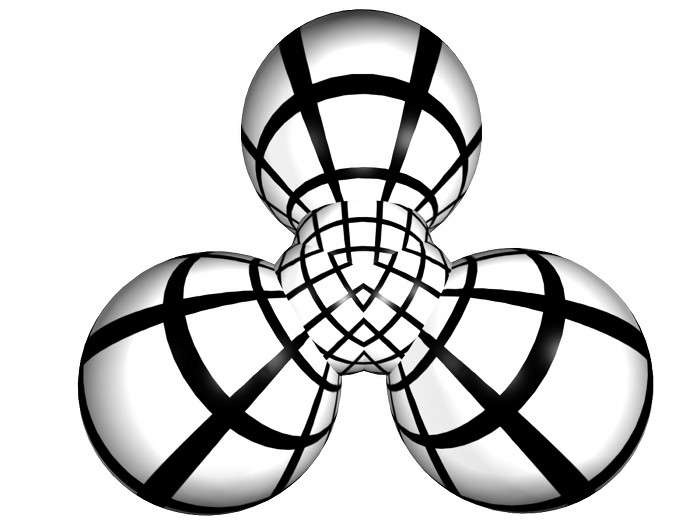}
  
  \caption{$\mu=3$, $W=24\pi$, $s=2.3\ii$, $t=-0.33\ii$, two views.}\label{cc3}
  \end{figure}  

  \begin{figure}[t!]\vspace{3ex}
  \centering\newcommand{\scaling}{.14}
  \parbox{.48\textwidth}{%
	  \includegraphics{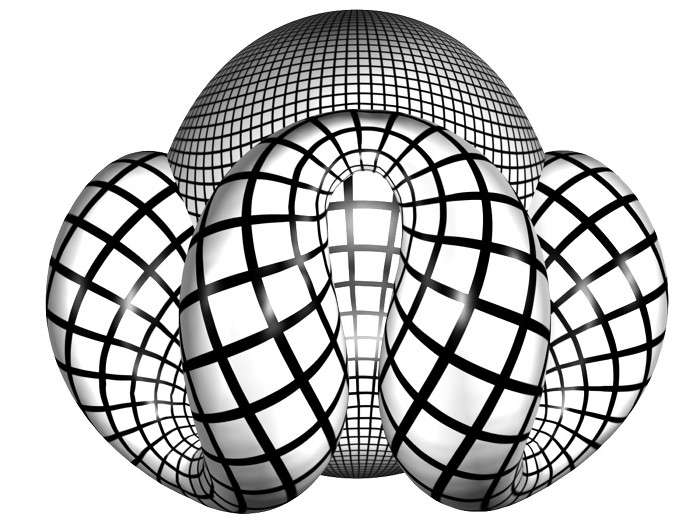}
	  \caption{$\mu=4$, $W=32\pi$, $s=0.72\ii$, $t=-0.54\ii$.}
	  \label{cc4-1}}\hfill
	  \refstepcounter{figure}
  \parbox{.48\textwidth}{%
	  \includegraphics{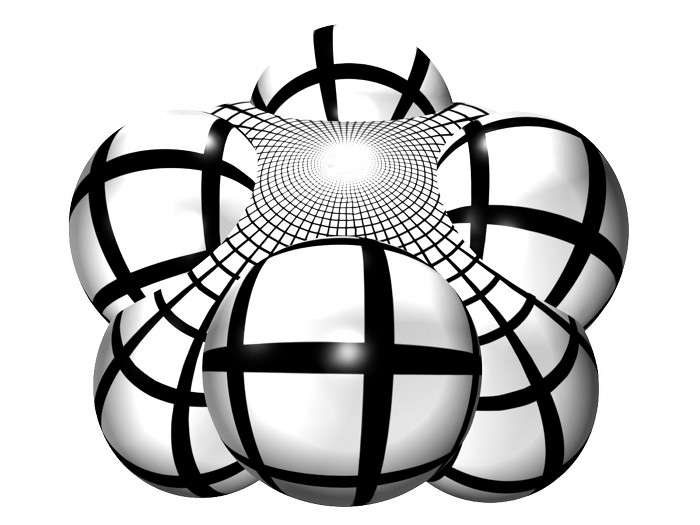}
	  \caption{Möbius inversion of Figure~\ref{cc4-1}.}
	  \label{cc4-2}}
  \end{figure}

  \begin{figure}[t!]\vspace{3ex}
  \centering\newcommand{\scaling}{.14}
  \parbox{.48\textwidth}{%
	  \includegraphics{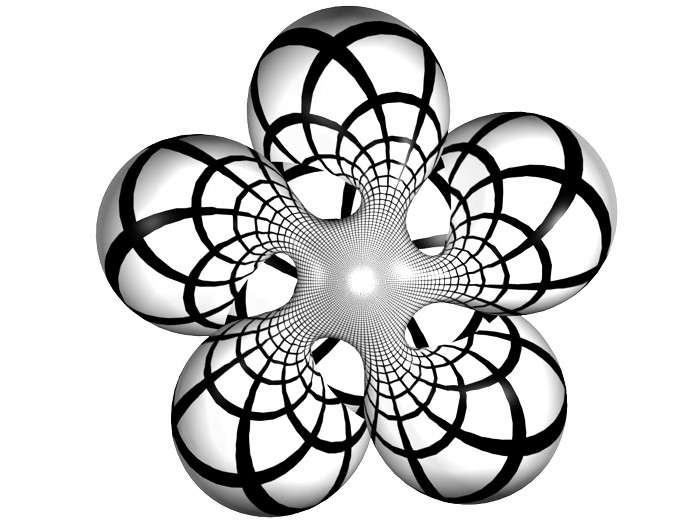}
	  \caption{$\mu=5$, $W=40\pi$, $s=1.6\ii$, $t=-0.36$.}
	  \label{cc5}}\hfill
	  \refstepcounter{figure}
  \parbox{.48\textwidth}{%
	  \includegraphics{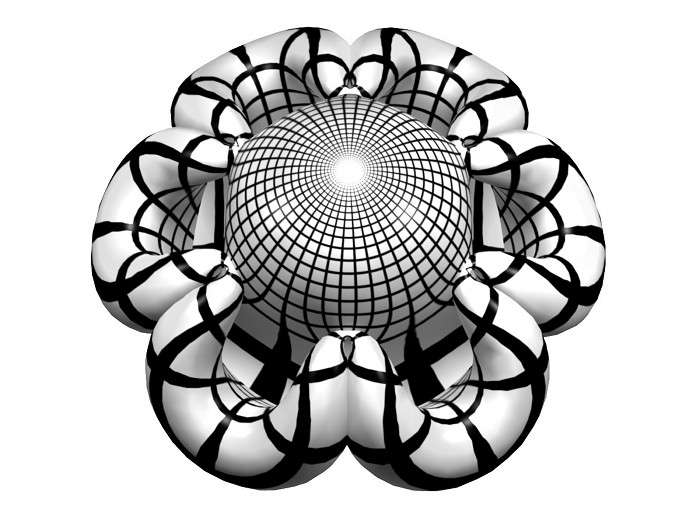}
	  \caption{$\mu=6$, $W=48\pi$, $s=0.22$, $t=-0.57$.}
	  \label{cc6}}
  \end{figure}

  \begin{figure}[t!]\vspace{3ex}
  \centering\newcommand{\scaling}{.14}
  \parbox{.48\textwidth}{%
	  \includegraphics{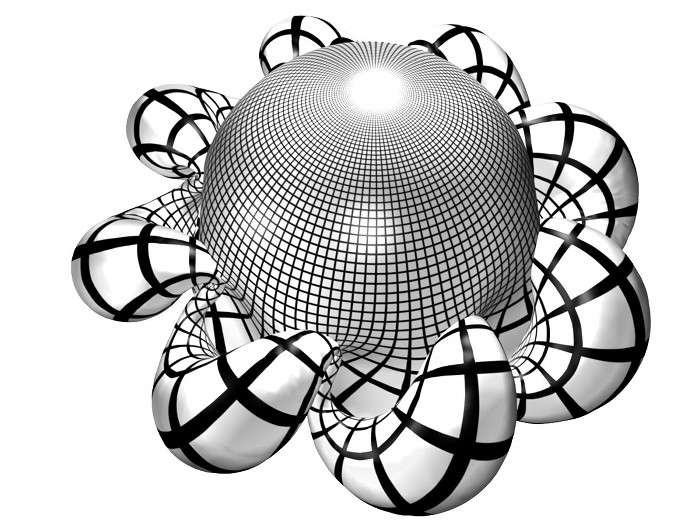}
	  \caption{$\mu=8$, $W=64\pi$, $s=-0.09-0.24\ii$, $t=-0.42$.}
	  \label{cc8-1}}\hfill
	  \refstepcounter{figure}
  \parbox{.48\textwidth}{%
	  \includegraphics{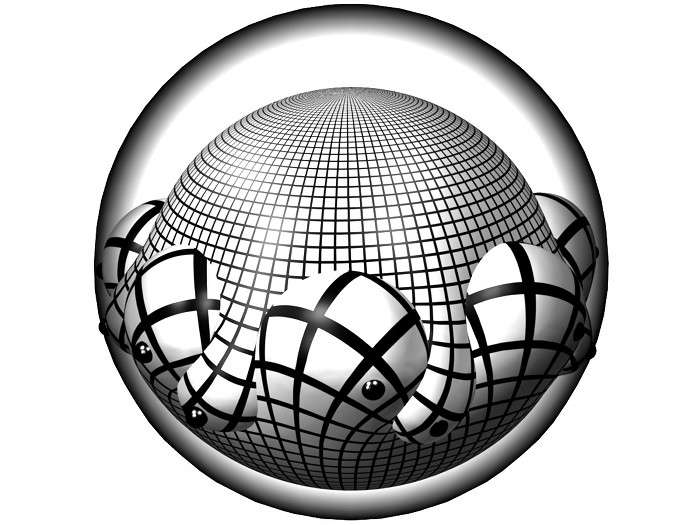}
	  \caption{Möbius inversion of Figure~\ref{cc8-1} in Poincar\'e ball
            with marked ends.} 
	  \label{cc8-2}}
  \end{figure}
\end{Exa}

\section{Example: Willmore spheres}\label{sec:Willmore}

We show that Willmore spheres in the conformal 4--sphere $S^4=\HP^1$ are
examples of soliton spheres.

\subsection{Mean curvature sphere congruence}\label{sec:mean-curv-sphere}
In case of immersed holomorphic curves in $\HP^1$, the canonical complex
structure defined in Section~\ref{sec:canonical_complex} can be interpreted as
the mean curvature sphere congruence.

The oriented totally umbilic 2--spheres in the conformal 4--sphere $\HP^1$ are
in one--to--one correspondence with the quaternionic linear complex structures
on $\H^2$: let $S\in
\End(\H^2)$ such that $S^2=-1$, then
\begin{align*}
  \mathcal S=\theset{[x]\in\HP^1}{[Sx]=[x]}
\end{align*}
is a totally umbilic 2--sphere in $\HP^1$. The complex structures $S$ and $-S$
define the same 2--sphere, but different orientations on the tangent spaces
$T_{p}\mathcal S=\H^2/p$, $p\in\mathcal S$. One therefore calls a map $S\colon
M\to\End(\H^2)$ with $S^2=-1$ a \emph{sphere congruence}. Its derivative may
be decomposed
\begin{equation*}
  \nabla S= 2({*}Q-{*}A)
\end{equation*}
into its $\bar K$--part $2{*}Q$ and its $K$--part $-2{*}A$ satisfying
${*}Q=-SQ=QS$ and ${*}A=SA=-AS$. The $\End(\H^2)$--valued 1--forms $A$ and $Q$
are called the \emph{Hopf fields} of $S$. An immersed holomorphic curve
$L\subset \H^2$ in $\HP^1$ admits \cite[Section~5]{BFLPP02} a unique sphere
congruence satisfying
\begin{align*}
  SL&=L,& *\delta&=S\delta=\delta S,& L\subset \ker(Q)\quad\text{(or,
    equivalently, }\image(A)\subset L)
\end{align*}
which is called the \emph{mean curvature sphere congruence} of $L$ and
coincides with the canonical complex structure defined in
Section~\ref{sec:canonical_complex}.  The name mean curvature sphere
congruence reflects the fact that the sphere $S_p$ at a point $p\in M$ is the
unique sphere that touches the curve at $L_p$ and has the same mean curvature
vectors with respect to any compatible metric of the conformal 4--sphere
$S^4$, cf.~\cite[Section~5.2]{BFLPP02}.

The Hopf fields $A$ and $Q$ measure the change of $S$ along the curve. The
integrals $2\int\langle A\wedge *A\rangle$ and 
$2\int\langle Q\wedge *Q\rangle$ measure the global
change of $S$ and coincide with the Willmore energies of the quaternionic
holomorphic line bundles $L^{-1}$ and $(L^\perp)^{-1}=\H^2/L$ which Kodaira
correspond, as in Section~\ref{sec:holom-curv-and-linear-systems}, to $L$ 
and $L^\perp$, respectively.

It can be shown, e.g.~\cite[Section~6]{BFLPP02}, that an immersed holomorphic
curve in $\HP^1$ is Willmore, i.e., a critical point of the Willmore energy,
if and only if its mean curvature sphere is harmonic. This is equivalent to
\begin{align*}
  d{*}A=0 \qquad\text{ which is again equivalent to }\qquad d{*}Q=0.
\end{align*}
Special examples of Willmore surfaces are twistor holomorphic curves which are
characterized by $A\equiv0$, see
Lemma~\ref{lem:twistor_hol_A=0}, and curves with $Q\equiv0$ for which 
the dual curve $L^\perp$ is twistor holomorphic.

\subsection{Willmore spheres in the 4--sphere}\label{sec:willmore_s4}
Bryant's classification~\cite{Br84} of Willmore spheres in the conformal
3--sphere has the following extension to the conformal 4--sphere
\cite{Ejiri,Musso,Montiel,BFLPP02}: an immersed Willmore sphere $L\subset
\H^2$ in the conformal 4--sphere $S^4=\HP^1$ is either
\begin{itemize}
	\item twistor holomorphic, which is equivalent to $A\equiv0$,
	\item its dual $L^\perp$ is twistor holomorphic, which is equivalent to
  		$Q\equiv0$,
	\item or it is Euclidean minimal,
\end{itemize}
where we call a holomorphic curve $L$ in $\HP^1$ \emph{Euclidean minimal} if
it is minimal in the Euclidean space $\H=\HP^1\setminus \{\infty\}$ for some
point $\infty\in\HP^1$. This is equivalent to the Möbius invariant condition
that all mean curvature spheres of $L$ intersect in one point. If a Euclidean
minimal curve in $\HP^1$ is immersed, the corresponding minimal immersion into
$\H=\HP^1\backslash \{\infty\}$ has planar ends~\cite{Br84} at the points
where the curve goes through $\infty$.

Theorem~\ref{T:equality_is_holomorphic_twistorlift_of_Ld} immediately implies
that the first two cases are soliton spheres with equality in the Pl\"ucker
estimate for the canonical linear system: if $A\equiv0$, then $L$ itself is
twistor holomorphic and equality in the Pl\"ucker estimate holds for the
canonical linear system of $(L^\perp)^{-1}=\H^2/L$. If $Q\equiv0$, then
$L^\perp$ is twistor holomorphic and equality holds for the canonical linear
system of $L^{-1}=\H^2/L^\perp$. It therefore remains to show that Euclidean
minimal spheres are soliton spheres.

\subsection{Euclidean minimal curves}\label{sec:eucl-minim-curv}
Let $L\subset\H^2$ be an immersed Euclidean minimal curve with mean curvature
sphere congruence $S$ and $\infty=[x]\in\HP^1$ the point contained in all mean
curvature spheres. Then $[S_px]=[x]$ for all $p\in M$ such that $\nabla Sx$
takes values in the subspace $[x]\subset \H^2$.  Using the type decomposition
$\nabla S=2{*}Q-2{*}A$, this implies
\begin{equation*}
  \im({*}Q)\subset [x]\subset \ker(*A),
\end{equation*}
because $L\subset \ker(Q)$, $\im(A) \subset L$ and $[x]=L_p$ at isolated $p\in
M$ only.  In particular $L$ is Willmore, because $d{*}Q=0$ which follows from
$d{*}Qx=\frac12 d (\nabla S)x=0$ and $d{*}Q\psi ={*}Q\land \delta\psi =0$ for
all $\psi\in\Gamma(L)$.

\begin{The}\label{T:willmore_spheres_are_soliton_spheres}
  Immersed Willmore spheres in $\HP^1$ are soliton spheres. 
\end{The}

\begin{proof}
  As seen in Section~\ref{sec:willmore_s4} it suffices to show that every
  immersed Euclidean minimal sphere $L\subset \H^2$ whose Hopf field $A$ does
  not vanish identically is a soliton sphere. We fix a point
  $\infty=[e_1]\in\HP^1$ such that $L$ does \emph{not} go through $\infty$
  and, using the notation of Section~\ref{sec:char-terms-weierst}, write
  $L=\psi\H$ with $\psi=\tvector{f \\ 1}$ for $f\colon \CP^1\rightarrow \H$.
  Then $f\colon \CP^1\rightarrow \H$ is not minimal in the Euclidean space
  $\R^4=\H$.

  Because $\CP^1$ is simply connected, there is a globally defined 1--step
  forward B\"acklund transform $g\colon \CP^1\to\H$ of $f$ that satisfies
  \begin{equation*}
    dg=e_2^*(2{*}Ae_1)
  \end{equation*}
  (see Appendix~\ref{sec:1-step-baecklund}).  It is non--constant, because $f$
  is not Euclidean minimal.  By assumption there is $a\in \H$ such that
  $(f-a)^{-1}$ is Euclidean minimal.
  Theorem~\ref{the:BT}~\ref{item:1-step-of-Euclidean-minimal}) thus implies
  that $g$ is twistor holomorphic.  By
  Theorem~\ref{T:equality_is_holomorphic_twistorlift_of_Ld} this yields that
  the linear system $H=\Span\{\psi ,\psi g\}\subset H^0(L)$ of the Euclidean
  holomorphic structure on $L$ defined by $\infty$ has equality in the
  Pl\"ucker estimate.  Hence $L$ is a soliton sphere by
  Theorem~\ref{T:Euclidean_Definition}.
\end{proof}

Using Proposition~\ref{P:2-d_equality_bundle_smooth_S} and
Corollaries~\ref{Cor:minimal_vs_twistor} and~\ref{cor:min_r3}, the proof of
Theorem~\ref{T:willmore_spheres_are_soliton_spheres} gives rise to the
following representation of Willmore spheres in the conformal 3-- and
4--sphere in terms of twistor holomorphic curves.
Appendix~\ref{sec:weierstrass-representation-and-1-step-bt} explains how this
representation is related to the Weierstrass representation of minimal
surfaces.

\begin{Cor}\label{cor:willmore-directly-from-S}
  Let $f\colon \CP^1\to\H$ be a conformally immersed sphere. Suppose that
  neither $f$ nor $\bar f$  is twistor holomorphic. Then $f$ is Willmore if 
  and only if there is a twistor holomorphic curve $L\colon \CP^1\to\HP^1$ 
  with smoothly immersed mean curvature sphere congruence $S$ such that
  \begin{equation*}
    f=e_2^*(Se_1)+c,
  \end{equation*}
  for some $c\in\H$, $e_1\in\H^2\backslash\{0\}$, and
  $e_2^*\in(\H^2)^*\backslash\{0\}$ such that $e_2^*(e_1)=0$.

  The Willmore sphere $f$ takes values in $\R^3=\Im\H$ if and only if the
  twistor holomorphic curve $L$ is hyperbolic superminimal with respect to the
  hyperbolic geometry defined by the Hermitian form
  $\spr{\tvector{x_1\\x_2},\tvector{y_1\\y_2}}=\bar x_2 y_1 + \bar x_1 y_2$
  (see Appendix~\ref{subsec:superminimal}) and \[f=\spr{a,Sa}+c\] for some
  $c\in \Im\H$ and $a\in\H^2\backslash\{0\}$ with $\spr{a,a}=0$.
\end{Cor}

\begin{Rem}\label{rem:willmore_energy}
  The Willmore energy of a Willmore sphere $f$ obtained as in
  Corollary~\ref{cor:willmore-directly-from-S} from a twistor holomorphic
  curve $L$ is
  \begin{equation*}
    W(f)=4\pi (2d-2-b),
  \end{equation*}
  where $d=-\deg(L)$ is the degree of $L$ and $b$ its branching degree,
  cf.~the proof of Lemma~\ref{lem:degree_d_spin}. If $f$ takes values in
  $\Im(\H)$, then $b=d-3$ and hence $W(f)=4\pi (d+1)$.
\end{Rem}

\begin{Exa}[Willmore spheres in $S^3$ with Willmore energy $16\pi$]%
  \label{exa:willmor_sphere_16pi}
  As an application of Corollary~\ref{cor:willmore-directly-from-S} we derive
  a formula for Willmore spheres in the conformal 3--sphere with Willmore
  energy $16\pi$, the lowest critical value of the Willmore energy for spheres
  in $S^3$ above the minimum~$4\pi$.

  Remark~\ref{rem:willmore_energy} implies that $d=3$ and $b=0$ for Willmore
  spheres in $S^3$ with Willmore energy $16\pi$. By
  Proposition~\ref{P:superminimal_polar_to_time_space_light_like_vector}, the
  twistor projection $L$ of the holomorphic curve $\hat
  L=[\varphi]\colon\CP^1\to\CP^3$ given by
  \[
    \varphi:=e_1 z +e_1\jj \tfrac16z^3 - e_2+e_2\jj\tfrac12z^2
  \]
  is hyperbolic minimal with respect to the Hermitian form in
  Corollary~\ref{cor:willmore-directly-from-S}, because in the basis $\hat
  e_1,\ldots,\hat e_6$ of
  Appendix~\ref{sec:weierstrass-representation-and-1-step-bt} the curve $\hat
  L$ has the tangent line congruence $\hat L_1=[\hat S]\colon\CP^1\to
  Q^4=\theset{[v]\in P(\Lambda^2(\H^2,\ii))}{v\land v=0}$ given by
   \[ \hat S=\varphi\land\varphi' = \tfrac1{24}
    (0,-12\ii z^2,12-z^4,-12\ii-\ii z^4,8z^3,-24 z)
  \]
  which is polar to the space like vector $\hat e_1$.

  Figure~\ref{fig:purple_red} shows the Willmore spheres $f=\spr{a,Sa}$ in
  $\R^3$ obtained for $a=e_2+e_2\jj$ (left) and $a=-e_1+e_2$ (right).  The
  left image in Figure~\ref{fig:blue_yellow} shows $f=\spr{a,Sa}$ with
  $\varphi$ replaced by $\varphi+e_2\jj3z$ and $a=e_2+e_2\jj$; the right
  images is obtained for $a=-e_1+e_2$ when $\varphi$ is replaced by
  $\varphi+e_2\jj7z$.
   
    \newcommand{\shorter}{\hspace{-4cm}\ }
  
  	\begin{figure}\label{fig:purple_red}
  		\centering\shorter
    	\includegraphics{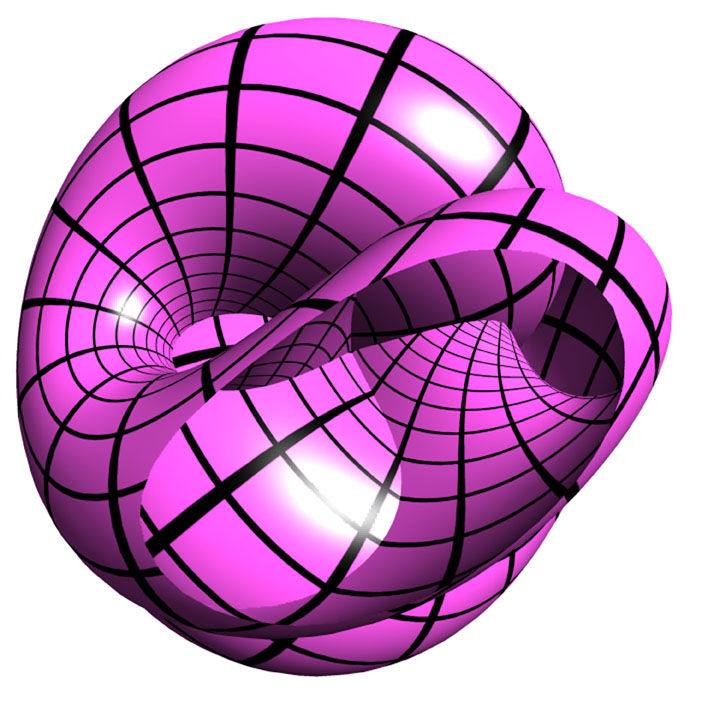}\qquad
    	\includegraphics{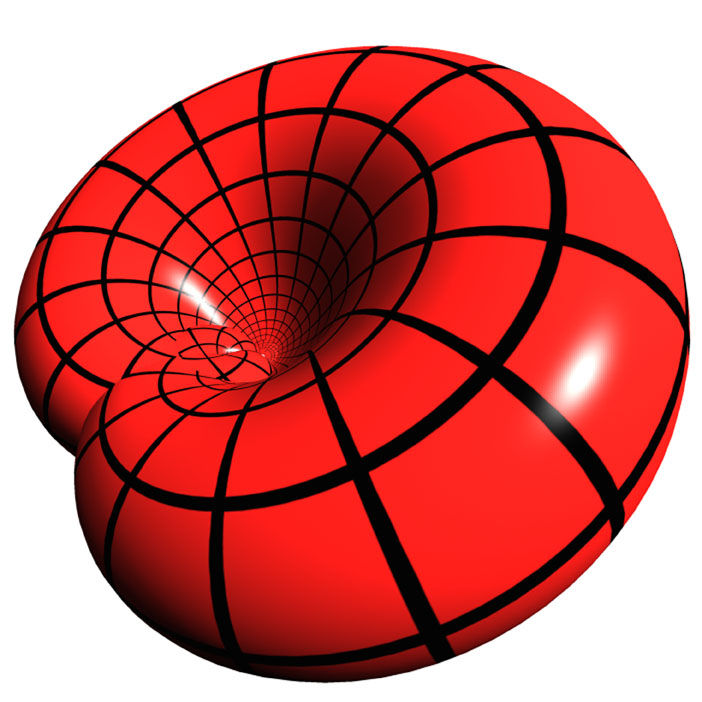}\shorter
    
    	\caption{Willmore spheres with Willmore energy $16\pi$.}
  	\end{figure}

   	\begin{figure}\label{fig:blue_yellow}
  		\centering
    	\includegraphics{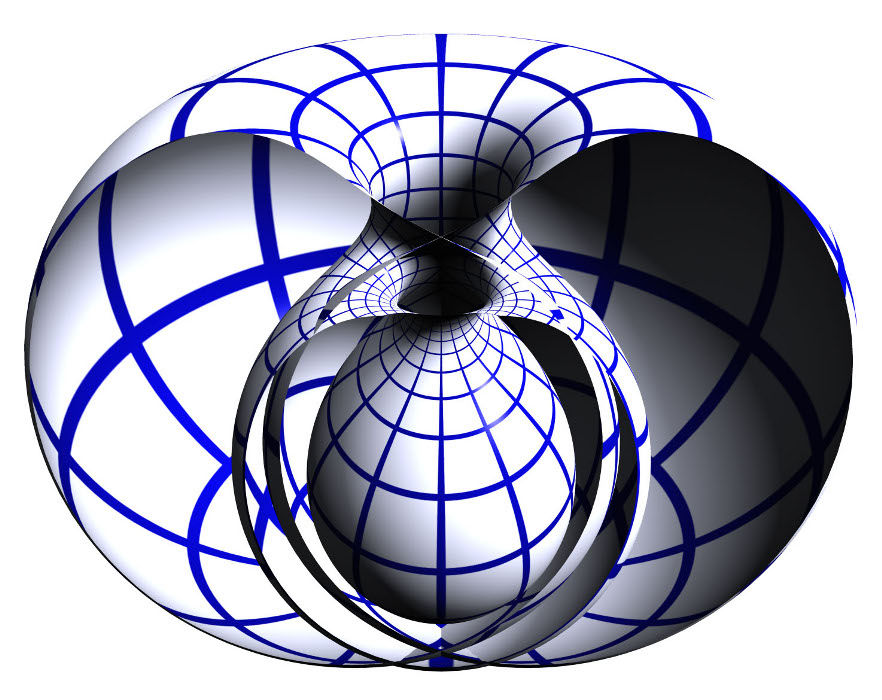}\qquad
    	\includegraphics{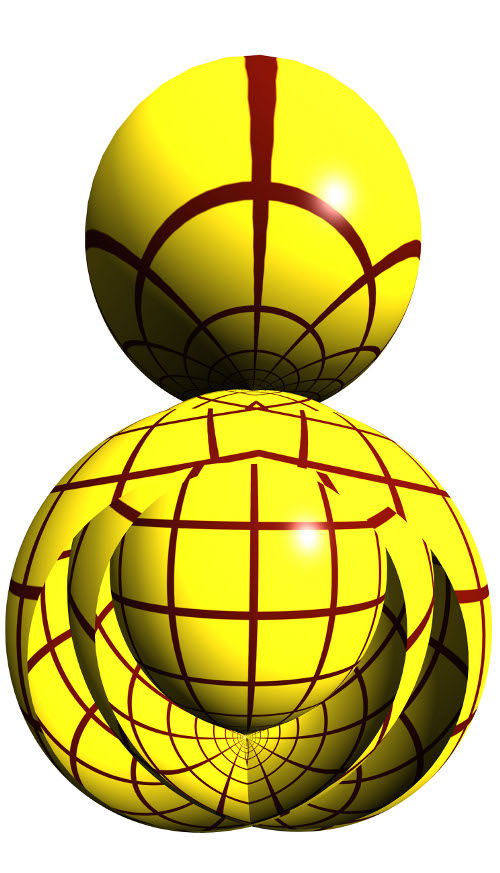}
    	    
    	\caption{Willmore spheres with Willmore energy $16\pi$.}
  	\end{figure}
\end{Exa}

\section{Willmore numbers of soliton spheres in 3--space}\label{sec:sos_s3}

In \cite{Br88} Bryant shows that the possible Willmore energies $W=\int
|\mathcal{H}|^2$ of Willmore spheres in $\R^3$ are $W=4\pi d$ with
$d\in(\N\setminus\{0,2,3,5,7\})$.  The same quantization holds for Bryant
spheres with smooth ends~\cite{SmoothEnds} and for Taimanov soliton spheres,
Corollary~\ref{cor:taimanov_quant}. In the present section we show that this
quantization more generally holds for all immersed soliton spheres in
3--space.  The main ingredient in the proof of this is
Theorem~\ref{the:special_darboux_of_soliton_spheres} which says that all
soliton spheres in 3--space with Willmore energy $W\leq 32\pi$ are Willmore
spheres or Bryant spheres with smooth ends.

\subsection{Equality in the Plücker estimate for spin bundles over
  $\CP^1$}\label{sec:equality_spin}
In order to investigate soliton spheres in the conformal 3--sphere we apply
the characterization in terms of Euclidean holomorphic line bundles given in
Section~\ref{sec:char-terms-weierst}. The advantage of the Euclidean point of
view is that, if a conformal immersion $L\subset \H^2$ into $\HP^1$ takes
values in a totally umbilic 3--sphere $S^3\subset \HP^1$, the Euclidean
holomorphic structure on $L$ defined by a point $\infty \in S^3$ not on $L$
makes $L$ into a quaternionic spin bundle, see
Section~\ref{sec:gener-weierstr-repr-R3}. 

Let $L=\psi\H$, $\psi=\tvector{f\\1}$ with affine representation
$f\colon\CP^1\to\Im\H$ and $\infty =\tvector{1\\0}\H$.  As explained in
Section~\ref{sec:char-terms-weierst}, the Euclidean holomorphic line bundle
corresponding to $\infty$ is $L$ equipped with the unique holomorphic
structure for which $\psi$ is holomorphic. Since $f$ takes values in
$\Im(\H)$, this quaternionic holomorphic line bundle $L$ is spin, i.e.,
$KL^{-1}\cong L$, and
\begin{equation*}
  d f=(\psi,\psi),
\end{equation*} 
see Theorem~\ref{the:weisterass_3d}. The Willmore energy of the Euclidean
holomorphic structure on $L$ is then $W(L)=\int |\mathcal{H}|^2$, see
Section~\ref{sec:gener-weierstr-repr-R4}.  In the following we call $W(L)$ the
Willmore energy of a soliton sphere in the conformal 3--sphere. Its relation
to the Willmore energy of the M\"obius invariant holomorphic line bundle
$\H^2/L$ is $W(L)=W(\H^2/L) +4\pi$, see Section~\ref{sec:willmore-energy}.

By Theorem~\ref{T:Euclidean_Definition}, the conformal immersion $f$ is a
soliton sphere if and only if the spin bundle $L$ admits a base point free
linear system $H\subset H^0(L)$ that contains $\psi$ and has equality in the
Plücker estimate.  Equality in the Plücker estimate for an
$(n+1)$--dimensional linear system $H\subset H^0(L)$ of a quaternionic spin
bundle $L$ over $\CP^1$ reads
\begin{equation*}
  \label{eq:equality_spin_bundle}	
  W(L)=4\pi\left[(n+1)^2+|\ord H|\right],
\end{equation*}
because $KL^{-1}\cong L$ and hence $\deg(L)=\frac12\deg(K)=-1$.

\subsection{The possible gaps in the sequence of Willmore
  numbers}\label{sec:only-20pi-28pi}
Examples of soliton spheres in the conformal 3--sphere with Willmore energy
$W(L)\in4\pi(\N\setminus\{0,2,3,5,7\})$ can be found in each of the special
classes of soliton spheres discussed in Sections~\ref{sec:Dirac_Taimanov},
\ref{sec:examples-ii-bryant_spheres}, and~\ref{sec:Willmore}, that is, among
immersed Taimanov soliton spheres, Corollary~\ref{cor:taimanov_quant}, Bryant
spheres with smooth ends~\cite{SmoothEnds}, and Willmore spheres in the
conformal 3--sphere, see~\cite{Br88}.  It is therefore sufficient to show that
$W\in 4\pi\{2,3,5,7\}$ does not occur as Willmore energy of immersed soliton
spheres in 3--space.

\begin{Lem}\label{lem:below_16pi}
  The only soliton sphere in 3--space with $W(L)<16\pi$ is the round sphere
  for which $W(L)=4\pi$.
\end{Lem}

\begin{proof}
  If $W(L)< 16\pi$, then $n=0$ and the linear system $H\subset H^0(L)$ above
  is 1--dimensional. Because it is base point free, it has no Weierstrass
  points.  Thus $W(L)=4\pi$ which implies that the immersion is
  the round sphere~\cite{W}.
\end{proof}

This shows that the Willmore energies $8\pi$ and $12\pi$ do not occur so that
it remains to check that $20\pi$ and $28\pi$ are impossible.

\subsection{Soliton spheres in 3--space related to superminimal
  curves}\label{sec:relat-superm-curv}
If $L$ is the spin bundle of a soliton sphere in 3--space with $16\pi\leq
W(L)\leq 32\pi$, the base point free linear system $H$ with equality in the
quaternionic Plücker estimate in Section~\ref{sec:equality_spin} is
2--dimensional and ``full'' in the sense that $H=H^0(L)$.  In particular we
are in the situation described by
Proposition~\ref{P:2-d_equality_bundle_smooth_S}, i.e., the dual curve $L^d$
of $H\cong \H^2$ is twistor holomorphic, extends through the Weierstrass
points of $H$, and has a mean curvature sphere congruence which is
everywhere defined and smooth.

\begin{Lem}\label{lem:degree_d_spin}
  Let $L^d\subset \H^2$ be a twistor holomorphic curve over $\CP^1$ with mean
  curvature sphere congruence that is everywhere defined and smooth.  Then
  $L\cong \H^2/L^d$ has the degree $\deg(L)=-1$ of a spin bundle if and only
  if
  \[ 
  	|b(L^d)|=d-3, 
  \] 
  where $|b(L^d)|$ is the branching order and $d=-\deg(L^d)$ the degree of the
  holomorphic curve $L^d$. The Willmore energy $W(L)$ of $L$ is then
  \[ W(L) =4\pi (d+1).\]
\end{Lem}

\begin{proof}
  Because the derivative $\delta \colon L^d\to KL$ of $L^d$ is complex
  holomorphic \cite[Section~13.2]{BFLPP02} one gets
  \begin{equation*}
    |b(L^d)|=\deg(KL)-\deg(L^d).
  \end{equation*}
  Hence, $\deg(L)=-1$ if and only if $|b(L^d)|=d-3$.  By
  Proposition~\ref{P:Branch_points_of_holomorphic_curves}, the Weierstrass
  order of $H=H^0(L)$ coincides with the branching order $b(L^d)$ of~$L^d$.
  Equality in the Plücker estimate for $H$ thus yields
  \begin{equation*}
    W(L)=4\pi\left[4+|b(L^d)|\right].
  \end{equation*}
\end{proof}

In order to derive a condition on the dual curve $L^d$ which guaranties that
$L$ is spin, i.e., $KL^{-1}\cong L$, we make use of a bundle isomorphism
induced by~$L^d$: let $L^d\subset H\cong \H^2$ be a twistor holomorphic curve
with everywhere defined and smooth mean curvature sphere congruence $S$, cf.\
Proposition~\ref{P:2-d_equality_bundle_smooth_S}.  Assuming that the Willmore
energy $W(L)$ of $L\cong H/L^d$ is not zero, the Hopf field $Q$ of $L^d$ does
not vanish identically and defines a unique holomorphic curve
\begin{displaymath}
	\tilde L\subset \ker(Q^*)\subset H^*
\end{displaymath}
called the 2--step forward Bäcklund transformation of $(L^d)^\perp$, see
Appendix~\ref{sec:bt}.  Because $L^d\subset\ker(Q)$ and $*Q=-S Q=Q S$, the
bundle homomorphism $*Q^*$ maps the trivial $H^*$--bundle to the bundle
$KL^{-1}$ of $(1,0)$--forms with values in $L^{-1}\cong(L^d)^\perp\subset
H^*$. This gives rise to a quaternionic line bundle homomorphism
\begin{align*}
  *Q^*\colon H^*/ \tilde L\to KL^{-1}  
  \qquad\quad (x\mod \tilde L) \mapsto
  {*}Q^*x
\end{align*}
which is complex linear with respect to the complex structure on $H^*/\tilde
L$ induced by $-S^*$ and the usual complex structure induced by $S^*$
on $L^{-1}\cong(L^d)^\perp\subset H^*$. 

\begin{Lem}\label{lem:*Q*hol} 
  Let $L^d\subset H\cong \H^2$ be a twistor holomorphic curve with everywhere
  defined and smooth mean curvature sphere congruence $S$ and non--trivial
  $Q$. Denote by $\tilde L\subset \ker(Q^*)$ the 2--step Bäcklund transform
  of~$(L^d)^\perp$.  Then $*Q^*\colon H^*/ \tilde L\to KL^{-1}$ is a
  holomorphic bundle homomorphism, where $KL^{-1}$ is equipped with the
  holomorphic structure paired with that on $L=H/L^d$
  (Section~\ref{sec:gener-weierstr-repr-R4}) and $H^*/ \tilde L$ is equipped
  with the unique holomorphic structure with respect to which all projections
  of constant sections of $H^*$ are holomorphic and whose complex structure is
  induced by $-S^*$.
\end{Lem}

\begin{proof}
  In order to check that the above holomorphic structure on $H^*/\tilde{L}$ is
  well defined we have to distinguish two cases: if $\tilde L$ is constant,
  the holomorphic structure is given by $D=\nabla''=\frac12(\nabla+*J\nabla)$
  where $J$ denotes the complex structure on $H^*/\tilde{L}$ induced by $-S^*$
  and $\nabla$ the trivial connection of $H^*/\tilde L$.  If $\tilde L$ is
  non--constant, then $H^*/\tilde{L}$ is the canonical holomorphic line bundle
  of the holomorphic curve $\tilde{L}^\perp$ as defined in
  Section~\ref{sec:holom-curv-and-linear-systems}, because $0=d{*}Q^*|_{\tilde
    L}=-{*}Q^*\wedge\tilde \delta$ implies $\tilde \delta
  =-S^*\tilde\delta$. Holomorphicity of the bundle homomorphism induced by
  $*Q^*$ also follows from $d{*}Q=0$: by definition of the holomorphic
  structure on $KL^{-1}$, a section $\alpha \in \Gamma(KL^{-1})$ is
  holomorphic if and only if $\alpha(a)$ is closed for all $a\in H$. But $*Q^*
  y\in\Gamma(KL^{-1})$ is closed for every $y\in H^*$ and hence $*Q^*$ maps
  the projection to $H^*/\tilde L$ of every constant section of $H^*$ to a
  holomorphic section of $KL^{-1}$,
  cf.~Section~\ref{sec:gener-weierstr-repr-R4}.
\end{proof}

Assume now that $L$ is a spin bundle, i.e., $KL^{-1}\cong L$.  The composition
of the bundle homomorphism $2{*}Q^*$ with the isomorphism $KL^{-1}\cong L$ is
then a holomorphic bundle homomorphism
\begin{align*}
  B\colon H^*/ \tilde L \to L.
\end{align*}
The image $B(H^*)\subset H^0(L)$ under $B$ of the space of sections of
$H^*/\tilde L$ obtained by projecting constant sections of $H^*$ is a 1-- or
2--dimensional subspace of $H^0(L)$ depending on whether $\tilde L$ is
constant or non--constant.

\begin{Rem}
  One can show that if $H$ has equality in the Plücker
  estimate, then $B(H^*)$ has equality and is either contained in the canonical
  linear system $H\subset H^0(L)$ of the curve $L^d$ or $H\, \cap \,
  B(H^*)=\{0\}$ and $B(H^*)$ is 2--dimensional. 
\end{Rem}

The following proposition characterizes the twistor holomorphic curves
$L^d\subset H\cong \H^2$ with smooth $S$ and non--trivial $Q$ for which
$L=H/L^d$ is spin and $B(H^*)$ is contained in the canonical linear system
$H\subset H^0(L)$ of~$L^d$.  Note that if $W(L)\leq 32\pi$ the Plücker
estimate implies $B(H^*) \subset H$, because then $H=H^0(L)$.

\begin{Pro}\label{P:2-d_equality_in_spin_bundle}
  Let $L$ be a quaternionic holomorphic line bundle over $\CP^1$ and $H\subset
  H^0(L)$ a base point free, $2$--dimensional linear system with equality in
  the Plücker estimate and dual curve $L^d\subset H\cong \H^2$.  Then $L$ is a
  quaternionic spin bundle with $B(H^*)\subset H$ if and only if either
  \begin{enumerate}[a)]		
  \item\label{i:type_a} $L^d$ is spherical or hyperbolic superminimal with 
    everywhere defined and immersed mean curvature sphere congruence, or
  \item\label{i:type_b} $L^d$ is Euclidean superminimal in $\HP^1\setminus
    \{\infty\}$ for some point $\infty\in\HP^1$ and the intersection divisor
    of $L^d$ with $\infty$ is equal to the branching divisor of the globally
    defined mean curvature sphere congruence of $L^d$.
  \end{enumerate}
\end{Pro}

\begin{Rem}
  By Proposition~\ref{P:superminimal_polar_to_time_space_light_like_vector},
  the tangent line congruence $\widehat{L^d_1}$ of the twistor lift
  $\widehat{L^d}$ of a curve $L^d$ belonging to type~\ref{i:type_a}) is a
  rational null immersion into the complex 3--quadric.  Denote by $d$ and $d_1$
  the degree of $L^d$ and $\widehat{L^d_1}$. Using that
  $\widehat{L^d}$ is self dual, see 
  Remark~\ref{rem:E_isom_to_Ed}, the complex Plücker 
  Formula~\cite[p.~270]{GriHa} for $\widehat{L^d_1}$ implies $0=-2d+2d_1-2$,
  because $\widehat{L^d_1}$ is unbranched if $S$ is immersed, see
  Lemma~\ref{L:branching_of_twistorlift}. By Lemma~\ref{lem:degree_d_spin} we
  thus have $\frac1{4\pi}W(L)=d+1=d_1$. This excludes type~\ref{i:type_a})
  curves for which $W(L)=20\pi$ or $W(L)=28\pi$, because rational null
  immersions of degree $5$ and $7$ into the complex 3--quadric do not exist,
  cf.~\cite{Br88}.
  
  To prove the quantization of the Willmore energy for soliton spheres in
  3--space it therefore remains to exclude type~\ref{i:type_b}) curves of
  degree $4$ and $6$. A direct proof of this turns out to be very
  technical. Below such curves are excluded by a M\"obius geometric argument,
  see the proof of Theorem~\ref{the:special_darboux_of_soliton_spheres}.
\end{Rem}

\begin{proof}
  From Proposition~\ref{P:2-d_equality_bundle_smooth_S} we know that a base
  point free, 2--dimensional linear system $H\subset H^0(L)$ with equality in
  the Pl\"ucker estimate has a globally defined, twistor holomorphic dual
  curve $L^d\subset H$ whose mean curvature sphere congruence $S$ extends
  smoothly through its branch points.

  Assume that $L$ is spin and $B(H^*)\subset H$. We distinguish the cases of
  non--constant and constant B\"acklund transform $\tilde L\subset \ker(Q^*)$.

  \ref{i:type_a}) If $\tilde L$ is \emph{non--constant}, then $B$ maps the
  space $H^*$ of holomorphic sections of $H^*/ \tilde L$ onto $H\subset
  H^0(L)$.  Because by assumption $H$ is base point free, the bundle
  homomorphism $B$ is an isomorphism.  Since $\nabla S= 2{*}Q$,
  Lemma~\ref{lem:twistor_hol_A=0}, this implies that the mean curvature sphere
  congruence $S$ of $L^d$ is immersed.

  The bundle isomorphism $B$ induces an isomorphism $\tilde B\colon H^* \to H$
  of vector spaces between the 2--dimensional linear systems $H^*$ and $H$. It
  maps $\tilde L\subset H^*$ onto $L^d$ and thus maps the mean curvature
  sphere congruence $-S^*$ of $\tilde L$ onto the mean curvature sphere
  congruence $S$ of $L$, i.e., $S\tilde B=-\tilde BS^*$. Differentiating the
  last equation one obtains that $\tilde B$ maps $(L^d)^\perp=\im(Q^*)$ onto
  $\tilde L^\perp=\im(Q)$. This implies that the adjoint $\tilde B^*$ of
  $\tilde B$ also induces a bundle isomorphism $H^*/ \tilde L\to L\cong
  H/L^d$.  Because every automorphism of a quaternionic holomorphic line
  bundle with non--trivial Hopf field acts by multiplication with a real
  constant, we obtain that $\tilde B^*=\pm \tilde B$.  Now $\tilde B^*=-
  \tilde B$ is impossible: for every $x\in H^*\backslash\{0\}$ the
  non--trivial 1--form $\langle *Q^*x,\tilde B x\rangle ={*}Q^*x(Bx)$ had to
  be real--valued, because $S\tilde B=-\tilde BS^*$ implies $Q\tilde B=-\tilde
  BQ^*$ and hence
  \begin{multline*}
    \overline{\langle {*}Q^*x,\tilde B x\rangle } 
    = \langle x,\tilde B^*{*}Q^*x\rangle = -\langle x,\tilde
    B{*}Q^* x\rangle 
    = \langle x,{*}Q\tilde B x\rangle  
    = \langle {*}Q^*x , \tilde B x\rangle.
  \end{multline*}
  But this contradicts $*\langle {*}Q^*x,\tilde B x\rangle =- \langle
  {*}Q^*S^*x,\tilde B x\rangle =N\langle {*}Q^*x , \tilde B x\rangle$ 
  with $N\in\Im\H$ defined by
  $S^*x=xN\mod \tilde L$.  Hence $\tilde B^*=\tilde B$ and $\tilde B$ defines a
  non--degenerate Hermitian form on $H$ with respect to which $S$ is
  skew. Thus $L$ is spherical or hyperbolic superminimal, by
  Lemma~\ref{lem:minimal_s_skew}.

  \ref{i:type_b}) If $\tilde L$ is \emph{constant}, then $\infty=\tilde
  L^\perp\subset H$ lies on all mean curvature spheres of $L^d$. Hence
  $L^d$ is Euclidean superminimal. The statement about the intersection
  divisor may be seen as follows: the vanishing divisor of a non--trivial
  holomorphic section $x\in H\subset H^0(L)$ with $x\in \infty$ is the
  intersection divisor of $L^d$ with $\infty$.  Let $y \in H^*$ with
  $y(x)=1$. Because $y\not\in \tilde L=\infty^\perp$ it induces a nowhere
  vanishing holomorphic section of $H^*/\tilde L$.  The vanishing divisor of
  $By$ thus equals the branching divisor of $S$, since $\nabla S=2{*}Q$.

  We prove now that the holomorphic sections $By$ and $x$ differ by a real
  constant only and hence have the same vanishing divisor: because $S$ is
  smooth and $S\infty=\infty$, there is a map $N\colon \CP^1\to S^2$ with
  $Sx=xN$. On the other hand, since $y(x)=1$ we have $-S^*y \equiv y N
  \mod \tilde L$ and $S(By)=(By)N$.  Thus, away from the isolated zeros
  of $By$ and $x$ there exist real valued functions $\lambda_1,\lambda_2$ with
  $By=x(\lambda_1+\lambda_2N)$.  Since $Q$ is non--trivial, the Leibniz rule
  in Section~\ref{sec:quat-holom-line} implies that $\lambda_1$ is constant
  and $\lambda_2\equiv0$.  Hence $By=x\lambda_1$ with $\lambda_1\in \R$ such
  that $By$ and $x$ have the same vanishing divisor.

  We show now that, conversely, if $L^d\subset H$ is a superminimal curve as
  in \ref{i:type_a}) or \ref{i:type_b}), then $L=H/L^d$ is a quaternionic spin
  bundle, i.e., $KL^{-1}\cong L$, which by construction satisfies
  $B(H^*)\subset H$.

  \ref{i:type_a}) Let $L^d$ be \emph{spherical or hyperbolic superminimal}.
  Then there exists a non--degenerate Hermitian form $\thespr$ with respect to
  which $S=-S^*$, see Lemma~\ref{lem:minimal_s_skew}, and $\nabla S=2{*}Q$
  because $L^d$ is twistor holomorphic. Identifying $H$ and $H^*$ via
  $\thespr$, this implies $Q=-Q^*$ and hence $\tilde L= \ker(Q^*)=
  \ker(Q)=L^d$. Since $\im(Q)=\tilde L^\perp = (L^d)^\perp \cong
  (H/L^d)^{-1}=L^{-1}$, the Hopf field $*Q$ induces a quaternionic holomorphic
  bundle homomorphism $L\to KL^{-1}$ which is an isomorphism because by
  assumption $S$ is an immersion.

  \ref{i:type_b}) If $L^d$ is \emph{Euclidean superminimal}, then
  $\infty={\tilde L}^\perp=\image(Q)$ is a constant point. As above, let
  $x\in\infty$ and $y\in(\H^2)^*$ such that $y(x)=1$. Then $*Q^*y$ is a
  holomorphic section of $KL^{-1}$ whose vanishing divisor by assumption
  coincides with the vanishing divisor of $x$ seen as a holomorphic section of
  $L$. As before, there exists a quaternion valued map $N$ defined away from
  the zeros of $x$ and $*Q^*y$, such that $Jx=xN$ and $J{*}Q^*y=*Q^*yN$, where
  $J$ denotes the complex structures of the respective quaternionic
  holomorphic line bundles $L$ and $KL^{-1}$. This proves $L\cong KL^{-1}$,
  because the identification of two holomorphic sections with the same
  vanishing divisor and the same ``normal vector'' $N$ gives rise to a
  holomorphic line bundle isomorphism.  This follows from Riemann's removable
  singularity theorem, because holomorphic homomorphisms between two
  quaternionic holomorphic line bundles are in particular complex holomorphic
  sections of the line bundle of homomorphisms between the underlying complex
  holomorphic line bundles.
\end{proof}

\subsection{Classification of soliton spheres in 3--space with Willmore energy
  $16\pi\leq W\leq 32\pi$}\label{sec:geometry_type_ab}
Let $L\subset \H^2$ be a soliton sphere with $16\pi\leq W\leq 32\pi$ in the
conformal 3--sphere $S^3\subset \HP^1$. For every $\infty \in S^3$ not on $L$,
the induced quaternionic spin structure on $L$ then admits a 2--dimensional
linear system related to a superminimal curve as described in
Proposition~\ref{P:2-d_equality_in_spin_bundle}.  Investigation of the
corresponding 3--dimensional linear system of the Möbius invariant holomorphic
line bundle $\H^2/L$ shows that $L$ is either a Willmore sphere or a Bryant
sphere with smooth ends.

\begin{The}\label{the:special_darboux_of_soliton_spheres}
  Let $L\subset \H^2$ be a soliton sphere in the conformal 3--sphere with
  Willmore energy $16\pi\leq W \leq 32\pi$. Then the full space of
  holomorphic sections $H^0(\H^2/L)$ of the Möbius invariant holomorphic line
  bundle $\H^2/L$ is a 3--dimensional linear system with equality in the
  Plücker estimate that also contains a (unique) 1--dimensional linear system
  $H\subset H^0(\H^2/L)$ with equality in the Plücker estimate. Moreover,
  either
  \begin{itemize}
  \item $L$ is a Willmore sphere or 
  \item $L$ is a Bryant sphere with smooth ends,
  \end{itemize}
  depending on whether $H$ is contained in the canonical linear system or not.
  In the Bryant case, the Darboux transform corresponding to $H$ is the
  hyperbolic Gauss map of $L$.  In the Willmore case, the Darboux 
  transform is constant and coincides with the
  point $\infty$ for which $L$ is Euclidean minimal in
  $\HP^1\setminus\{\infty\}$.
\end{The}

\begin{Rem}
  The assumption $16\pi\leq W \leq 32\pi$ in
  Theorem~\ref{the:special_darboux_of_soliton_spheres} may be replaced by the
  weaker assumption $B(H^*)\subset H$ of
  Proposition~\ref{P:2-d_equality_in_spin_bundle}.  In the proof of
  Theorem~\ref{the:special_darboux_of_soliton_spheres} we show that a soliton
  sphere $L$ in $\R^3=\Im (\H)$ whose spin bundle is of type a) in
  Proposition~\ref{P:2-d_equality_in_spin_bundle} is a Willmore sphere or a
  Bryant sphere with smooth ends (in the ball model of hyperbolic space). A
  soliton sphere in $\R^3$ whose spin bundle is of type b) in
  Proposition~\ref{P:2-d_equality_in_spin_bundle} is a Bryant sphere with
  smooth ends (in the half space model of hyperbolic space).
\end{Rem}

For the proof of Theorem~\ref{the:special_darboux_of_soliton_spheres} we need
to derive some properties of 1--dimensional linear systems with equality in
the Pl\"ucker estimate.  If $\varphi$ is an arbitrary quaternionic holomorphic
section, then $N$ with $J\varphi =\varphi N$ is continuous at the zeros of
$\varphi$ and smooth elsewhere, cf.\ the appendix to~\cite{Boh}.  The
following lemma together with Riemann's removable singularity theorem implies
that $N$ is everywhere smooth in case $\varphi$ spans a 1--dimensional linear
system with equality in the Pl\"ucker estimate.

\begin{Lem}\label{lem:1--dim-eq}
  Let $L$ be a quaternionic holomorphic line bundle with complex structure $J$
  over a compact Riemann surface.
  \begin{enumerate}[i)]
  \item\label{i:1-d-eq-N-hol} Let $\varphi\in H^0(L)\setminus\{0\}$ and 
    define $N$ by $J\varphi =\varphi N$. The
    1--dimensional linear system spanned by $\varphi$  has equality in the
    Plücker estimate if and only if $d N'=\frac12(dN-N{*}dN)\equiv0$.
  \item\label{i:1-d-eq-Eucl-min} Let $\varphi$, $\psi\in H^0(L)$ and 
  	$\psi=\varphi f$ with
    non--constant $f$ defined away from the zeros of $\varphi$. 
    The 1--dimensional
    linear system spanned by $\varphi$ has equality in the Pl\"ucker estimate
    if and only if $f$ is Euclidean minimal.
  \item\label{i:1-d-eq-unique} If the Willmore energy of $L$ is non--zero, 
  	then $L$ has at most one linear system with equality in the Pl\"ucker 
  	estimate.
  \end{enumerate}
\end{Lem}

\begin{proof}
  \ref{i:1-d-eq-N-hol}) Equality in the Plücker estimate for the
  1--dimensional linear system spanned by $\varphi$ is equivalent to
  $W(L^{-1},\nabla'')=\frac12\int dN'\wedge*dN'=0$, see
  Section~\ref{sec:degree_formula}. This is equivalent to $dN'\equiv0$,
  because $dN'\wedge*dN$ is a positive real valued 2--form.
  \ref{i:1-d-eq-Eucl-min}) follows from \ref{i:1-d-eq-N-hol}), because $*df
  =Ndf$ and $dN'=df H$, see Appendix~\ref{sec:mean-curv-sphere-affine}.
  \ref{i:1-d-eq-unique}) Assume $\varphi$ and $\psi\in H^0(L)$ span two
  different 1--dimensional linear systems with equality.  Then $f$ and
  $f^{-1}$ defined by $\psi = \varphi f$ are both Euclidean minimal. This
  implies that $f$ is planar and, consequently, the Willmore energy of $L$
  vanishes, cf.\ Section~\ref{sec:eucl-minim-curv}.
\end{proof}

\begin{proof}[Proof of
Theorem~\ref{the:special_darboux_of_soliton_spheres}]
Let $L=\psi\H\subset \H^2$, $\psi=\tvector{f\\1}$ be a soliton sphere in the
conformal 3--sphere $S^3=\Im \H\cup \{\infty\}\subset \HP^1$ with affine
representation $f\colon\CP^1\to\Im\H$.  As explained in
Section~\ref{sec:equality_spin}, the choice $\infty=\tvector{1\\0}\H$ makes
$L$ into a quaternionic spin bundle for which $\psi\in H^0(L)$. Using the
Plücker estimate and Theorem~\ref{T:Euclidean_Definition}, the assumption
$16\pi\leq W(L)\leq 32\pi$ implies that the full space of holomorphic sections
$H^0(L)$ of $L$ is 2--dimensional and has equality in the Plücker estimate.
Since $H=H^0(L)$, we obtain in particular that the assumption $B(H^*)\subset H$
of Proposition~\ref{P:2-d_equality_in_spin_bundle} is satisfied.
  
Denote by $\varphi\in H^0(\H^2/L)$ the projection of $e_1$ to the Möbius
invariant holomorphic line bundle $\H^2/L$ of $L$. The canonical linear system
of $L$ is then spanned by $\varphi, \varphi f$,
cf.~Section~\ref{sec:holom-curv-and-linear-systems}. Let $\nabla$ be the flat
connection of $\H^2/L$ such that $\nabla\varphi=0$.  Then $\varphi d f$ is a 
holomorphic section of $K\H^2/L$ when $K\H^2/L$ is, as in
Section~\ref{sec:char-terms-weierst}, equipped with the holomorphic
structure~$d^\nabla$. The derivative $\delta\colon L\to K\H^2/L$ of $L$ is
then a complex quaternionic line bundle isomorphism, cf.\
Section~\ref{sec:holom-curv-hpn}, which is holomorphic because
$\delta$ maps the holomorphic section $\psi\in H^0(L)$ to $\varphi d f\in
H^0(K\H^2/L)$. Using this isomorphism, Lemma~\ref{L:Equality_on_the_ladder}
applied to $\varphi$ implies that $H^0(\H^2/L)$ is 3--dimensional and has 
equality in the Plücker estimate,
because $\CP^1$ is simply connected.
  
\ref{i:type_a}) Suppose that the dual curve $L^d$ of $H^0(L)$ is of
type~\ref{i:type_a}) in Proposition~\ref{P:2-d_equality_in_spin_bundle}, i.e.,
the mean curvature sphere congruence $S$ of $L^d$ is everywhere defined,
immersed, and skew with respect to a non--degenerate Hermitian form
$\thespr$. Since $L^d$ is twistor holomorphic one has $\nabla S=2{*}Q$. In the
proof of Proposition~\ref{P:2-d_equality_in_spin_bundle} it is shown that the
isomorphism of $L$ and $KL^{-1}$ is provided by $\nabla S=2{*}Q$, i.e., $d
f=(\psi,\psi)=2\spr{\psi,{*}Q\psi}$. Thus, without loss of generality
\begin{equation*}
  f=\spr{\psi,S\psi}.
\end{equation*}
  
If $\spr{\psi,\psi}=0$, then $f^{-1}$ is Euclidean minimal by
Corollary~\ref{cor:min_r3}, and  Lemma~\ref{lem:1--dim-eq} implies that the
1--dimensional linear system spanned by $\varphi f$ has equality in the
Plücker estimate.
  
If $\spr{\psi,\psi}\neq0$ we may assume that $\spr{\psi,\psi}=1$. Let
$e_1,e_2$ be a basis of $H^0(L)$ such that $e_2=\psi$, $\spr{e_1,e_2}=0$, and
$L^d=\tvectork{g\\1}$ with nowhere vanishing~$g$. From the formula for $S$
in Appendix~\ref{sec:mean-curv-sphere-affine} we then obtain
\begin{equation*}
  f=\spr{\psi,S\psi}=H_g g -R_g.
\end{equation*}
Since $L^d$ is twistor holomorphic, $d R_g=d R'=H_g d g$ and hence $df=d H_g
g$. By Lemma~\ref{lem:dt_affine}, this implies that $f^\sharp=f-H_g g=-R_g$ is
a Darboux transform of $f$.  This shows that $f$ is a Bryant sphere with
smooth ends: because $f^\sharp =-R_g$ is totally umbilic and both $f$ and
$f^\sharp=-R_g$ take values in $\Im(\H)$, we obtain from
\cite{JMN01,SmoothEnds} (see also Remark~\ref{rem:bryant_dt_in_s3}) that, away
from the isolated points where $f$ and $f^\sharp$ intersect, $f$ is a Bryant
surface. Since the immersion $f$ is defined on all of~$\CP^1$, the isolated
intersection points are smooth Bryant ends.

Moreover, by Lemma~\ref{lem:1--dim-eq} the 1--dimensional linear system of
$\H^2/L$ that is spanned by $\varphi H_g$ has equality in the Plücker
estimate, since $H_g^{-1}$ is Euclidean minimal by
Corollary~\ref{Cor:minimal_vs_twistor}.
  
\ref{i:type_b}) Suppose now that the dual curve $L^d$ of $H^0(L)$ is of
type~\ref{i:type_b}) in Proposition~\ref{P:2-d_equality_in_spin_bundle}, i.e.,
there exists a line $\infty\subset H^0(L)$ such that $L^d$ is Euclidean
superminimal in $PH^0(L)\setminus\{\infty\}$. Let $e_1,e_2\in H^0(L)$ be a
basis such that $e_1=\psi$ and $e_2\in\infty$ and let $L^d=\tvectork{g\\1}$
with respect to this basis. Then $g^{-1}$ is Euclidean superminimal. The
description of the spin pairing $KL^{-1}\cong L$ in the proof of
Proposition~\ref{P:2-d_equality_in_spin_bundle} implies $(e_2,e_2) =
e_2^*(2{*}Qe_2)=d H_g g$ and
\begin{align*}
  df=(\psi,\psi)= (e_2 g^{-1},e_2 g^{-1}) = \bar g^{-1} dH_g = -d\bar H_g
  g^{-1},
\end{align*}
where the last equation holds because $df$ is $\Im(\H)$--valued. Thus
$dfg+d\bar H_g=0$ and Lemma~\ref{lem:dt_affine} implies that $f^\sharp=f+\bar
H_g g^{-1}$ is a Darboux transform of~$f$. Moreover, $\bar H_g^{-1}$ is
Euclidean minimal by Corollary~\ref{Cor:minimal_vs_twistor} and twistor
holomorphic, by \ref{item:1-step-of-Euclidean-minimal}) of
Theorem~\ref{the:BT}, because $\bar H_g$ is a $1$--step forward Bäcklund
transform of $\bar g^{-1}$ which is Euclidean minimal. Thus $f^\sharp$ is
planar by Lemma~\ref{lem:dt_twistor}, because both $g^{-1}$ and $\bar
H_g^{-1}$ are Euclidean superminimal.  By Lemma~\ref{lem:1--dim-eq}, the
1--dimensional linear system spanned by $\varphi \bar H_g$ has equality in the
Pl\"ucker estimate.

To see that $f$ is a Bryant sphere with smooth ends, by
\cite{JMN01,SmoothEnds} it remains to proof that the planar Darboux transform
$f^\sharp$ is $\Im(\H)$--valued (see Remark~\ref{rem:bryant_dt_in_s3}): by
assumption $e_2\H$ is a fixed line of $S$ such that $e_2^*(Se_2)$ is a
quaternionic that squares to $-1$ and is hence imaginary.  Because on the
other hand $e_2^*(Se_2)=H_gg-R_g$ we obtain that $H_gg$ takes values in
$\Im(\H)$. This implies $f^\sharp=f+\bar H_g g^{-1}$ is $\Im(\H)$--valued.

\end{proof}

Theorem~\ref{the:special_darboux_of_soliton_spheres} yields that the possible
Willmore energies of soliton spheres in 3--space coincide with the possible
Willmore energies of Willmore spheres in 3--space and of Bryant spheres with
smooth ends.
	
\begin{The}\label{The:gaps}
  The possible Willmore energies $W=\int |\mathcal{H}|^2$ of immersed soliton
  spheres in 3--space are $W\in 4\pi (\N\setminus\{0,2,3,5,7\})$.
\end{The}

\appendix

\section{Formulae in affine coordinates}\label{app:affine_coordinates}

This appendix relates projective invariants of holomorphic curves in $\HP^1$
to the Euclidean invariants of their representations in affine charts.

\subsection{Branch points of holomorphic curves in $\HP^1$}%
\label{subsec:branch_points_holomorphic_curves}
A \emph{branch point of order $k\in\N$\/} of a smooth map $f\colon M\to N$
from a 2--dimensional manifold $M$ to an $(n+2)$--dimensional manifold $N$ is a
point $p\in M$ for which there exists an integer $k\geq1$ and centered
coordinates $z\colon M\supset U\to\C$ and $u=(u_1,...,u_{n+2})\colon N\supset
V\to\R^{n+2}$ at $p$ and $f(p)$ that satisfy
\begin{align*}
  u_1\circ f+{\ii} u_2 \circ   f &= z^{k+1} + O(k+2),&
  u_l\circ f&=O(k+2),\ l=3,\ldots,n+2,
\end{align*}
cf.~\cite{GOR}. We write $b_p(f)=k$ for the \emph{branching order} of $f$ at
$p$ and $b(f)$ for the \emph{branching divisor} of $f$. A map $f$ is called a
\emph{branched immersion} if all points at which $df$ fails to be injective
are branch points. If $M$ is a Riemann surface, a branched immersion is called
\emph{conformal} if it is conformal away from its branch points. The following
proposition relates the branching divisor of a holomorphic curve in $\HP^1$ to
the Weierstrass divisor of its canonical linear system, cf.\
Sections~\ref{sec:holom-curv-and-linear-systems}
and~\ref{sec:weierstr-gaps-flag}.

\begin{ProA}\label{P:Branch_points_of_holomorphic_curves}
  A non--constant holomorphic curve in $\HP^1$ is a branched conformal
  immersion. Its branching divisor coincides with the Weierstrass divisor of
  its canonical linear system.
\end{ProA}

\begin{proof}
  Using the affine chart induced by the basis $e_1$, $e_2$
  (cf.~Section~\ref{sec:ident-hp1-with}), a holomorphic curve can be written
  as $L=\tvector{f\\1}\H \subset\H^2$ with $f\colon M\rightarrow \H$.  The
  induced basis $e_1^*$, $e_2^*\in H^0(L^{-1})$ of the canonical linear system
  then satisfies $e_1^*=e_2^*\bar f$.  Assume that $L_p=e_2\H$ at $p\in M$ and
  hence $f(p)=0$.  By \cite[Section~3.3]{FLPP01} there exists a centered
  holomorphic coordinate $z\colon M\supset U\to\C$ at $p$ and a nowhere
  vanishing section $\varphi\in\Gamma(L^{-1}_{|U})$, such that the holomorphic
  section $e_1^*=e_2^*\bar f$ of $L^{-1}=(\H^2)^*/L^\perp$ can be written as
  \begin{equation*}
    e_2^*\bar f=e_1^*= z^{k+1}\varphi+ O(k+2) \text{\qquad as $z\to\infty$}.
  \end{equation*}
  This proves the claim, because $b_p(f)=k=\ord_p(e^*_1)-1$ and
  $\ord_p(\H^2)^*=\ord_p(e_1^*)-1=k$ as $e_2^*$ induces a section of $L^{-1}$
  which does not vanish at~$p$.
\end{proof}

\subsection{Left-- and right normal vectors}\label{sec:left-right-normal}
Let $L=\tvector{f\\1}\H\subset \H^2$ be a smooth map into $\HP^1$. Then
$f\colon M \rightarrow \H$ is the image of $L$ under the affine chart $\sigma$
defined by $e_1,e_2$. As explained in Section~\ref{sec:ident-hp1-with}, the
map $L$ is a holomorphic curve if and only if there exists a smooth map
$R\colon M \rightarrow S^2\subset \Im\H$ such that
\[ *df = -df R.\] 
Analogously, the dual map
$L^\perp$ is a holomorphic curve if and only if there exists a smooth map
$N\colon M \rightarrow S^2\subset \Im\H$ such that
\[*df = N df,\] because $L^\perp= (e_1^*-e_2^* \bar f)\H \subset
(\H^2)^*$. The maps $N$ and $R$ are called \emph{left normal} and \emph{right
  normal} of $f$, respectively. If $f$ is a conformal immersion, both normal
vectors exist and are smooth. One can prove (cf.\ appendix to \cite{Boh}) that
if $f$ is a non--constant map for which one of the two normals exists and is
smooth, the other normal vector can be globally defined as a continuous map
which is smooth away from the branch points of $f$.

Recall from Section~\ref{sec:holom-curv-and-linear-systems} that if $L$ is a 
holomorphic curve then $L^{-1}=(\H^2)^*/L^\perp$ is a Möbius invariant
holomorphic line bundle. The constant sections of $(\H^2)^*$ project to the
canonical linear system of holomorphic sections which is spanned by
$e_1^*=e_2^* \bar f$ and $e_2^*$. In particular, the complex structure $J$ of
$L^{-1}$ satisfies $Je_2^* = e_2^* R$. Similarly, if $L^\perp$ is holomorphic,
then $\H^2/L=(L^\perp)^{-1}$ is a Möbius invariant holomorphic line bundle
 whose canonical
linear system is spanned $e_1$ and $e_2=-e_1 f$ and whose complex structure
$J$ satisfies $Je_1 = e_1 N$.

\subsection{Mean curvature sphere congruence}\label{sec:mean-curv-sphere-affine}
Assume now that $L$ is immersed. Then both M\"obius invariant quaternionic
holomorphic line bundles are defined. The left and right normals $N$ and $R$
are both smooth and $*df= Ndf = -df R$. In \cite[Section~7.1]{BFLPP02} 
it is shown that the mean curvature vector $\mathcal{H}$ of $f$ is
related to $N$ and $R$ via 
\[ dN'=\frac12(dN-N{*}dN)=df H\quad \textrm{ and } \quad
dR'=\frac12(dR-R{*}dR)=Hdf,\] where $H=-R\bar{\mathcal{H}}=-\bar{\mathcal{H}}N$.
In particular $f$ is Euclidean minimal if and only if $N$ and $R$ are
anti--holomorphic, i.e., $dN'=dR'=0$.

The mean curvature sphere congruence of $L$
(Section~\ref{sec:mean-curv-sphere}) is given by
\[ S=\Ad\dvector{1 & f \\ 0 & 1} \dvector{N & 0 \\ -H & -R}= \dvector{N-fH & f
  H f -N f -f R 
  \\ -H & H f -R} \] (see \cite[Section~7.2]{BFLPP02}) and its Hopf fields are
\[
2{*}A = \Ad \dvector{ 1 & f \\ 0 & 1} \dvector{0 & 0 \\ w & dR''}
\quad\textrm{ and } \quad 2{*}Q = \Ad \dvector{1 & f \\ 0 & 1} \dvector{dN'' &
  0 \\w-dH & 0},
\]
with $dR'' =\tfrac12(dR+R{*}dR)$, $dN''=\tfrac12(dN+N{*}dN)$, and
\[w=\tfrac12(dH+R{*}dH-H{*}dN'')=\tfrac12(dH+{*}dHN-{*}dR''H).\] The mean
curvature sphere congruence of the dual curve $L^\perp$ is $S^\perp= S^*$ with
Hopf fields $A^\perp=-Q^*$ and $Q^\perp=-A^*$.

By Lemma~\ref{lem:twistor_hol_A=0}, the holomorphic curve $L$ is twistor
holomorphic if and only if $dR''=0$.  Similarly $L^\perp$ is twistor
holomorphic if and only if $dN''=0$. Moreover, $L$ is totally umbilic if and
only if $S$ is constant which is equivalent to both $L$ and $L^\perp$ being
twistor holomorphic., i.e., to $dN''=dR''=0$.

\section{B\"acklund transformations of Willmore surfaces} \label{sec:bt} 

In this appendix we collect some results from \cite[Section~9]{BFLPP02} about
1-- and 2--step B\"acklund transformations of Willmore surfaces in the
conformal 4--sphere $S^4=\HP^1$ and derive consequences that are central for
the present paper. As an example we show how the Weierstrass representation of
minimal surfaces fits into this context.

\subsection{2--step Bäcklund transformation}\label{sec:2-step-baecklund}
A holomorphic curve $L\subset \H^2$ in $\HP^1$ is Willmore if and only if its
Hopf fields $A$ and $Q$ are co--closed, see Section~\ref{sec:mean-curv-sphere}.
If $A$ does not vanish 
identically, then $\fwd L =\ker(A)$ extends smoothly 
through the isolated zeroes of $A$ to either a branched conformal
immersion or a constant map. If $Q$ does not
vanish identically, the same is true for $\bwd L=\im(Q)$, 
cf.~\cite[Appendix~13.1]{BFLPP02}.
The maps $\fwd L$ and $\bwd L$ defined by
\[
\fwd L \supset \ker(A) \qquad 
\text{and}\qquad 
\bwd L \subset \im(Q)
\]
are then called the \emph{2--step forward} or \emph{backward B\"acklund
  transform} of $L$, respectively. If $\fwd L$ is defined and non--constant,
by \cite[Theorem~7]{BFLPP02} its Hopf field is $\fwd Q=A$ which implies that
$\fwd L$ is again Willmore and $\bwd{\text{$\fwd L$}}=L$. Analogously, 
$\bwd A = Q$ and $\fwd{\text{$\bwd L$}}$.

\subsection{1--step Bäcklund transformation}\label{sec:1-step-baecklund}
In contrast to 2--step B\"acklund transformations, 1--step B\"acklund
transformations are not M\"obius invariant but depend on the choice of a point
$\infty\in \HP^1$: as in Appendix~\ref{app:affine_coordinates} we chose
$\infty = \tvectork{1\\0}$ and write $L=\tvector{f\\1}\H\subset \H^2$ with
$f\colon M\rightarrow \H$ the representation in the affine chart defined by
$e_1$, $e_2$.  Because $w$ in Appendix~\ref{sec:mean-curv-sphere-affine}
satisfies $*w=-Rw$ and $*(w-dH)=(w-dH)N$ we have
\[df\wedge w= (w-dH)\wedge df =0.\] 

One can check \cite[Section~7]{BFLPP02} that $w$ is closed if and only if $A$
or, equivalently, $Q$ is co--closed which in turn is equivalent to $f$ being
Willmore (see Section~\ref{sec:Willmore}).

If $f$ is a Willmore immersion, a smooth map $g\colon M\rightarrow \H$ that
satisfies
\begin{equation*}
	dg=w=e_2^*(2{*}Ae_1)
\end{equation*} 
is called a \emph{1--step forward B\"acklund transform} of $f$. Up to
similarity transformation $g$ is uniquely determined by the choice of
$\infty$. If $w$ does not vanish identically, then $g$ is a branched conformal
immersion.  Similarly, a map $h$ with $dh = w-dH=e_2^*(2{*}Qe_1)$ is called a
\emph{1--step backward B\"acklund transform}.

\subsection{Properties of the Bäcklund transformations}
The following theorem describes the relation between 1--step and 2--step
Bäcklund transformations.

\begin{TheA}\label{the:BT} 	
  Let $f\colon M\rightarrow \H$ be a Willmore immersion and $g\colon M\to\H$ a
  1--step forward Bäcklund transform, i.e., $dg=w$. If $w\not\equiv 0$, then:
  \begin{enumerate}[i)]
  \item\label{item:forward=backward} the 1--step forward Bäcklund 
    transform $g$ is a
    branched conformal Willmore immersion. Away from its branch points $df
    =w_g-d H_g$. In particular, $f$ is a 1--step backward Bäcklund transform 
    of $g$	 and $f+H_g$ is a 1--step forward Bäcklund transform of $g$.
  \item\label{item:2times1is2} the 2--step forward B\"acklund transform $\fwd
    f$ of $f$ coincides with the 1--step forward B\"acklund transform $\fwd f
    = f+H_g$ of $g$.
  \item\label{item:1-step-of-Euclidean-minimal} 
    the 1--step forward B\"acklund transform $g$ of $f$ is 
    twistor holomorphic, i.e., 
    $A_g\equiv0$, if and only if there exists a point $a \in \H$ such that the
    inversion $(f-a)^{-1}$ is Euclidean minimal.
  \end{enumerate}
  The analogous statements hold for backward B\"acklund transformations.
\end{TheA}

\begin{proof}
  Statement~\ref{item:forward=backward}) is an immediate consequence of
  the fact \cite[Proposition~16]{BFLPP02} that if two branched conformal
  immersions satisfy $df\wedge dg=0$, then $dg(w_g-dH_g)=w df$.
  Statement~\ref{item:2times1is2}) follows from
  \cite[Lemma~10]{BFLPP02}.  To see
  \ref{item:1-step-of-Euclidean-minimal}), note that since $A\not
  \equiv 0$ the characterization of Euclidean minimal surfaces given in
  Section~\ref{sec:eucl-minim-curv} yields that $(f-a)^{-1}$ is Euclidean
  minimal if and only if $\tvectork{a\\1}\subset \ker A$. But this means that
  $\tvectork{a\\1}$ is the 2--step forward Bäcklund transform of~$f$ which,
  by~\ref{item:forward=backward}) and~\ref{item:2times1is2}), is
  equivalent to $w_g\equiv0$. Because $w_g\equiv0$, the assumption
  $A_g\not\equiv0$ would imply that $g$ is Euclidean minimal, i.e.,
  $H_g\equiv0$, which contradicts $a=\fwd f = f+H_g$. Thus, for the 1--step
  B\"acklund transformation $g$ we have that $w_g\equiv0$ is equivalent to
  $A_g\equiv0$.
\end{proof}

Using the formula for $S$ in Section~\ref{sec:mean-curv-sphere-affine} we
obtain:

\begin{CorA}\label{Cor:minimal_vs_twistor}
  Let $f\colon M\rightarrow \H$ be a nowhere vanishing conformal immersion of
  a simply connected Riemann surface.  Its inversion $f^{-1}$ is Euclidean
  minimal if and only if there is a branched twistor holomorphic immersion $g$
  with mean curvature sphere congruence $S_g$ such that
  \begin{equation*}
  	f=-H_g =e_2^*(S_g e_1)
  \end{equation*}	
  away from the branch points of $g$.
\end{CorA}

In other words, $f^{-1}$ is Euclidean minimal if and only if $\bar
f=-N_g\mathcal{H}_g=-\mathcal{H}_gR_g$, i.e., $\bar f$ is the rotation of the
mean curvature vector $\mathcal{H}_g$ of a twistor holomorphic curve~$g$ by
minus $\pi/2$ in the normal bundle of $g$. Note that locally every branched
twistor holomorphic immersion $g$ into $\H$ that is neither totally umbilic
nor Euclidean minimal is a 1--step forward B\"acklund transformation of $f$ as
in the preceding corollary.

By Richter's theorem \cite[Theorem 9]{BFLPP02}, a Willmore surface $f\colon M
\rightarrow \H$ that is not Euclidean minimal takes values in $\R^3=\Im\H$
if and only if it has a 1--step forward B\"acklund transform $g$ that is
minimal with respect to the hyperbolic geometry defined by the Hermitian form
\begin{equation*}
	\spr{\tvector{x_1\\x_2},\tvector{y_1\\y_2}}=\bar x_2 y_1 + \bar x_1 y_2.
\end{equation*}
(By Lemma~\ref{lem:minimal_s_skew} the latter is equivalent to $S_g$ being
skew with respect to $\thespr$.  The hyperbolic minimal B\"acklund transforms
are then characterized by the property that $\Re(g)=\frac12 H$.)  We obtain:

\begin{CorA}\label{cor:min_r3} 
  Let $f\colon M\rightarrow \R^3=\Im\H$ be a nowhere vanishing conformal
  immersion of a simply connected Riemann surface. Its inversion $f^{-1}$ is
  Euclidean minimal if and only if there is a branched twistor holomorphic
  immersion $g$ with mean curvature sphere congruence $S_g$ 
  that is hyperbolic minimal with respect to $\thespr$ and a null 
  vector $e_1\in\H^2$ such that
  \begin{equation*}
  	f=-H_g =\langle e_1,S_g e_1\rangle
  \end{equation*}	
  away from the branch points of $g$.
\end{CorA}

In other words, $f^{-1}$ is Euclidean minimal if and only if
$f=N_g\mathcal{H}_g=\mathcal{H}_gR_g$, i.e., $f$ is the rotation of the mean
curvature vector $\mathcal{H}_g$ of a twistor holomorphic curve~$g$ by $\pi/2$
in the normal bundle of $g$.

\subsection{Weierstrass representation and 1--step Bäcklund transformation of
  twistor holomorphic curves}\label{sec:weierstrass-representation-and-1-step-bt}
The formula $f=-H_g =e_2^*(Se_1)$ in Corollary~\ref{Cor:minimal_vs_twistor}
can be seen as an integral free version of the Weierstrass representation of
the minimal surface $f^{-1}$: we show that $f^{-1}=(e_2^*(Se_1))^{-1}$ is the
imaginary part of a holomorphic null curve in $\C^4$ obtained from $g$.

Right multiplication by $\ii$ makes $\H^2$ into a complex 4--dimensional
vector space $\C^4\cong(\H^2,\ii)$.  Let $\tvectork{g\\1} \subset\H^2$ be a
twistor holomorphic curve as in Corollary~\ref{Cor:minimal_vs_twistor} (i.e.,
$g$ is neither Euclidean minimal nor totally umbilic).  Away from the branch
points of $\tvectork{g\\1}$, the $\ii$--eigenspaces of its mean curvature
sphere congruence $S$ coincide with the tangent line congruence (or first
osculating curve) of its twistor lift, cf.\
Appendix~\ref{sec:twist-holom-curv}.  Because by assumption $H_g$ is an
immersion, the vector $e_1$ is only at isolated points an eigenvector of
$S$. The tangent line congruence of the twistor holomorphic curve
$\tvectork{g\\1}$ is, away from these isolated points, the holomorphic null
immersion
\begin{align*}
  [\hat S] \colon M\to Q^4,\qquad \hat S = (e_1-Se_1\ii)\land
  (e_1+Se_1\ii)\jj{}
\end{align*}
in $Q^4=\theset{[v]\in P(\Lambda^2(\H^2,\ii))}{v\land v=0}$, where $e_1$,
$e_2$ is the standard basis of $\H^2$ and where \emph{null} means that
the tangent lines of $[\hat S]$ are contained in $Q^4$. Consider the real
structure on $\Lambda^2(\H^2,\ii)$ defined by $x\land y\mapsto x\jj\land
y\jj$.  A real basis of $\Lambda^2(\H^2,\ii)$ is then given by
\begin{align*}
  \hat e_1&=e_1\land e_2\jj-e_1 \jj\land e_2,&\hat e_2&=(e_1\land e_2\jj+e_1
  \jj\land e_2)\ii,\\
  \hat e_3&=e_1\land e_2+e_1 \jj\land e_2\jj,&\hat e_4&=(e_1\land e_2-e_1
  \jj\land e_2\jj)\ii,\\
  \hat e_5&=e_1\land e_1\jj,&\hat e_6&=e_2\land e_2\jj.&
\end{align*}
A complex bilinear symmetric form $\thespr$ on $\Lambda^2(\H^2,\ii)$ can be
defined by $\spr{x,y}\hat e_5\land \hat e_6=x\land y$.  The affine part of the
stereographic projection with ``pole'' $[e_5]$ onto the tangent plane to $Q^4$
at $[e_6]$ is then
\begin{equation*}
  \hat\sigma (\hat S)
  = (\hat S - \spr{\hat S, \hat e_6}\hat e_5)\spr{\hat S,\hat
    e_5}^{-1}-\hat e_6,
\end{equation*}
which is a meromorphic null curve in $\C^4\cong \Span\{\hat e_1, \hat e_2,\hat
e_3, \hat e_4\}$. Writing $Se_1=e_1 h_1+e_2 h_2$ we obtain
\begin{align*}
  \spr{\hat S,\hat e_5}=-|h_2|^2\text{\qquad and\qquad} \Im(\hat S - \spr{\hat
    S, \hat e_6}\hat e_5-\spr{\hat S, \hat e_5}\hat e_6)=-\bar h_2,
\end{align*}
where we identify $\Im(\C^4)\cong\H$ via $\hat e_1\ii\mapsto 1$, $\hat
e_2\ii\mapsto \ii$, $\hat e_3\ii\mapsto \jj$, $\hat e_4\ii\mapsto \kk$.  Hence
\begin{equation*}
  \Im(\hat\sigma (\hat S))=\bar h_2|h_2|^{-2}=
  h_2^{-1}=(e_2^*(Se_1))^{-1}=f^{-1}. 
\end{equation*}

The minimal surface $f^{-1}$ and hence $f$ are contained in $\Im\H=\R^3$ if
and only if $\hat S$ is orthogonal to $\hat e_1$ in $\Lambda^2(\H^2,\ii)$ with
respect to $\thespr$.  But this can be shown to be equivalent to $S$ being
skew Hermitian with respect to the Hermitian form $\thespr$ in
Corollary~\ref{cor:min_r3} which, by Lemma~\ref{lem:minimal_s_skew}, is again
equivalent to $g$ being hyperbolic superminimal.

\section{Twistor holomorphic curves in $\HP^1$}\label{sec:twist-holom-curv}
  
This appendix relates properties of the mean curvature sphere congruence $S$
of a twistor holomorphic curve to properties of the osculating curves of its
twistor lift.

A twistor holomorphic curve in $\HP^1$ is a holomorphic curve whose Hopf field
$A$ vanishes identically (Lemma~\ref{lem:twistor_hol_A=0}) and which hence is
in particular Willmore (cf.\
Section~\ref{sec:mean-curv-sphere}). Equivalently, a holomorphic curve in
$\HP^1$ is twistor holomorphic if and only if the derivative $\nabla S$ of its
mean curvature sphere congruence $S$ satisfies
  \begin{equation*}
  	\nabla S=2{*}Q.
  \end{equation*}

\subsection{Mean curvature sphere congruence}
The \emph{$k$--th osculating curve $\hat L_k$} of a holomorphic curve $\hat
L\subset \C^{n+1}$ is the analytic continuation of the rank $k+1$ subbundle
spanned by the $k$--th derivatives of sections of $\hat L$, cf.\
\cite[p.~262]{GriHa}.  The first osculating curve or tangent line congruence
$\hat L_1$ of the twistor lift $\hat L=\theset{\psi\in
  L}{J\psi=\psi\ii}\subset (\H^2,{\ii})$ of a twistor holomorphic curve $L$ in
$\HP^1$ is the ${\ii}$--eigenspace of the mean curvature sphere congruence $S$
of $L$, because $S\nabla\varphi=(\nabla S)\varphi \ii+
\nabla\varphi\ii=\nabla\varphi\ii$ for all $\varphi\in \Gamma(\hat
L)$. Although $S$ is not defined at branch points of $L$, the tangent line
congruence $\hat L_1$ is, because all osculating curves of \emph{complex}
holomorphic curves extend through their Weierstrass points. We obtain the
following characterization of the branch points of $L$ through which the mean
curvature sphere congruence extends smoothly.

\begin{LemA}\label{L:smooth_S_equiv_E1_not_quaternionic}
  The mean curvature sphere congruence of a twistor holomorphic curve $L$
  extends smoothly through a branch point $p$ of $L$ if and only if at $p$ the
  tangent line congruence $\hat L_1$ of its twistor lift is not tangent to the
  fiber of the twistor projection, i.e., if\/ $\H^2=\hat {L_1}_{|_p}\oplus
  \hat L_1{\jj}_{|_p}$.
\end{LemA}

\subsection{2--step Bäcklund transformation}
The complex dual curve of the twistor lift of a twistor holomorphic curve $L$
in $\HP^1$ can again be projected to $\HP^1$. We prove now that the curve thus
obtained is the dual curve of the $2$--step backward Bäcklund transform $\bwd
L$ of $L$ (Appendix~\ref{sec:2-step-baecklund}).

\begin{LemA}\label{L:twistor_projection_of_the_dual}
  Let $L\subset \H^2$ be a non totally umbilic twistor holomorphic curve in
  $\HP^1$ with twistor lift $\hat L$.  Its $2$--step backward Bäcklund
  transform $\bwd L=\im(Q)$ extends smoothly through zeros of $Q$ and branch
  points of $L$ (where the mean curvature sphere $S$ of $L$ and hence $Q$ is
  not defined) and the dual $\bwd L^\perp\subset (\H^2)^*$ is a twistor
  holomorphic curve with mean curvature sphere congruence $-S^*$ whose twistor
  lift is the complex dual curve $\hat L^d$ of $\hat L$.
\end{LemA}  

\begin{proof}
  Since $L$ is twistor holomorphic we have $\nabla S=2{*}Q$, and $\hat L_1\oplus \hat
  L_1{\jj}=\H$ away from the isolated branch points of $L$, by
  Lemma~\ref{L:smooth_S_equiv_E1_not_quaternionic}.  It suffices to
  show that $\image(Q)\subset \hat L_2$, because $\hat L^d=\hat
  L_2^\perp\subset (\H^2,{\ii})^*=(\C^4)^*$. Let $\varphi\in\Gamma(\hat L_1)$
  and $\omega_1,\omega_2\in\Omega^1(\hat L_1)$ such that $\nabla
  \varphi=\omega_1+\omega_2{\jj}$. Then
   \begin{equation*}
    \begin{aligned}
      2{*}Q\varphi &= \nabla S\varphi=\nabla(S\varphi)-S\nabla\varphi
      =\nabla \varphi{\ii} -S\nabla\varphi\\
      &=\omega_1{\ii}+\omega_2{\jj}{\ii}
      -\omega_1{\ii}-\omega_2{\ii}{\jj}=2\omega_2\jj\ii.
    \end{aligned}
  \end{equation*}
  Thus $\image(Q)\subset \hat L_2$, because
  $Q\varphi=*\omega_2\ii\jj=*\omega_1\ii-*\nabla\varphi\ii\in\Omega^1(\hat
  L_2)$ and $Q\varphi\jj=-*\omega_2\ii \in\Omega^1(\hat L_2)$ and $Q\varphi\jj
  \in\Omega^1(\hat L_1)$.
\end{proof}

\subsection{Euclidean minimal twistor holomorphic curves}
Using the characterization of Euclidean minimal curves in
Section~\ref{sec:eucl-minim-curv} we obtain the following characterization of
holomorphic curves in $\HP^1$ that are twistor holomorphic and Euclidean
minimal.

\begin{LemA}\label{lem:eucl_supermin} 
  Let $L\subset \H^2$ be a twistor holomorphic curve in $\HP^1$ that is not
  totally umbilic. Then $L$ is Euclidean minimal in $\H = \HP^1\backslash
  \{\infty\}$ for some $\infty \in \HP^1$ if and only if $\bwd L=\im(Q)$ is
  constant if and only if the twistor lift of $\hat L\subset (\H^2,\ii)$ lies
  in a complex plane. The point $\infty\in\HP^1$ then corresponds to the
  unique $\jj$--invariant line in the plane containing $\hat L$.
\end{LemA}

\begin{proof}
  If $\bwd L$ is constant, than $\infty=\bwd L$ is a point contained in all
  mean curvature spheres (because $S$ anti--commutes with $Q$) and $L$ is
  Euclidean minimal in $\HP^1\backslash \{\infty\}$. The point $\infty$ is
  contained in all mean curvature spheres if and only if the corresponding
  complex line in $P(\H^2,\ii)\cong\CP^3$ intersects all tangent lines $\hat
  L_1$ of $\hat L$.  But this implies that $\hat L$ is planar (a space curve
  admitting a line that intersects all its tangents is always planar).  If
  $\hat L$ is planar, then $\hat L^d$ is constant and
  Lemma~\ref{L:twistor_projection_of_the_dual} implies that $\bwd L$ is
  constant.
\end{proof}

\subsection{Branching divisors of the osculating curves of the twistor lift}
Let $L$ be a twistor holomorphic curve with twistor lift of $\hat L$.  The
following lemma relates the branching divisors of the osculating curves of
$\hat L$ to the zero divisor $\ord(Q)$ of the Hopf field $Q$ of $L$ and the
branching divisors $b(L)$ and $b(\bwd L)$ of $L$ and its backward Bäcklund
transform $\bwd L$ (extended to the whole Riemann surface as in
Lemma~\ref{L:twistor_projection_of_the_dual}).

\begin{LemA}\label{L:branching_of_twistorlift}
  Let $L\subset\H^2$ be a twistor holomorphic curve in $\HP^1$ with
  holomorphic twistor lift $\hat L\subset (\H^2,{\ii})$ for which globally
  $\H^2=\hat L_1\oplus \hat L_1{\jj}$.  Then
  \begin{equation*}
    b(\hat L)=b(L),\qquad b(\hat L_1)=\ord (Q),\qquad b(\hat L_2)=b(\bwd L).
  \end{equation*} 
  The last equality holds if $\hat L_2$ and hence $\bwd L$ is non--constant.
\end{LemA}

\begin{proof}
  Clearly $b_p(\hat L)\leq b_p(L)$ for all $p\in M$ with equality if
  ${\hat{L}}_{1|p}$ is not a quaternionic subspace. Thus $\hat L_1\cap \hat
  L_1{\jj}=\{0\}$ implies $b(\hat L)=b(L)$. The same argument applied to
  $(\bwd L)^\perp$ shows $b(\hat L_2)=b(\bwd L)$, because the complex dual
  $\hat L^d$ of $\hat L$ is the twistor lift of $(\bwd L)^\perp$, by
  Lemma~\ref{L:twistor_projection_of_the_dual}, and $b(\hat L_2)=b(\hat
  L^d)$.  The displayed formula in the proof of
  Lemma~\ref{L:twistor_projection_of_the_dual} and $\H^2=\hat L_1\oplus \hat
  L_1{\jj}$ yield $ b_p(\hat L_1)=
  \min\theset{\ord_p(\nabla\varphi{\ii}-S\nabla\varphi)}{\varphi\in\Gamma(\hat
    L_1)} =\min\theset{\ord_p(Q\varphi)}{\varphi\in\Gamma(\hat
    L_1)}=\ord_p(Q)$.
\end{proof}

\subsection{Superminimal curves in
  $\HP^1$}\label{subsec:superminimal} A superminimal curve is a twistor
holomorphic curve that is minimal with respect to some 4--dimensional space
form subgeometry \cite{Br82,Fr84,Fr88,Fr97}.

As in~\cite{BFLPP02,J03} we use quaternionic Hermitian forms in order to
describe the space form subgeometries of 4--dimensional M\"obius geometry in
the quaternionic projective framework: let $\thespr$ be a non--trivial
quaternionic Hermitian form on $\H^2$ and denote
$I=\theset{[v]\in\HP^1}{\spr{v,v}=0}$ its set of null lines. Depending on
whether $\thespr$ is definite, indefinite, or degenerate, the set $I$ is
empty, a round 3--sphere in $\HP^1$, or a point.

If $\thespr$ is non--degenerate one can define the Riemannian metric
\begin{equation*}
  g_{[x]}(v,w)
  :=\frac4{\spr{x,x}^2}\Re\big(\spr{v(x),w(x)}\spr{x,x}
  -\spr{v(x),x}\spr{x,w(x)}\big)
\end{equation*}
on $\HP^1\setminus I$, where $x\in\H^2$ and $v,w\in
T_{[x]}\HP^1=\Hom([x],\H^2/[x])$. The Riemannian metric is compatible with the
conformal structure on $\HP^1$ and the Riemannian manifold $(\HP^1/I,\pm g)$
is isometric to either $S^4$ or two copies of hyperbolic 4--space depending on
the signature of $\thespr$.

If $\thespr$ is degenerate, the affine chart
$\HP^1\backslash\{\infty\}\rightarrow \H$ of
Section~\ref{sec:ident-hp1-with} induces, uniquely up to scale, a Euclidean
structure on $\HP^1\setminus I\cong S^4\setminus \{\infty \}$.

\begin{LemA}\label{lem:minimal_s_skew}
  A holomorphic curve in $\HP^1$ is minimal with respect to the space form
  geometry defined by a Hermitian form $\thespr$ if and only if its mean
  curvature sphere congruence is skew Hermitian with respect to $\thespr$.
\end{LemA}
\begin{proof}
  A surface in a space form is minimal if and only if all its mean curvature
  spheres are totally geodesic.  The 2--sphere described by the eigenlines of
  an endomorphism $S$ of $\H^2$ that squares to $-1$ is totally geodesic if
  and only if the corresponding M\"obius transformation is an isometry of the
  space form defined by $\thespr$ which is equivalent to $S$ being skew
  Hermitian.
\end{proof}

We call a twistor holomorphic curve in $\HP^1$ all of whose mean curvature
spheres are skew with respect to some Hermitian form on $\H^2$
\emph{spherical}, \emph{hyperbolic}, or \emph{Euclidean superminimal}
depending on the type of the Hermitian form. Note that compact superminimal
curves in $\HP^1$ exist in all three cases, although in the hyperbolic and
Euclidean case they go through infinity~$I$.

The $(\jj\C)$--part $\jj\Omega=\frac12(\thespr +\ii\thespr \ii)$ of the
Hermitian form $\thespr$ defines an alternating complex 2--form $\Omega$ on
$(\H^2,\ii)\cong \C^4$.  Push forward with the multiplication by the
quaternion $\jj$ induces real structures on the complex vector spaces
$\Lambda^2(\H^2,\ii)\cong \C^6$ and $\Lambda^4(\H^2,\ii)\cong \C$.  One can
check that, fixing an element in the real part of $\Lambda^4(\H^2,\ii)$, the
wedge product defines a Minkowski product on
$\Re(\Lambda^2(\H^2,\ii))\cong\R^6$.

Because $\Omega$ is a real linear form on $\Lambda^2(\H^2,\ii)$ it can be
realized as the scalar product with a real element $\Omega^\sharp \in
\Re(\Lambda^2(\H^2,\ii))$. One can check that $\Omega^\sharp$
is time like, space like, or light like depending on whether the Hermitian
form is definite, indefinite, or degenerate. Moreover, it can be proven that a
twistor holomorphic curve $L$ in $\HP^1$ is superminimal with respect to the
Hermitian form $\thespr$ on $\H^2$ if and only if the tangent line congruence
$\hat L_1$ of its twistor lift $\hat L$ is polar to $\Omega^\sharp$. 

\begin{ProA}\label{P:superminimal_polar_to_time_space_light_like_vector}
  A holomorphic curve $L\subset \H^2$ is spherical, hyperbolic, or Euclidean
  superminimal if and only if its twistor lift $\hat L\subset (\H^2,\ii)$ is
  holomorphic and its tangent line congruence $\hat L_1\subset
  \Lambda^2(\H^2,\ii)$ is polar to a time like, space like, or light like
  vector in $\Re(\Lambda^2(\H^2,\ii))$.
\end{ProA}

\begin{RemA}\label{rem:E_isom_to_Ed}
  The 2--form $\Omega$ corresponding to $\Omega^\sharp \in
  \Re(\Lambda^2(\H^2,\ii))$ with $\Omega^\sharp\wedge\Omega^\sharp\neq0$
  induces an isomorphism between the complex vector space $(\H^2,\ii)$ and its
  dual. The tangent line congruence $\hat L_1$ of $\hat L$ is polar to
  $\Omega^\sharp$ if and only if this isomorphism maps $\hat L$ to its complex
  dual $\hat L^d$: take a local holomorphic section $\psi$ of $\hat L$. Then
  $\hat L_1$ is polar to $\Omega^\sharp$ if and only if
  $\Omega(\psi,\psi')=0$. But this implies $\Omega(\psi,\psi'')=0$ and is
  hence equivalent to $\ker(\Omega(\psi,.))=\operatorname{span}_\C\{\psi,
  \psi',\psi''\}$, i.e., to $\hat L$ being self--dual in the sense that $\hat
  L\cong \hat L^d$ with respect to the isomorphism induced by $\Omega$.
  This shows that the twistor lift of a spherical or hyperbolic superminimal
  curve in $\HP^1$ is self dual. 
\end{RemA}


\begin{thebibliography}{99}

\bibitem{Ab87} Uwe Abresch. Constant mean curvature tori in terms of elliptic
  functions.  \emph{J.\ Reine Angew.\ Math.}  {\bf 374}  (1987), 169--192. 

\bibitem{Bo91} Alexander I.~Bobenko. Surfaces of constant mean curvature and
  integrable equations. (Russian) Uspekhi Mat. Nauk 46 (1991), no. 4 (280),
  3--42, 192; translation in \emph{Russian Math. Surveys} {\bf 46} (1991),
  no. 4, 1--45.

\bibitem{Boh} Christoph Bohle, Constrained Willmore tori in the $4$--sphere.
   \href{http://arxiv.org/abs/0803.0633}{arXiv:0803.0633  [math.DG]}.
 
\bibitem{BLPP} Christoph Bohle, Katrin Leschke, Franz Pedit, and Ulrich
  Pinkall, Conformal maps from a 2--torus to the 4--sphere.
  \href{http://arxiv.org/abs/0712.2311}{arXiv:0712.2311v1 [math.DG]}.

\bibitem{SmoothEnds} Christoph Bohle and G.~Paul Peters. \emph{Bryant surfaces
    with smooth ends}. To appear in Comm.\ Anal.\ Geom.
  \href{http://de.arxiv.org/abs/math/0411480}{arXiv:math/0411480 [math.DG]}.

\bibitem{Br82} Robert L.~Bryant. Conformal and minimal immersions of compact
  surfaces into the $4$--sphere.  \emph{J.\ Differential Geom.} {\bf 17}
  (1982), no.~3, 455--473.

\bibitem{Br84} Robert L.~Bryant. A duality theorem for Willmore
  surfaces. \emph{J.\ Differential Geom.} {\bf 20} (1984), no.~1,
  23--53.
  
\bibitem{Br87} Robert L.~Bryant. Surfaces of mean curvature one in hyperbolic
  space. Th{\'e}orie des vari{\'e}t{\'e}s minimales et applications (Palaiseau,
  1983--1984). \emph{Ast{\'e}risque} No.~{\bf 154--155} (1987), 12, 321--347,
  353 (1988).
  
\bibitem{Br88} Robert L. Bryant. Surfaces in conformal geometry.
  \emph{The mathematical heritage of Hermann Weyl } (Durham, NC, 1987),
  227--240, Proc. Sympos. Pure Math., {\bf 48}, Amer.\ Math.\ Soc.,
  Providence, RI, 1988. 

\bibitem{BFLPP02} Francis E.\ Burstall, Dirk Ferus, Katrin Leschke, Franz Pedit,
  and Ulrich Pinkall. \emph{Conformal Geometry of Surfaces in $S^4$ and
    Quaternions}, Lecture Notes in Mathematics 1772, Springer, Berlin, 2002.

\bibitem{Co84}  Celso J. ~Costa. Example of a complete minimal immersion in
  $\mathbb R\sp 3$ of genus one and three embedded ends.
  \emph{Bol.\ Soc.\ Brasil.\ Mat.}  {\bf 15}  (1984),  no.~1--2, 47--54. 

\bibitem{Ejiri} Norio Ejiri. Willmore surfaces with a duality in $S\sp N(1)$.
  \emph{Proc.\ London Math.\ Soc.} (3) {\bf 57} (1988), no.~2, 383--416.


\bibitem{FLPP01} Dirk Ferus, Katrin Leschke, Franz Pedit, and Ulrich Pinkall.
  Quaternionic holomorphic geometry: Plücker formula, Dirac eigenvalue
  estimates and energy estimates of harmonic 2--tori.  \emph{Invent.\ Math.}
  \textbf{146} (2001), 507--593.
  
\bibitem{Fr84} Thomas Friedrich. On surfaces in four--space.
  \emph{Ann.\ Global Anal.\ Geom.} {\bf 2} (1984), no.~3, 257--287.
  
\bibitem{Fr88} Thomas Friedrich. The geometry of $t$--holomorphic
  surfaces in $S^4$. \emph{Math.\ Nachr.} {\bf 137} (1988),  49--62.
  
\bibitem{Fr97} Thomas Friedrich. On superminimal surfaces.
  \emph{Arch. Math. (Brno)} {\bf 33} (1997), no.~1--2, 41--56.

\bibitem{GriHa} Philipp Griffith and Joseph Harris. \emph{Principles of
    algebraic geometry}. Pure and applied mathematics, John Wiley \& Sons, New
  York, 1978, Wiley Classics Library Edition, 1994.

\bibitem{GOR} Robert D.\ Gulliver II, Robert Osserman, and Halsey L.\
  Royden. A theory of branched immersions of surfaces.  \emph{Amer.\ J.\
    Math.} {\bf 95} (1973), 750--812.
    
\bibitem{JMN01} Udo Hertrich-Jeromin, Emilio Musso, and Lorenzo
  Nicolodi, M{\"o}bius geometry of surfaces of constant mean curvature 1 in
  hyperbolic space. \emph{Ann.\ Global Anal.\ Geom.} {\bf 19}
  (2001), no.~2, 185--205.

\bibitem{J03} Udo Hertrich--Jeromin. \emph{Introduction to M{\"o}bius
    Differential Geometry}. Lecture Notes Series 300, Cambridge University
  Press, 2003.

\bibitem{Hi90} Nigel J. Hitchin. Harmonic maps from a $2$--torus to the
  $3$--sphere.  \emph{J.\ Differential Geom.}  31 (1990), no.~3, 627--710.

\bibitem{HoMe85} David A.~Hoffman and William H.~Meeks III. A complete embedded
  minimal surface in $\mathbb R\sp 3$ with genus one and three ends.
  \emph{J. Differential Geom.} {\bf 21}  (1985),  no.~1, 109--127.


\bibitem{Ko96} Boris Konopelchenko. Induced surfaces and their integrable
  dynamics. \emph{Stud.\ Appl.\ Math.} {\bf 96} (1996), no.~1, 9--51.

\bibitem{cime} William H. Meeks, Antonio Ros, and Herold Rosenberg.
  \emph{The global theory of minimal surfaces in flat spaces.}
  Lectures given at the 2nd C.I.M.E. Session held in Martina Franca, July
  7--14, 1999. Edited by Gian Pietro Pirola. Lecture Notes in Mathematics,
  1775. Springer--Verlag, Berlin, 2002.


\bibitem{Montiel} Sebasti{\'a}n Montiel. Willmore two-spheres in the
  four--sphere.  \emph{Trans.\ Amer.\ Math.\ Soc.}  {\bf 352} (2000), no.~10,
  4469--4486.

\bibitem{Musso} Emilio Musso. Willmore surfaces in the four--sphere.
  \emph{Ann.\ Global Anal.\ Geom.} {\bf 8} (1990), no.~1, 21--41.

  
\bibitem{Os86}  Robert A. Osserman. \emph{A survey of minimal
    surfaces}. Second edition. Dover Publications, Inc., New York, 1986 

\bibitem{PP98} Franz Pedit and Ulrich Pinkall. Quaternionic analysis on
  Riemann surfaces and differential geometry. \emph{Doc.\ Math.\ J.\ DMV},
  Extra Volume ICM Berlin 1998, Vol.~II, 389--400.

\bibitem{P04} G.~Paul Peters. \emph{Soliton Spheres}. Thesis, TU-Berlin,
  \href{http://nbn-resolving.de/urn:nbn:de:kobv:83-opus-8422}{urn:nbn:de:kobv:83-opus-8422},
  2004. 

\bibitem{PiSt89} Ulrich Pinkall and Ivan Sterling.  On the classification of
  constant mean curvature tori.  \emph{Ann.\ of Math.\ (2)} {\bf 130} (1989),
  no.~2, 407--451.

\bibitem{Ri} J{\"o}rg Richter. \emph{Conformal maps of a Riemann surface
    into the space of quaternions}. Thesis, TU--Berlin, 1997.

  
\bibitem{Ta97} Iskander A. Taimanov. Modified Novikov--Veselov equation and
  differential geometry of surfaces. \emph{Solitons, geometry, and topology:
    on the crossroad}, 133--151, Amer.\ Math.\ Soc.\ Transl.\ Ser.~2, 179,
  Amer.\ Math.\ Soc., Providence, RI, 1997.

 \bibitem{Ta98} Iskander A. Taimanov.  The Weierstrass representation of closed
  surfaces in $R\sp 3$. (Russian) Funktsional. Anal. i Prilozhen.  32 (1998),
  no.~4, 49--62, 96; translation in \emph{Funct.\ Anal.\ Appl.}  {\bf 32}
  (1998), no.~4, 258--267.

\bibitem{Ta99} Iskander A. Taimanov. The Weierstrass representation of
    spheres in $\R^3$, the Willmore numbers, and soliton spheres.
  Tr.\ Mat.\ Inst.\ Steklova, 225 (1999), Solitony Geom.\ Topol.\ na
  Perekrest., 339--361; translation in \emph{Proc.\ Steklov Inst.\ Math.}
  (1999), no.~2 {\bf 225}, 322--343.


\bibitem{Ta06} Iskander A. Taimanov. The two--dimensional Dirac operator and
  the theory of surfaces. (Russian)  Uspekhi Mat.~Nauk  61  (2006),
  no.~1(367), 85--164;  translation in  \emph{Russian Math.\ Surveys}  {\bf 61}
  (2006), no.~1, 79--159.

\bibitem{We86} Henry C.~Wente. Counterexample to a conjecture of H.~Hopf.
  \emph{Pacific J.\ Math.} {\bf 121} (1986), no.~1, 193--243.

\bibitem{W} Thomas J.~Willmore, \emph{Riemannian geometry}. Oxford University
  Press, Oxford, New York, 1993.

\end{thebibliography}
\end{document}